\documentclass[final]{siamltex}
\usepackage{amssymb,amsmath,amsfonts, mathrsfs}
\usepackage{algorithm}
\usepackage{algpseudocode}
\usepackage{subfigure,float}
\usepackage{graphicx}
\usepackage{epsfig}
\usepackage{url}
\usepackage{overpic}
\usepackage{wrapfig}
\usepackage{fancyhdr}
\usepackage{subeqnarray}
\usepackage{psfrag}
\usepackage{indentfirst}
\usepackage{booktabs, rotating, multirow}
\usepackage[usenames,dvipsnames]{color}

\numberwithin{algorithm}{section}

\begin{document}

\title{Eigenvalue Problems of Discrete Curl Operators on Various Lattices for Simulating Three Dimensional Photonic Crystals}
\author{%
Tsung-Ming Huang\thanks{Department of Mathematics, National Taiwan Normal University, Taipei, 116, Taiwan.}
\and Tiexiang Li\thanks{School of Mathematics, Southeast University, Nanjing 211189, People's Republic of China}
\and Wei-De Li\thanks{Department of Mathematics, National Tsing Hua University, Hsinchu 300, Taiwan.}
\and Jia-Wei Lin\thanks{Department of Applied Mathematics, National Chiao Tung University, Hsinchu 300, Taiwan.}
\and Wen-Wei Lin\thanks{Department of Applied Mathematics, National Chiao Tung University, Hsinchu 300, Taiwan.}
\and Heng Tian\thanks{Department of Applied Mathematics, National Chiao Tung University, Hsinchu 300, Taiwan.}
}

\date{\today }
\date{\today }
\maketitle

\begin{abstract}
There are many numerical methods for simulate three-dimensional photonic crystals, after comparison, we choose Yee's scheme to be our discrete method. So far, this method can only be applied to simple cubic lattice and face-centered cubic lattice, but there are 14 Bravais lattices in three-dimensional space. In this paper we extended Yee's scheme to all of the 14 Bravais lattices. After discretization, we get a general eigenvalue problem, and we will analyze this general eigenvalue problem. The second important task of this paper is to find the eigen-decomposition of the discrete curl operator, after a series of complicated calculations, we find that all the lattices can be summed up into two kinds of decomposition. We are interested in finding the several few smallest real eigenvalues, but the large dimension of null space is $\tfrac{1}{3}$ of all, this seriously affected the convergence of calculation, we use a technique which is called nullspace-free method to avoid this trouble. But this technique transforms the sparse matrix in our problem to a dense matrix, fortunately, the eigenvectors we found before are related to discrete Fourier transformation. The efficiency of calculation has been significantly improved by using the fast Fourier transformation. Finally, we calculate the band structures of photonic crystals on various lattices, and implement high performance calculations on the GPU.
\end{abstract}


\section{Introduction}\label{section1}

\subsection{Photonic Crystal}

Optoelectronics is one of the hottest areas in today's industry, and the most important component in dealing with optical signals is photonic crystals. A photonic crystal is a periodic structure in two-dimensional (2D) or three-dimensional (3D) real space which is composed of different periodic dielectrics, this structure is to imitate the arrangement of atoms in a crystalline solid, but the scale is hundreds of times than the crystalline solid. This material was first proposed by E. Yablonovitch \cite{yabl1987} and S. John \cite{john1987} in 1987, with the aim of extending the characteristics of electrons in semiconductors to electromagnetic waves. Photonic crystals can be classified as 1D, 2D, or 3D, since the numerical simulations of 1D or 2D photonic crystals are simpler, we only concentrate on the calculations of 3D photonic crystals.

The basic optical properties of photonic crystals are band gaps. Physicists have long known that the electrons in the crystal are affected by the periodic potential of the lattice, which leads to the energy band structure and the energy band gaps. Similarly, the propagation of light waves in photonic crystals also produce photonic band gaps, thst is, for specific structure, light wave with wavelength in some range can not pass through the photonic crystal, almost all of the applications and research about photonic crystals are relative to this feature. However, the width of the band gap and the corresponding frequency is difficult to explore from the experiment and observation, so numerical simulation plays an important role in understanding and designing the structure of photonic crystals. 

The initial study on photonic crystals focuses on trying to make the band gap larger, or hope to find the full band gap \cite{glml2001}. In addition, the defect states of photonic crystals have also been a lot of research, as long as to destruct the periodicity at local position, one can create a lot of useful nano-components, such as such as photonic crystal laser \cite{plsy1999, pavs1999}, photonic crystal waveguide \cite{ldvs2000}, photonic crystal fiber \cite{cmkb1999}, and integrated optical circuit \cite{ntyc2000}. In the other hand, the anomalous refraction effect is also an important research topic, especially the realization of negative refraction \cite{ljjp2002}, so that the manufacture of invisible cloak is no longer just fantasy.

In recent years, the emergence of topological photonic crystals (or photonic topological insulator) \cite{rzpl2013} make the research about photon crystal once again reached the peak. Topological photonic crystal is a kind of photonic crystal that behaves as a conductor on the surface but as an insulator in its interior, which brought a ray of light to overcome the bottleneck of the Moore's law. Although several 2D topological photonic crystals have been realized \cite{lujs2014}, very few 3D topological photonic crystals has been identified, hence finding the 3D topological photonic crystal has become a popular task. Lu et al. \cite{lfjs2013} proposed a numerical scheme to determine whether a photonic crystal is a topological photonic crystal, in this scheme, band structure calculation is involved in all major steps, therefore an efficient numerical computation method is indispensable. 

\subsection{Governing Equations}

The studies of photonic crystals are usually to discuss the propagation and scattering of light wave in photonic crystals, because the light is a kind of electromagnetic waves, and the behaviors of electromagnetic waves can be described by the Maxwell's equations, so we first introduce the Maxwell's equations \cite{thid2004}:
\begin{subequations}\label{eq:maxwell_original}
\begin{align}
\nabla \cdot D({\bf x}, t) &= {\bf \rho}_{f},\\
\nabla \cdot B({\bf x}, t) &= 0,\\
\nabla \times E({\bf x}, t) &= -\frac{\partial}{\partial t}B({\bf x}, t),\\
\nabla \times H({\bf x}, t) &= {\bf J}_{f} + \frac{\partial}{\partial t}D({\bf x}, t),
\end{align}
\end{subequations}
where the four vector fields $E$, $H$, $D$, and $B$ represent the electric field, magnetic field, electric displacement field, and magnetic induction field, repectively, all these vector fields depend on space and time, the vector ${\bf J}_{f}$ is the free current density, ${\bf\rho}_{f}$ is the free charge density. These four equations are called the Causs's law, the Gauss's law for magnetism, the Maxwell-Faraday equation, and the Maxwell-Amp\`{e}re equation, respectively. In numerical simulations of photonic crystals, we usually assume that the source of electomagnetic wave is far-away, ${\bf J}_{f}$ and ${\bf\rho}_{f}$ approach zero and so can be ignored. At this time, the equations \eqref{eq:maxwell_original} are reduced to the simple form which are called source free Maxwell's equations:

\begin{subequations}\label{eq:maxwell_free}
\begin{align}
\nabla \cdot D &= 0,\\
\nabla \cdot B &= 0,\\
\nabla \times E &= -\frac{\partial B}{\partial t},\label{eq:faraday}\\
\nabla \times H &= \frac{\partial D}{\partial t}.\label{eq:ampere}
\end{align}
\end{subequations}

Under the influence of electric field and magnetic field, the materials may have the polarisation and magnetisation. To specify the phenomena, we have to describe the so-called constitutive relation. That is, the vector fields $E$, $H$, $D$, and $B$ satisfy the following relations:
\begin{align}\label{eq:constitutive_relation}
D = \varepsilon E, \; B = \mu H,
\end{align}
where $\varepsilon$ is the permittivity and $\mu$ is the permeability. The two parameters depend on the materials that the electromagnetic wave travels. In vacuum, the parameters $\varepsilon = \varepsilon_{0}$ and $\mu = \mu_{0}$ are two constants. In dielectric materials, the parameters $\varepsilon = \varepsilon({\bf x})$ and $\mu = \mu({\bf x})$ are two position-dependent function, and photonic crystal is a kind of dielectric materials. Since we do not deal with magnetic materials in this artical, we set the permeability $\mu$ to be the vacuum permeability $\mu_{0}$. The permittivity $\varepsilon$ is called relative dielectric constants.

Moreover, photonic crystal is a kind of material with the periodic structure, and the electomagnetic wave travels in the material with periodic structure will satisfy the time harmonic assumption \cite{hvjs1961} and Bloch's theorem \cite{bloch1929}. If a time dependence factor $e^{-\imath \omega t}$ is assumed where $\omega$ is the frequency of the electomagnetic wave and $\imath = \sqrt{-1}$, then the electric field and magnetic field can be written as:
\begin{align}\label{eq:time_harmonic}
E({\bf x}, t) = e^{-\imath \omega t}E({\bf x}), \; H({\bf x}, t) = e^{-\imath \omega t}H({\bf x}).
\end{align}
Under the time harmonic assumption, the Maxwell's equations \eqref{eq:maxwell_free} is transformed from the time domain into the frequence domain. Substituting \eqref{eq:constitutive_relation} and \eqref{eq:time_harmonic} in \eqref{eq:maxwell_free}, we obtain the govering equations
\begin{subequations}\label{eq:governing}
\begin{align}
&\nabla \times \nabla \times E = \lambda\varepsilon E,\label{eq:double_curl}\\
&\nabla \cdot(\varepsilon E) = 0,\label{eq:divergence_free}
\end{align}
\end{subequations}
where $\lambda = \mu_{0}\omega^{2}$ is the unknown eigenvalue \cite{jjwm2011}.

Using Bloch's theorem, the electric field and the magnetic field are decomposed as:
\begin{align}\label{eq:bloch}
E({\bf x}) = e^{\imath 2\pi{\bf k}\cdot {\bf x}}u_{\bf k}({\bf x}), \; H({\bf x}) = e^{\imath 2\pi{\bf k}\cdot {\bf x}}v_{\bf k}({\bf x}),
\end{align}
where $2\pi{\bf k}$ is the Bloch wave vector, and $u_{\bf k}({\bf x})$ and $v_{\bf k}({\bf x})$ are periodic functions that satisfy
\begin{align}\label{eq:periodic_function}
u_{\bf k}({\bf x} + {\bf a}_{\ell}) = u_{\bf k}({\bf x}), \; v_{\bf k}({\bf x} + {\bf a}_{\ell}) = v_{\bf k}({\bf x}), \; \ell = 1, 2, 3,
\end{align}
here, the three vectors $\{{\bf a}_{\ell}\}_{\ell = 1}^{3}$ are the lattice translation vectors of crystals which we will introduce more details in next chapter. The Bloch wave vector $2\pi{\bf k}$ is the direction of wave propagation, we just have to consider the wave vectors belong to the first Brillouin zone for some well-known reasons \cite{hsieh2016}. Combining \eqref{eq:bloch} and \eqref{eq:periodic_function}, we can observe that
\begin{align}
E({\bf x} + {\bf a}_{\ell}) = e^{\imath 2\pi{\bf k}\cdot({\bf x}+{\bf a}_{\ell})}u_{\bf k}({\bf x} + {\bf a}_{\ell}) = e^{\imath 2\pi{\bf k}\cdot{\bf x}}e^{\imath 2\pi{\bf k}\cdot{\bf a}_{\ell}}u_{\bf k}({\bf x}) = e^{\imath 2\pi{\bf k}\cdot{\bf a}_{\ell}}E({\bf x}),\label{eq:quasi_periodic}\\
H({\bf x} + {\bf a}_{\ell}) = e^{\imath 2\pi{\bf k}\cdot({\bf x}+{\bf a}_{\ell})}v_{\bf k}({\bf x} + {\bf a}_{\ell}) = e^{\imath 2\pi{\bf k}\cdot{\bf x}}e^{\imath 2\pi{\bf k}\cdot{\bf a}_{\ell}}v_{\bf k}({\bf x}) = e^{\imath 2\pi{\bf k}\cdot{\bf a}_{\ell}}H({\bf x}),\label{eq:quasi_periodic_mag}
\end{align}
for $\ell = 1, 2, 3$. The equations \eqref{eq:quasi_periodic} and \eqref{eq:quasi_periodic_mag} are called quasi-periodic conditions \cite{resi1978}, and the special cases while ${\bf k}\cdot{\bf a}_{\ell} \in \mathbb{Z}$ are called periodic conditions. Given any wave vector, we aim to find Bloch eigenfunctions $E$ in \eqref{eq:double_curl} that satisfy the quasi-periodic condition \eqref{eq:quasi_periodic}. For other different boundary conditions, one can find more details by refer to \cite{hsieh2016}.

\subsection{Numerical Methods}

For the numerical simulation of photonic crystals, the common methods include: plane-wave expansion method (PWE) \cite{hocs1990, jovf1997, jojo2001, sako2004}, multiple scattering method \cite{guky2004, twer1952}, transfer-matrix method \cite{fowl1989}, finite-difference time-domain method (FDTD) \cite{jjwm2011, kulu1993, taha2005, yee1966}, and finite-difference frequency-domain method (FDFD) \cite{gwar2004, hhlw2013eig, hhlw2013matrix, xzhw2003, yasu2005, yuch2004}. Each method can be applicable on different objects and has its own advantages and disadvantages. Multiple scattering method can be used to calculate the electromagnetic field of cylindrical photonic crystals. Transfer-matrix method can be used to simulate one-dimensional photonic crystals. PWE is currently the most widely used method, this method is suitable for calculating the band structure of photonic crystals with periodic boundary condition in any directions, it is very efficient if we only want to calculate a few eigenpairs. But PWE can not be used in other boundary conditions, and the size of matrix in calculation will grow fast as the number of desired eigenpairs increase. To calculate more eigenpairs, the matrix size will be very large. FDTD and FDFD are very similar, both of them can be used to calculate the electromagnetic field of photonic crystals in any structure, but in the past these two methods are very time consuming. Different from the PWE, the size of matrix in calculation is always retaining in FDTD and FDFD, and they can be easily applied on any boundary conditions, so these two methods are more suitable when we want to find more eigenpairs and the boundary condition is not periodic. FDFD method has an additional limit, it can only be used while the time harmonic assumption is satisfied.

We will focus on FDFD in this artical. Huang et al \cite{hhlw2013matrix} used FDFD to give the explicit matrix form of the discrete curl type operators in face-centered cubic structure, and transformed \eqref{eq:double_curl} into a generalized eigenvalue problem. The dimension of null space of this generalized eigenvalue problem accounts for $\tfrac{1}{3}$ of all,  and if the several smallest positive eigenvalues are the target that we want to find,  any of the currently known eigensolver will be affected by null space and difficult to get the answer. Huang et al\cite{hhlw2013eig} calculated the eigen-decomposition of the matrix, and then reduced the generalized eigenvalue problem to a null space free standard eigenvalue problem by some matrix computation techniques. The price is the original sparse matrix becomes dense matrix, so the matrix-vector multiplication is expensive. Because the matrix-vector multiplication can be superseded by a variant fast Fourier transformation (FFT) , the computation cost is reduced significantly, hence FDFD is now the best way to calculate the band structure of 3D photonic crystals. Unfortunately, apart from the simple cubic structure and the face-centered cubic structure, these accerlerating techniques can not be applied so far, and each structure has its own matrix representation, we will extend these methods to various structures.

\subsection{Notations  and Organization}

Throught this dissertation, the imaginary number $\sqrt{-1}$ is defined as $\imath$ and the identity matrix of order $n$ is defined as $I_{n}$. We denote $\top$ and $\ast$ the transpose and the conjugate transpose of a vector or a matrix by the superscript, respectively. For the matrix operations, we denote the Kronecker product of two matrices $\otimes$ and direct sum of two matrices $\oplus$. The operator $\mbox{diag}({\bf v})$ create a square diagonal matrix with the elements of any vector ${\bf v} \in \mathbb{C}^{n}$ on the main diagonal. The operator $\mbox{vec}(\cdot)$ vectorizes a matrix by stacking the columns of the matrix. The conjugate of a complex scalar $z \in \mathbb{C}$ is written as $\overline{z}$ and the conjugate of a complex vector ${\bf z} \in \mathbb{C}^{n}$ is written as $\overline{\bf z}$. Other notations will be defined clearly whenever they appear.

This dissertation is organized as follows. We first introduce some basic background of crystallography in section \ref{section2}. Next, we use the Yee's scheme to discretize the curl operators and give the explicit matrix representations of the curl operators for various lattice structures in section \ref{section3}. In section \ref{section4}, we derive eigen-decompositions of the partial derivative operators. Applying the decompositions we propose the singular value decompositions (SVD) for single curl operators and some techniques of matrix computation in section \ref{section5}. We compute the band structures of 3D photonic crystals for 14 Bravais lattices in section \ref{section6}. Finally, we conclude this article in section \ref{section7}.


\section{Background}\label{section2}

We only consider the perfect crystals in this paper, and there are still quasicrystals, non-crystalline solid, and crystals with defects which we have not studied yet. In order to study properties of various crystals, it is necessary to have some background knowledge of crystallography \cite{hahn2002}. In this section, we give introduction to some basics of crystallography that will be needed.

Crystals have periodicity and symmetry, the periodicity is determined by the quasi-periodic condition \eqref{eq:quasi_periodic}, and the symmetry is determined by the point groups and space groups. Any point group and space group is composed of several symmetry operations. According to the Crystallographic restriction theorem, there are only 32 point groups and 230 space groups in 3D real space \cite{hahn2002}.

\subsection{Bravais lattice}

A crystal structure can be regarded as a lattice structure plus a basis. At present, millions of crystals are known, and each crystal has a different nature. Fortunately, there are only 14 Bravais lattices in 3D Euclidean space, and they belong to 7 lattice systems. The following table \ref{table:bravais} is a classification of the 14 Bravais lattices. In table \ref{table:bravais}, each graph is a unit cell, the simplest repeating unit in a crystal. The physical dimension of a 3D unit cell can be characterized by 6 lattice constants (or lattice parameters) : $a$, $b$, $c$, $\alpha$, $\beta$, and $\gamma$, where $a$, $b$, and $c$ are the edge lengths of the unit cell and $\alpha$, $\beta$, and $\gamma$ are the three angles between the sides. We show all these lattice constants in figure \ref{fig:lattice constant}.

\begin{figure}[H]
\center
\includegraphics[height=4cm]{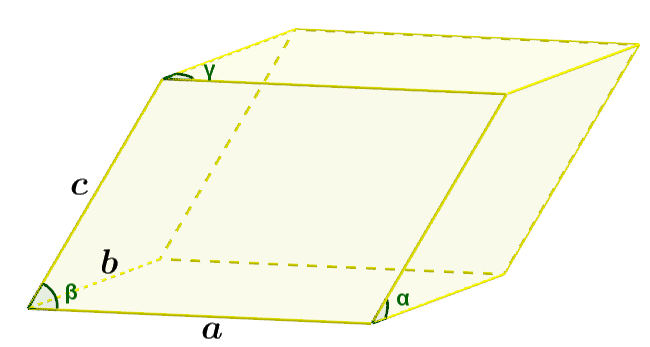}
\caption{Lattice constants of the unit cell.}
\label{fig:lattice constant}
\end{figure}

According to the periodicity, one may use the primitive cell to replace the unit cell in the calculation. The primitive cell is a minimum volume cell that is spanned by the primitive lattice vectors  $\tilde{\bf a}_1$, $\tilde{\bf a}_2$, and $\tilde{\bf a}_3$. Any other lattice vector is the linear combination of the primitive lattice vectors, with integral coefficients. In other words, the primitive lattice vectors constitute a basis of a selected Bravais lattice. Starting from a lattice point, one can reach every point by an integral linear combination of these lattice vectors.
\begin{table}[H]
\center
\includegraphics[height=7cm]{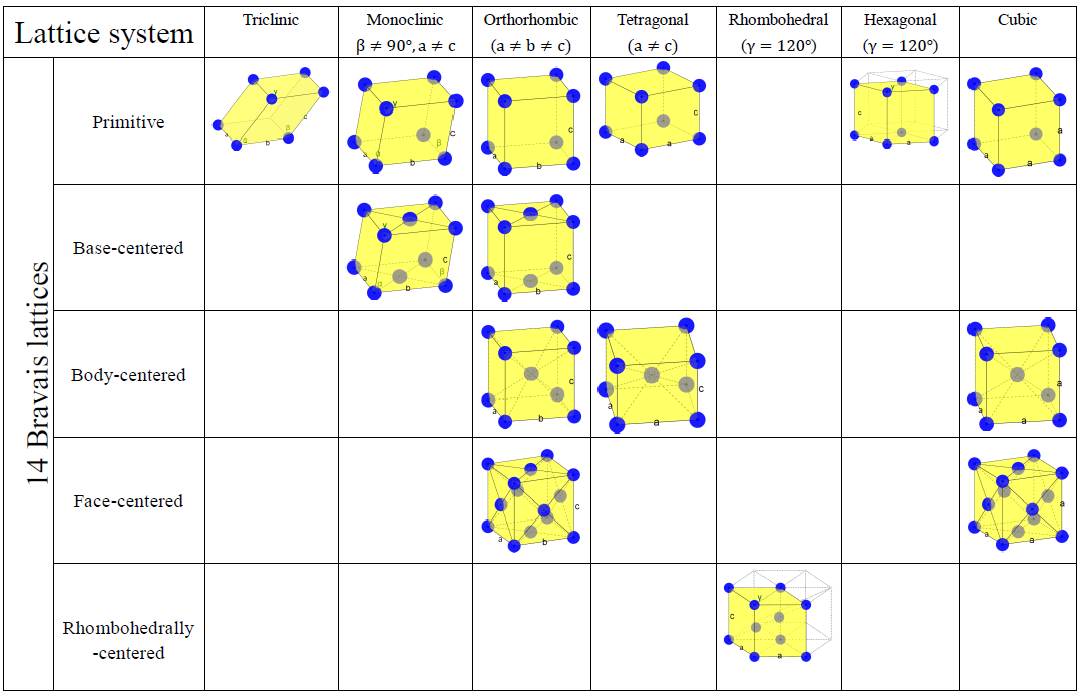}
\caption{The classification of 14 Bravais lattices in 3D space.}
\label{table:bravais}
\end{table}

\subsection{Primitive Translation Vectors}

For the purpose of introduce the FFT technique, we only consider the rectangular with edge parallel to the Cartesian coordinate axes, and we made a little modification to the standard primitive lattice vectors. After take an isometric transformation on them, we can get a new set of  lattice translation vectors ${\bf a}'_1$, ${\bf a}'_2$, and ${\bf a}'_3$. The isometric transformation will consist of the rotation of one vector to the $x$-axis, followed by another vector be located on the $xy$-plane. Furthermore, we get a new primitive cell is spanned by the lattice translation vectors. The new lattice translation vectors ${\bf a}'_{1}$, ${\bf a}'_{2}$ and ${\bf a}'_{3}$ are of the form
\begin{align*}
\begin{bmatrix}
\mbox{\bf a}'_{1} & \mbox{\bf a}'_{2} & \mbox{\bf a}'_{3}
\end{bmatrix} = \begin{bmatrix}
\|{\bf a}_1\| & \|{\bf a}_2\|\cos\theta_{\gamma}' & \|{\bf a}_3\|\cos\theta_{\beta}'\\
0 & \|{\bf a}_2\|\sin\theta_{\gamma}' & \|{\bf a}_3\|\frac{\cos\theta_{\alpha}' - \cos\theta_{\gamma}'\cos\theta_{\beta}'}{\sin\theta_{\gamma}'}\\
0 & 0 & \|{\bf a}_3\|\frac{\sqrt{1-\cos^{2}\theta_{\gamma}' - \cos^{2}\theta_{\beta}' - \cos^{2}\theta_{\alpha}' + 2\cos\theta_{\gamma}'\cos\theta_{\beta}'\cos\theta_{\alpha}'}}{\sin\theta_{\gamma}'}
\end{bmatrix},
\end{align*}
where $\|{\bf a}_1\|$, $\|{\bf a}_2\|$, and $\|{\bf a}_3\|$ are lengths of $\tilde{\bf a}_{1}$, $\tilde{\bf a}_{2}$, and $\tilde{\bf a}_{3}$, respectively, $\theta_{\gamma}'$ is the angle between $\tilde{\bf a}_{1}$ and $\tilde{\bf a}_{2}$, $\theta_{\beta}'$ is the angle between $\tilde{\bf a}_{1}$ and $\tilde{\bf a}_{3}$, and $\theta_{\alpha}'$ is the angle between $\tilde{\bf a}_{2}$ and $\tilde{\bf a}_{3}$ ($0 < \theta_{\gamma}', \theta_{\beta}', \theta_{\alpha}' < \pi$). These constants must satisfy the conditions $\|{\bf a}_1\| \geq \|{\bf a}_2\|, \|{\bf a}_3\|$ and $\|{\bf a}_2\|\sin\theta_{\gamma}' > \|{\bf a}_3\|\frac{\left|\cos\theta_{\alpha}' - \cos\theta_{\gamma}'\cos\theta_{\beta}'\right|}{\sin\theta_{\gamma}'}$, the constrains make all the lattice translation vectors ${\bf a}'_{1}$, ${\bf a}'_{2}$, and ${\bf a}'_{3}$ fall in the new primitive cell, otherwise the matrix representation will not have a general form while the angle changes.

First let the length of the primitive cell $l_{1} = \|{\bf a}_1\|$, then the mesh size along the $x$-axis direction $\delta_{x} = \frac{l_{1}}{n_{1}}$. According to the angle $\theta_{\gamma}'$, we have two cases as below:

\begin{itemize}
\item[{\bf{Case I:}}] $0 < \theta_{\gamma}' \leq \frac{\pi}{2}$. There is a nonnegative integer $0 \leq m_{1} \leq n_{1}$. We define
\begin{align}
m_{1} = \begin{cases}
\lfloor \frac{\|{\bf a}_2\|}{\delta_{x}}\cos\theta_{\gamma}'\rfloor &\mbox{if } 0 \leq \frac{\|{\bf a}_2\|}{\delta_{x}}\cos\theta_{\gamma}' - \lfloor \frac{\|{\bf a}_2\|}{\delta_{x}}\cos\theta_{\gamma}'\rfloor < \frac{1}{2},\\
\lceil \frac{\|{\bf a}_2\|}{\delta_{x}}\cos\theta_{\gamma}'\rceil &\mbox{if } 0 < \lceil \frac{\|{\bf a}_2\|}{\delta_{x}}\cos\theta_{\gamma}'\rceil - \frac{\|{\bf a}_2\|}{\delta_{x}}\cos\theta_{\gamma}' \leq \frac{1}{2}.
\end{cases}
\end{align}
Note that if $\lfloor \frac{\|{\bf a}_2\|}{\delta_{x}}\cos\theta_{\gamma}'\rfloor \neq \frac{\|{\bf a}_2\|}{\delta_{x}}\cos\theta_{\gamma}'$, the periodicity may be fail.  That is, either the initial point or the end point of ${\bf a}_{2}'$ will not be the grid point. In such case, we need to correct the angle $\theta_{\gamma}'$. To correct the angle, we set $\cos\theta_{\gamma} = \frac{m_{1}\delta_{x}}{\|{\bf a}_2\|}$, $\sin\theta_{\gamma} = \frac{\sqrt{\|{\bf a}_2\|^{2}-m_{1}^{2}\delta_{x}^{2}}}{\|{\bf a}_2\|}$. Of course $\theta_{\gamma} = \theta_{\gamma}'$ if $\lfloor \frac{\|{\bf a}_2\|}{\delta_{x}}\cos\theta_{\gamma}'\rfloor = \frac{\|{\bf a}_2\|}{\delta_{x}}\cos\theta_{\gamma}'$.

\item[{\bf{Case II:}}] $\frac{\pi}{2} < \theta_{\gamma}' < \pi$. There is a positive integer $0 \leq m_{1} \leq n_{1}$. We define
\begin{align}
m_{1} = \begin{cases}
\lfloor \frac{\|{\bf a}_1\| + \|{\bf a}_2\|\cos\theta_{\gamma}'}{\delta_{x}} \rfloor &\mbox{if } 0 \leq \frac{\|{\bf a}_1\| + \|{\bf a}_2\|\cos\theta_{\gamma}'}{\delta_{x}} - \lfloor \frac{\|{\bf a}_1\| + \|{\bf a}_2\|\cos\theta_{\gamma}'}{\delta_{x}}\rfloor < \frac{1}{2},\\
\lceil \frac{\|{\bf a}_1\| + \|{\bf a}_2\|\cos\theta_{\gamma}'}{\delta_{x}}\rceil &\mbox{if } 0 < \lceil \frac{\|{\bf a}_1\| + \|{\bf a}_2\|\cos\theta_{\gamma}'}{\delta_{x}}\rceil - \frac{\|{\bf a}_1\| + \|{\bf a}_2\|\cos\theta_{\gamma}'}{\delta_{x}} \leq \frac{1}{2}.
\end{cases}
\end{align}
If $\lfloor \frac{\|{\bf a}_1\| + \|{\bf a}_2\|\cos\theta_{\gamma}'}{\delta_{x}}\rfloor \neq \frac{\|{\bf a}_1\| + \|{\bf a}_2\|\cos\theta_{\gamma}'}{\delta_{x}}$, the periodicity may be fail.  In such case, we need to correct the angle $\theta_{\gamma}'$. To correct the angle, we set $\cos\theta_{\gamma} = \frac{m_{1}\delta_{x} - \|{\bf a}_1\|}{\|{\bf a}_2\|}$, $\sin\theta_{\gamma} = \frac{\sqrt{\|{\bf a}_2\|^{2} - (m_{1}\delta_{x} - \|{\bf a}_1\|)^{2}}}{\|{\bf a}_2\|}$. Of course $\theta_{\gamma} = \theta_{\gamma}'$ if $\lfloor \frac{\|{\bf a}_1\| + \|{\bf a}_2\|\cos\theta_{\gamma}'}{\delta_{x}}\rfloor = \frac{\|{\bf a}_1\| + \|{\bf a}_2\|\cos\theta_{\gamma}'}{\delta_{x}}$.
\end{itemize}

Now we can define the width of the primitive cell $l_{2} = \|{\bf a}_2\|\sin\theta_{\gamma}$, and the mesh size along the $y$-axis direction $\delta_{y} = \frac{l_{2}}{n_{2}}$. Similarly, according to the angle $\theta_{\beta}'$, we have two cases as below:
\begin{itemize}
\item[{\bf{Case I:}}] $0 < \theta_{\beta}' \leq \frac{\pi}{2}$. There is a nonnegative integer $0 \leq m_{2} \leq n_{1}$. We define
\begin{align}
m_{2} = \begin{cases}
\lfloor \frac{\|{\bf a}_3\|}{\delta_{x}}\cos\theta_{\beta}'\rfloor &\mbox{if } 0 \leq \frac{\|{\bf a}_3\|}{\delta_{x}}\cos\theta_{\beta}' - \lfloor \frac{\|{\bf a}_3\|}{\delta_{x}}\cos\theta_{\beta}'\rfloor < \frac{1}{2},\\
\lceil \frac{\|{\bf a}_3\|}{\delta_{x}}\cos\theta_{\beta}'\rceil &\mbox{if } 0 < \lceil \frac{\|{\bf a}_3\|}{\delta_{x}}\cos\theta_{\beta}'\rceil - \frac{\|{\bf a}_3\|}{\delta_{x}}\cos\theta_{\beta}' \leq \frac{1}{2}.
\end{cases}
\end{align}
Note that if $\lfloor \frac{\|{\bf a}_3\|}{\delta_{x}}\cos\theta_{\beta}'\rfloor \neq \frac{\|{\bf a}_3\|}{\delta_{x}}\cos\theta_{\beta}'$, the periodicity may be fail.  In such case, we need to correct the angle $\theta_{\beta}'$. To correct the angle, we set $\cos\theta_{\beta} = \frac{m_{2}\delta_{x}}{\|{\bf a}_3\|}$, $\sin\theta_{\beta} = \frac{\sqrt{\|{\bf a}_3\|^{2}-m_{2}^{2}\delta_{x}^{2}}}{\|{\bf a}_3\|}$. Of course $\theta_{\beta} = \theta_{\beta}'$ if $\lfloor \frac{\|{\bf a}_3\|}{\delta_{x}}\cos\theta_{\beta}'\rfloor = \frac{\|{\bf a}_3\|}{\delta_{x}}\cos\theta_{\beta}'$.

\item[{\bf{Case II:}}] $\frac{\pi}{2} < \theta_{\beta}' < \pi$. There is a positive integer $0 \leq m_{2} \leq n_{1}$. We define
\begin{align}
m_{2} = \begin{cases}
\lfloor \frac{\|{\bf a}_1\| + \|{\bf a}_3\|\cos\theta_{\beta}'}{\delta_{x}} \rfloor &\mbox{if } 0 \leq \frac{\|{\bf a}_1\| + \|{\bf a}_3\|\cos\theta_{\beta}'}{\delta_{x}} - \lfloor \frac{\|{\bf a}_1\| + \|{\bf a}_3\|\cos\theta_{\beta}'}{\delta_{x}}\rfloor < \frac{1}{2},\\
\lceil \frac{\|{\bf a}_1\| + \|{\bf a}_3\|\cos\theta_{\beta}'}{\delta_{x}}\rceil &\mbox{if } 0 < \lceil \frac{\|{\bf a}_1\| + \|{\bf a}_3\|\cos\theta_{\beta}'}{\delta_{x}}\rceil - \frac{\|{\bf a}_1\| + \|{\bf a}_3\|\cos\theta_{\beta}'}{\delta_{x}} \leq \frac{1}{2}.
\end{cases}
\end{align}
If $\lfloor \frac{\|{\bf a}_1\| + \|{\bf a}_3\|\cos\theta_{\beta}'}{\delta_{x}}\rfloor \neq \frac{\|{\bf a}_1\| + \|{\bf a}_3\|\cos\theta_{\beta}'}{\delta_{x}}$, the periodicity may be fail.  In such case, we need to correct the angle $\theta_{\beta}'$. To correct the angle, we set $\cos\theta_{\beta} = \frac{m_{2}\delta_{x} - \|{\bf a}_1\|}{\|{\bf a}_3\|}$, $\sin\theta_{\beta} = \frac{\sqrt{\|{\bf a}_3\|^{2} - (m_{2}\delta_{x} - \|{\bf a}_1\|)^{2}}}{\|{\bf a}_3\|}$. Of course $\theta_{\beta} = \theta_{\beta}'$ if $\lfloor \frac{\|{\bf a}_1\| + \|{\bf a}_3\|\cos\theta_{\beta}'}{\delta_{x}}\rfloor = \frac{\|{\bf a}_1\| + \|{\bf a}_3\|\cos\theta_{\beta}'}{\delta_{x}}$.
\end{itemize}

Finally, according to the value $\cos\theta_{\alpha}' - \cos\theta_{\gamma}\cos\theta_{\beta}$, we have two cases as below:
\begin{itemize}
\item[{\bf{Case I:}}] $0 < \cos\theta_{\alpha}' - \cos\theta_{\gamma}\cos\theta_{\beta}$. There is a nonnegative integer $0 \leq m_{3} \leq n_{2}$. We define
\begin{align}
m_{3} = \begin{cases}
\lfloor \frac{\|{\bf a}_3\|}{\delta_{y}}\frac{\cos\theta_{\alpha}' - \cos\theta_{\gamma}\cos\theta_{\beta}}{\sin\theta_{\gamma}}\rfloor \hspace{1em}\mbox{if }\\
\hspace{2cm} 0 \leq \frac{\|{\bf a}_3\|}{\delta_{y}}\frac{\cos\theta_{\alpha}' - \cos\theta_{\gamma}\cos\theta_{\beta}}{\sin\theta_{\gamma}} - \lfloor \frac{\|{\bf a}_3\|}{\delta_{y}}\frac{\cos\theta_{\alpha}' - \cos\theta_{\gamma}\cos\theta_{\beta}}{\sin\theta_{1}}\rfloor < \frac{1}{2},\\
\lceil \frac{\|{\bf a}_3\|}{\delta_{y}}\frac{\cos\theta_{\alpha}' - \cos\theta_{\gamma}\cos\theta_{\beta}}{\sin\theta_{\gamma}}\rceil \hspace{1em}\mbox{if }\\
\hspace{2cm} 0 < \lceil \frac{\|{\bf a}_3\|}{\delta_{y}}\frac{\cos\theta_{\alpha}' - \cos\theta_{\gamma}\cos\theta_{\beta}}{\sin\theta_{\gamma}}\rceil - \frac{\|{\bf a}_3\|}{\delta_{y}}\frac{\cos\theta_{\alpha}' - \cos\theta_{\gamma}\cos\theta_{\beta}}{\sin\theta_{\gamma}} \leq \frac{1}{2}.
\end{cases}
\end{align}
Note that if $\lfloor \frac{\|{\bf a}_3\|}{\delta_{y}}\frac{\cos\theta_{\alpha}' - \cos\theta_{\gamma}\cos\theta_{\beta}}{\sin\theta_{\gamma}}\rfloor \neq \frac{\|{\bf a}_3\|}{\delta_{y}}\frac{\cos\theta_{\alpha}' - \cos\theta_{\gamma}\cos\theta_{\beta}}{\sin\theta_{\gamma}}$, the periodicity may be fail.  In such case, we need to correct the angle $\theta_{\alpha}'$. To correct the angle, we set $\cos\theta_{\alpha} = \frac{m_{3}\delta_{y}}{\|{\bf a}_3\|}\sin\theta_{\gamma} + \cos\theta_{\gamma}\cos\theta_{\beta}$. Of course $\theta_{\alpha} = \theta_{\alpha}'$ if $\lfloor \frac{\|{\bf a}_3\|}{\delta_{y}}\frac{\cos\theta_{\alpha}' - \cos\theta_{\gamma}\cos\theta_{\beta}}{\sin\theta_{\gamma}}\rfloor = \frac{\|{\bf a}_3\|}{\delta_{y}}\frac{\cos\theta_{\alpha}' - \cos\theta_{\gamma}\cos\theta_{\beta}}{\sin\theta_{\gamma}}$.

\item[{\bf{Case II:}}] $\cos\theta_{\alpha}' - \cos\theta_{\gamma}\cos\theta_{\beta} < 0$. There is a positive integer $0 \leq m_{3} \leq n_{2}$. We define
\begin{align}
m_{3} = \begin{cases}
\lfloor \frac{\|{\bf a}_2\|\sin^{2}\theta_{\gamma} + \|{\bf a}_3\|(\cos\theta_{\alpha}' - \cos\theta_{\gamma}\cos\theta_{\beta})}{\delta_{y}\sin\theta_{\gamma}} \rfloor \hspace{1em}\mbox{if }\\
 \hspace{1em}0 \leq \frac{\|{\bf a}_2\|\sin^{2}\theta_{\gamma} + \|{\bf a}_3\|(\cos\theta_{\alpha}' - \cos\theta_{\gamma}\cos\theta_{\beta})}{\delta_{y}\sin\theta_{\gamma}} - \lfloor \frac{\|{\bf a}_2\|\sin^{2}\theta_{\gamma} + \|{\bf a}_3\|(\cos\theta_{\alpha}' - \cos\theta_{\gamma}\cos\theta_{\beta})}{\delta_{y}\sin\theta_{\gamma}}\rfloor < \frac{1}{2},\\
\lceil \frac{\|{\bf a}_2\|\sin^{2}\theta_{\gamma} + \|{\bf a}_3\|(\cos\theta_{\alpha}' - \cos\theta_{\gamma}\cos\theta_{\beta})}{\delta_{y}\sin\theta_{\gamma}}\rceil \hspace{1em}\mbox{if }\\
\hspace{1em} 0 < \lceil \frac{\|{\bf a}_2\|\sin^{2}\theta_{\gamma} + \|{\bf a}_3\|(\cos\theta_{\alpha}' - \cos\theta_{\gamma}\cos\theta_{\beta})}{\delta_{y}\sin\theta_{\gamma}}\rceil - \frac{\|{\bf a}_2\|\sin^{2}\theta_{\gamma} + \|{\bf a}_3\|(\cos\theta_{\alpha}' - \cos\theta_{\gamma}\cos\theta_{\beta})}{\delta_{y}\sin\theta_{\gamma}} \leq \frac{1}{2}.
\end{cases}
\end{align}
If $\lfloor \frac{\|{\bf a}_2\|\sin^{2}\theta_{\gamma} + \|{\bf a}_3\|(\cos\theta_{\alpha}' - \cos\theta_{\gamma}\cos\theta_{\beta})}{\delta_{y}\sin\theta_{\gamma}}\rfloor \neq \frac{\|{\bf a}_2\|\sin^{2}\theta_{\gamma} + \|{\bf a}_3\|(\cos\theta_{\alpha}' - \cos\theta_{\gamma}\cos\theta_{\beta})}{\delta_{y}\sin\theta_{\gamma}}$, the periodicity may be fail.  In such case, we need to correct the angle $\theta_{\alpha}'$. To correct the angle, we set $\cos\theta_{\alpha} = \frac{m_{3}\delta_{y}\sin\theta_{\gamma} - \|{\bf a}_2\|\sin^{2}\theta_{\gamma}}{\|{\bf a}_3\|} + \cos\theta_{\gamma}\cos\theta_{\beta}$. Of course $\theta_{\alpha} = \theta_{\alpha}'$ if $\lfloor \frac{\|{\bf a}_2\|\sin^{2}\theta_{\gamma} + \|{\bf a}_3\|(\cos\theta_{\alpha}' - \cos\theta_{\gamma}\cos\theta_{\beta})}{\delta_{y}\sin\theta_{\gamma}}\rfloor = \frac{\|{\bf a}_2\|\sin^{2}\theta_{\gamma} + \|{\bf a}_3\|(\cos\theta_{\alpha}' - \cos\theta_{\gamma}\cos\theta_{\beta})}{\delta_{y}\sin\theta_{\gamma}}$.
\end{itemize}

After the modification, the primitive lattice vectors ${\bf a}_{1}$, ${\bf a}_{2}$, and ${\bf a}_{3}$ now becomes
\begin{align}\label{eq:transvec}
\begin{split}
\begin{bmatrix}
{\bf a}_{1} & {\bf a}_{2} & {\bf a}_{3}\\
\end{bmatrix} = \begin{bmatrix}
\|{\bf a}_1\| & \|{\bf a}_2\|\cos\theta_{\gamma} & \|{\bf a}_3\|\cos\theta_{\beta}\\
0 & \|{\bf a}_2\|\sin\theta_{\gamma} & \|{\bf a}_3\|\frac{\cos\theta_{\alpha} - \cos\theta_{\gamma}\cos\theta_{\beta}}{\sin\theta_{\gamma}}\\
0 & 0 & \|{\bf a}_3\|\frac{\sqrt{1-\cos^{2}\theta_{\gamma} - \cos^{2}\theta_{\beta} - \cos^{2}\theta_{\alpha} + 2\cos\theta_{\gamma}\cos\theta_{\beta}\cos\theta_{\alpha}}}{\sin\theta_{\gamma}}
\end{bmatrix}.
\end{split}
\end{align}

The primitive cell that was spanned by the lattice translation vectors ${\bf a}_{1}$, ${\bf a}_{2}$ and ${\bf a}_{3}$ has length $l_{1}$, width $l_{2}$, and height $l_{3}$, these three numbers are given as $l_{1} = \|{\bf a}_1\|$, $l_{2} = \|{\bf a}_2\|\sin\theta_{\gamma}$, and $l_{3} = \|{\bf a}_3\|\frac{\sqrt{1-\cos^{2}\theta_{\gamma} - \cos^{2}\theta_{\beta} - \cos^{2}\theta_{\alpha} + 2\cos\theta_{\gamma}\cos\theta_{\beta}\cos\theta_{\alpha}}}{\sin\theta_{\gamma}}$, respectively. Furthermore, the grid length are $\delta_{x} = \frac{l_{1}}{n_{1}}$, $\delta_{y} = \frac{l_{2}}{n_{2}}$, and $\delta_{z} = \frac{l_{3}}{n_{3}}$.

\subsection{Reciprocal Lattice and Brillouin Zone}

Any Bravais lattice will become reciprocal lattice after Fourier transformation \cite{shmu2008}. In physics, reciprocal lattice is often used to describe the periodicity of the momentum space (also called reciprocal space or ``{\bf k}-space"). For any kind of Bravais lattice, if the primitive lattice vectors $\tilde{\bf a}_{1}$, $\tilde{\bf a}_{2}$, and $\tilde{\bf a}_{3}$ are fixed, then the corresponding reciprocal primitive lattice vectors are defined as
\begin{align*}
\begin{bmatrix}
\tilde{\bf b}_{1} & \tilde{\bf b}_{2} & \tilde{\bf b}_{3}
\end{bmatrix} = 2\pi\begin{bmatrix}
\tilde{\bf a}_{1} & \tilde{\bf a}_{2} & \tilde{\bf a}_{3}
\end{bmatrix}^{-\top}.
\end{align*}
Here, $\{\tilde{\bf b}_{1}, \tilde{\bf b}_{2}, \tilde{\bf b}_{3}\}$ constitute a basis of {\bf k}-space, any vectors in {\bf k}-space can be represented by the linearly combination of this basis.

The first Brillouin zone is a primitive cell contains the origin in the reciprocal space, and is the set of points in {\bf k}-space which start from the origin and do not pass through any Bragg plane. The irreducible Brillouin zone is a part of the first Brillouin zone, it can be obtained from reduce all of the symmetries in the point group. The importance of the Brillouin zone is that every Bloch vector in the periodic material can be certeined in this space. We provide all the first Brillouin zones and irreducible Brillouin zones in Appendix. If we want to compute the band structure, we only have to solve the eigenvalue problems associated with the wave vectors along the segments connecting any two corner points of the associated irreducible Brillouin zone.

Since we transferred the lattice vectors $\tilde{\bf a}_{1}$, $\tilde{\bf a}_{2}$, and $\tilde{\bf a}_{3}$ into the lattice translation vectors ${\bf a}_{1}$, ${\bf a}_{2}$, and ${\bf a}_{3}$, we need to re-define the corresponding reciprocal lattice vectors as
\begin{align*}
\begin{bmatrix}
{\bf b}_{1} & {\bf b}_{2} & {\bf b}_{3}
\end{bmatrix} = 2\pi\begin{bmatrix}
{\bf a}_{1} & {\bf a}_{2} & {\bf a}_{3}
\end{bmatrix}^{-\top}.
\end{align*}
The transformation $\Omega$ between reciprocal primitive lattice and the new reciprocal lattice satisfies
\begin{align*}
\begin{bmatrix}
{\bf b}_{1} & {\bf b}_{2} & {\bf b}_{3}
\end{bmatrix} = \Omega\begin{bmatrix}
\tilde{\bf b}_{1} & \tilde{\bf b}_{2} & \tilde{\bf b}_{3}
\end{bmatrix},
\end{align*}
so the transformation $\Omega$ can be defined as
\begin{align*}
\Omega = \begin{bmatrix}
{\bf b}_{1} & {\bf b}_{2} & {\bf b}_{3}
\end{bmatrix}\begin{bmatrix}
\tilde{\bf b}_{1} & \tilde{\bf b}_{2} & \tilde{\bf b}_{3}
\end{bmatrix}^{-1} = \begin{bmatrix}
{\bf a}_{1} & {\bf a}_{2} & {\bf a}_{3}
\end{bmatrix}^{-\top}\begin{bmatrix}
\tilde{\bf a}_{1} & \tilde{\bf a}_{2} & \tilde{\bf a}_{3}
\end{bmatrix}^{\top}.
\end{align*}
Using the transformation, we can define the new corner points of the new irreducible Brillouin zone. Take BCC as an example, the corner points are: 
\begin{align*}
\widetilde{\Gamma} &= \tfrac{2\pi}{a}\begin{bmatrix}
0 & 0 & 0
\end{bmatrix}^{\top}, &\widetilde{H} = \tfrac{2\pi}{a}\begin{bmatrix}
0 & 1 & 0
\end{bmatrix}^{\top},\\
\widetilde{N} &= \tfrac{2\pi}{a}\begin{bmatrix}
\tfrac{1}{2} & \tfrac{1}{2} & 0
\end{bmatrix}^{\top}, &\widetilde{P} = \tfrac{2\pi}{a}\begin{bmatrix}
\tfrac{1}{2} & \tfrac{1}{2} & \tfrac{1}{2}
\end{bmatrix}^{\top}.
\end{align*}
We redefine the new corner points to be $\Gamma = \Omega\tilde{\Gamma}$, $H = \Omega\tilde{H}$, $N = \Omega\tilde{N}$, and $P = \Omega\tilde{P}$.

\subsection{Point Groups and Space Groups}

In crystallography, the most important tools used to describe symmetry are the point groups and the space groups. Each point group is a symmetric group containing a variety symmetric operators, and at least one point is fixed after the operators. Space group is a collection of all possible symmetric operations in various dimensions, usually in 3D real space. The most difference between this two is: point group is a microscopic proprty, from the fixed point of view, the other atoms seem to be invariant after the symmetric operations of the point group, since we only consider the directions from the fixed point to other atoms, the distance is not taken into account, so there is no translation operator in the point group; space group is a macroscopic property, and can describe the relation position of the atomic arrangement in the crystal from the outside. All the materials we study belong to one of the 230 space groups.

There are many methods of naming point groups, and the two most commonly used are Sch\"{o}nflies notation and Hermann-Mauguin notation, where Sch\"{o}nflies notation is the language used by mathematicians, and all symbols are the same as in algebra; Hermann-Mauguin notation also known as international symbol, one can easily understanding the symmetry in point group from this notation. Now let us detail the point groups and the space groups.

Although there are infinite point groups in 3D real space, however, there are only 32 point groups in crystals from the crystallographic restriction theorem. Compared to the point groups, the space groups contain screw axes and glide planes in addition, where screw axis means the rotation axis add a translation operator along the axis, and glide plane means the mirror plane add a translation operator along one side of the plane. Please find more details by refer to \cite{hahn2002}.


\section{Explicit matrix representation of the curl operator}\label{section3}

In \cite{hhlw2013matrix}, Huang et al derived the matrix representation of curl operator for FCC lattice, since all the length of lattice vectors ${\bf a}_{1}$, ${\bf a}_{2}$, ${\bf a}_{3}$ are equal and the angles between each two vectors are special, this case is relatively simple. Now we are going to give an exact matrix representation of the curl operator for various lattices.  As we have illustrated before, the govering equation \eqref{eq:double_curl} can be rewritten as following
\begin{align*}
\nabla \times E &= \imath\omega\mu_{0} H,\\
\nabla \times H &= -\imath\omega\varepsilon E,
\end{align*}
and we use the Yee's scheme to descretize these two equations. The curl operator can be represented as the matrix form
\begin{align}
\nabla \times E = \begin{bmatrix}
0 & -\partial_{z} & \partial_{y}\\
\partial_{z} & 0 & -\partial_{x}\\
-\partial_{y} & \partial_{x} & 0
\end{bmatrix}\begin{bmatrix}
E_{1}\\
E_{2}\\
E_{3}
\end{bmatrix}.
\end{align}
Let $C$ be the discrete matrix of curl operator which is defined by
\begin{align*}
C = \begin{bmatrix}
0 & -C_{3} & C_{2}\\
C_{3} & 0 & -C_{1}\\
-C_{2} & C_{1} & 0
\end{bmatrix}\in \mathbb{C}^{3n \times 3n},
\end{align*}
where $C_{1}$, $C_{2}$, $C_{3}$ are the discretization of the partial derivative operators $\partial_{x}$, $\partial_{y}$, and $\partial_{z}$, respectively, the double curl operator $\nabla \times\nabla\times$ has the discrete matrix form $C^{\ast}C$. From now on, we set $\varepsilon = [\varepsilon_{1}, \varepsilon_{2}, \varepsilon_{3}]^{\top}$ and 
\begin{align*}
B_{1} = B_{1}(\hat{i}, j, k) = \varepsilon_{1}({\bf x}(\hat{i}, j, k)),\\
B_{2} = B_{2}(i, \hat{j}, k) = \varepsilon_{2}({\bf x}(i, \hat{j}, k)),\\
B_{3} = B_{3}(i, j, \hat{k}) = \varepsilon_{3}({\bf x}(i, j, \hat{k})),
\end{align*}
for $i = 0, 1, \cdots, n_{1} - 1$, $j = 0, 1, \cdots, n_{2} - 1$, $k = 0, 1, \cdots, n_{3}-1$, and set
\begin{align*}
B = \mbox{diag}\left[[\mbox{vec}({B_{1}})^{\top}, \mbox{vec}({B_{2}})^{\top}, \mbox{vec}({B_{3}})^{\top}]^{\top}\right] \in \mathbb{C}^{3n \times 3n}.
\end{align*}
Let ${\bf e}_{\ell} = \mbox{vec}(E_{\ell})$ and ${\bf h}_{\ell} = \mbox{vec}(H_{\ell})$ for $\ell = 1, 2, 3$, then we can define
\begin{align*}
{\bf e} = [{\bf  e}_{1}^{\top}, {\bf  e}_{2}^{\top}, {\bf  e}_{3}^{\top}]^{\top} \in \mathbb{C}^{3n} \;\; \mbox{and} \;\; {\bf h} = [{\bf  h}_{1}^{\top}, {\bf  h}_{2}^{\top}, {\bf  h}_{3}^{\top}]^{\top} \in \mathbb{C}^{3n}.
\end{align*}

In this section, we will derive the matrix representation of $C_{1}$, $C_{2}$, and $C_{3}$, and therefore obtain the explicit form of $C$.


\subsection{Matrix representation of curl operator}

We first list all the matrix representations of the partial derivative operators:
\begin{align*}
C_{1} = I_{n_{3}} \otimes I_{n_{2}} \otimes K_{1} \in \mathbb{C}^{n \times n}, 
\end{align*}
and
\begin{align}\label{eq:K1}
K_{1} = \dfrac{1}{\delta_{x}}\begin{bmatrix}
-1 & 1 & & & \\
 & -1 & 1 &  & \\
 & & \ddots & \ddots & \\
 & & & -1 & 1\\
e^{\imath 2\pi {\bf k}\cdot{\bf a}_{1}} & & & & -1 
\end{bmatrix} \in \mathbb{C}^{n_{1} \times n_{1}},
\end{align}
\begin{align*}
C_{2} = I_{n_{3}} \otimes K_{2} \in \mathbb{C}^{n \times n},
\end{align*}
and
\begin{align}\label{eq:K2}
K_{2} = \dfrac{1}{\delta_{y}}\begin{bmatrix}
-I_{n_{1}} & I_{n_{1}} & & & \\
 & -I_{n_{1}} & I_{n_{1}} & & \\
 & & \ddots & \ddots & \\
 & & & -I_{n_{1}} & I_{n_{1}}\\
e^{\imath 2\pi{\bf k}\cdot{\bf a}_{2}}J_{2} & & & & -I_{n_{1}}
\end{bmatrix} \in \mathbb{C}^{n_{1}n_{2} \times n_{1}n_{2}},
\end{align}
\begin{align}\label{eq:K3}
C_{3} = K_{3} = \dfrac{1}{\delta_{z}}\begin{bmatrix}
-I_{n_{1}\times n_{2}} & I_{n_{1}\times n_{2}} & & & \\
  & -I_{n_{1}\times n_{2}} & I_{n_{1}\times n_{2}} &  & \\
 & & \ddots & \ddots & \\
 & & & -I_{n_{1}\times n_{2}} & I_{n_{1}\times n_{2}}\\
e^{\imath 2\pi{\bf k}\cdot{\bf a}_{3}}J_{3} & & & & -I_{n_{1}\times n_{2}}
\end{bmatrix} \in \mathbb{C}^{n \times n},
\end{align}
where $J_{2}$ and $J_{3}$ are very different in each lattices as following.

\begin{subequations}
\begin{align}\label{eq:acuteJ2}
J_{2} = \begin{bmatrix}
0 & e^{-\imath 2\pi{\bf k}\cdot{\bf a}_{1}}I_{m_{1}}\\
I_{(n_{1}-m_{1})} & 0
\end{bmatrix} \hspace{1em}\mbox{if}\hspace{1em}0 < \theta_{\gamma} \leq \frac{\pi}{2},
\end{align}
and
\begin{align}\label{eq:obtuseJ2}
J_{2} = \begin{bmatrix}
0 & I_{m_{1}}\\
e^{\imath 2\pi{\bf k}\cdot{\bf a}_{1}}I_{(n_{1}-m_{1})} & 0
\end{bmatrix} \hspace{1em}\mbox{if}\hspace{1em} \frac{\pi}{2} < \theta_{\gamma} < \pi.
\end{align}
\end{subequations}

$J_{3}$ is depend on $\theta_{\alpha}$, $\theta_{\beta}$, and $\theta_{\gamma}$, all the situations we will list in the following:
\begin{subequations}\label{eq:J3case16}
\begin{itemize}
\item[1.] $0 \leq \cos\theta_{\gamma}$, $0 \leq \cos\theta_{\beta}$, $0 \leq \cos\theta_{\alpha} - \cos\theta_{\gamma}\cos\theta_{\beta}$.
\begin{itemize}
\item[(a)] $m_{1} \leq m_{2}$
\begin{align}
&J_{3}\notag\\
=& \begin{bmatrix}
\begin{smallmatrix}
 & e^{-\imath 2\pi {\bf k}\cdot{\bf a}_{2}}I_{m_{3}}\otimes\begin{bmatrix}
\begin{smallmatrix}
 & e^{-\imath 2\pi {\bf k}\cdot{\bf a}_{1}}I_{m_{2} - m_{1}}\\
I_{n_{1} - m_{2} + m_{1}} &
\end{smallmatrix}
\end{bmatrix}\\
I_{n_{2} - m_{3}}\otimes\begin{bmatrix}
\begin{smallmatrix}
 & e^{-\imath 2\pi {\bf k}\cdot{\bf a}_{1}}I_{m_{2}}\\
I_{n_{1} - m_{2}} & 
\end{smallmatrix}
\end{bmatrix} &
\end{smallmatrix}
\end{bmatrix}.
\end{align}
\item[(b)] $m_{1} > m_{2}$
\begin{align}
&J_{3}\notag\\
=& \begin{bmatrix}
\begin{smallmatrix}
 & e^{-\imath 2\pi {\bf k}\cdot{\bf a}_{2}}I_{m_{3}}\otimes\begin{bmatrix}
\begin{smallmatrix}
 & I_{n_{1} - m_{1} + m_{2}}\\
e^{\imath 2\pi {\bf k}\cdot{\bf a}_{1}}I_{ m_{1} - m_{2}} &
\end{smallmatrix}
\end{bmatrix}\\
I_{n_{2} - m_{3}}\otimes\begin{bmatrix}
\begin{smallmatrix}
 & e^{-\imath 2\pi {\bf k}\cdot{\bf a}_{1}}I_{m_{2}}\\
I_{n_{1} - m_{2}} & 
\end{smallmatrix}
\end{bmatrix} &
\end{smallmatrix}
\end{bmatrix}.
\end{align}
\end{itemize}

\item[2.] $0 \leq \cos\theta_{\gamma}$, $0 \leq \cos\theta_{\beta}$, $0 > \cos\theta_{\alpha} - \cos\theta_{\gamma}\cos\theta_{\beta}$.
\begin{itemize}
\item[(a)] $m_{1} + m_{2} \leq n_{1}$
\begin{align}
&J_{3}\notag\\
=& \begin{bmatrix}
\begin{smallmatrix}
 & I_{m_{3}}\otimes\begin{bmatrix}
\begin{smallmatrix}
 & e^{-\imath 2\pi {\bf k}\cdot{\bf a}_{1}}I_{m_{2}}\\
I_{n_{1} - m_{2}} &
\end{smallmatrix}
\end{bmatrix}\\
e^{\imath 2\pi {\bf k}\cdot{\bf a}_{2}}I_{n_{2} - m_{3}}\otimes\begin{bmatrix}
\begin{smallmatrix}
 & e^{-\imath 2\pi {\bf k}\cdot{\bf a}_{1}}I_{m_{1} + m_{2}}\\
I_{n_{1} - m_{1} - m_{2}} & 
\end{smallmatrix}
\end{bmatrix} &
\end{smallmatrix}
\end{bmatrix}.
\end{align}
\item[(b)] $m_{1} + m_{2} > n_{1}$
\begin{align}
&J_{3}\notag\\
=& \begin{bmatrix}
\begin{smallmatrix}
 & I_{m_{3}}\otimes\begin{bmatrix}
\begin{smallmatrix}
 & e^{-\imath 2\pi {\bf k}\cdot{\bf a}_{1}}I_{m_{2}}\\
I_{ n_{1} - m_{2}} &
\end{smallmatrix}
\end{bmatrix}\\
e^{\imath 2\pi {\bf k}\cdot({\bf a}_{2} - {\bf a}_{1})}I_{n_{2} - m_{3}}\otimes\begin{bmatrix}
\begin{smallmatrix}
 & e^{-\imath 2\pi {\bf k}\cdot{\bf a}_{1}}I_{m_{1} + m_{2} - n_{1}}\\
I_{2n_{1} - m_{1} - m_{2}} & 
\end{smallmatrix}
\end{bmatrix} &
\end{smallmatrix}
\end{bmatrix}.
\end{align}
\end{itemize}

\item[3.] $0 \leq \cos\theta_{\gamma}$, $0 > \cos\theta_{\beta}$, $0 \leq \cos\theta_{\alpha} - \cos\theta_{\gamma}\cos\theta_{\beta}$.
\begin{itemize}
\item[(a)] $m_{1} \leq m_{2}$
\begin{align}
&J_{3}\notag\\
=& \begin{bmatrix}
\begin{smallmatrix}
 & e^{-\imath 2\pi {\bf k}\cdot{\bf a}_{2}}I_{m_{3}}\otimes\begin{bmatrix}
\begin{smallmatrix}
 & I_{m_{2} - m_{1}}\\
e^{\imath 2\pi {\bf k}\cdot{\bf a}_{1}}I_{n_{1} - m_{2} + m_{1}} &
\end{smallmatrix}
\end{bmatrix}\\
I_{n_{2} - m_{3}}\otimes\begin{bmatrix}
\begin{smallmatrix}
 & I_{m_{2}}\\
e^{\imath 2\pi {\bf k}\cdot{\bf a}_{1}}I_{n_{1} - m_{2}} & 
\end{smallmatrix}
\end{bmatrix} &
\end{smallmatrix}
\end{bmatrix}.
\end{align}
\item[(b)] $m_{1} > m_{2}$
\begin{align}
&J_{3}\notag\\
=& \begin{bmatrix}
\begin{smallmatrix}
 & e^{\imath 2\pi {\bf k}\cdot({\bf a}_{1} - {\bf a}_{2})}I_{m_{3}}\otimes\begin{bmatrix}
\begin{smallmatrix}
 & I_{n_{1} - m_{1}+m_{2}}\\
e^{\imath 2\pi {\bf k}\cdot{\bf a}_{1}}I_{ m_{1} - m_{2}} &
\end{smallmatrix}
\end{bmatrix}\\
I_{n_{2} - m_{3}}\otimes\begin{bmatrix}
\begin{smallmatrix}
 & I_{m_{2}}\\
e^{\imath 2\pi {\bf k}\cdot{\bf a}_{1}}I_{n_{1}- m_{2}} & 
\end{smallmatrix}
\end{bmatrix} &
\end{smallmatrix}
\end{bmatrix}.
\end{align}
\end{itemize}

\item[4.] $0 \leq \cos\theta_{\gamma}$, $0 > \cos\theta_{\beta}$, $0 > \cos\theta_{\alpha} - \cos\theta_{\gamma}\cos\theta_{\beta}$.
\begin{itemize}
\item[(a)] $m_{1} + m_{2}\leq n_{1}$
\begin{align}
&J_{3}\notag\\
=& \begin{bmatrix}
\begin{smallmatrix}
 & I_{m_{3}}\otimes\begin{bmatrix}
\begin{smallmatrix}
 & I_{m_{2}}\\
e^{\imath 2\pi {\bf k}\cdot{\bf a}_{1}}I_{n_{1} - m_{2}} &
\end{smallmatrix}
\end{bmatrix}\\
e^{\imath 2\pi {\bf k}\cdot{\bf a}_{2}}I_{n_{2} - m_{3}}\otimes\begin{bmatrix}
\begin{smallmatrix}
 & I_{m_{1} + m_{2}}\\
e^{\imath 2\pi {\bf k}\cdot{\bf a}_{1}}I_{n_{1} - m_{1} - m_{2}} & 
\end{smallmatrix}
\end{bmatrix} &
\end{smallmatrix}
\end{bmatrix}.
\end{align}
\item[(b)] $m_{1}  + m_{1} > n_{1}$
\begin{align}
&J_{3}\notag\\
=& \begin{bmatrix}
\begin{smallmatrix}
 & I_{m_{3}}\otimes\begin{bmatrix}
\begin{smallmatrix}
 & I_{m_{2}}\\
e^{\imath 2\pi {\bf k}\cdot{\bf a}_{1}}I_{ n_{1} - m_{2}} &
\end{smallmatrix}
\end{bmatrix}\\
e^{\imath 2\pi {\bf k}\cdot{\bf a}_{2}}I_{n_{2} - m_{3}}\otimes\begin{bmatrix}
\begin{smallmatrix}
 & e^{-\imath 2\pi {\bf k}\cdot{\bf a}_{1}}I_{m_{1} + m_{2} - n_{1}}\\
I_{2n_{1}- m_{1} - m_{2}} & 
\end{smallmatrix}
\end{bmatrix} &
\end{smallmatrix}
\end{bmatrix}.
\end{align}
\end{itemize}

\item[5.] $0 > \cos\theta_{\gamma}$, $0 \leq \cos\theta_{\beta}$, $0 \leq \cos\theta_{\alpha} - \cos\theta_{\gamma}\cos\theta_{\beta}$.
\begin{itemize}
\item[(a)] $m_{1} \leq m_{2}$
\begin{align}
&J_{3}\notag\\
=& \begin{bmatrix}
\begin{smallmatrix}
 & e^{-\imath 2\pi {\bf k}\cdot({\bf a}_{1} + {\bf a}_{2})}I_{m_{3}}\otimes\begin{bmatrix}
\begin{smallmatrix}
 & e^{-\imath 2\pi {\bf k}\cdot{\bf a}_{1}}I_{m_{2} - m_{1}}\\
I_{n_{1} - m_{2} + m_{1}} &
\end{smallmatrix}
\end{bmatrix}\\
I_{n_{2} - m_{3}}\otimes\begin{bmatrix}
\begin{smallmatrix}
 & e^{-\imath 2\pi {\bf k}\cdot{\bf a}_{1}}I_{m_{2}}\\
I_{n_{1} - m_{2}} & 
\end{smallmatrix}
\end{bmatrix} &
\end{smallmatrix}
\end{bmatrix}.
\end{align}
\item[(b)] $m_{1} > m_{2}$
\begin{align}
&J_{3}\notag\\
=& \begin{bmatrix}
\begin{smallmatrix}
 & e^{-\imath 2\pi {\bf k}\cdot{\bf a}_{2}}I_{m_{3}}\otimes\begin{bmatrix}
\begin{smallmatrix}
 & e^{-\imath 2\pi {\bf k}\cdot{\bf a}_{1}}I_{n_{1} - m_{1}+m_{2}}\\
I_{ m_{1} - m_{2}} &
\end{smallmatrix}
\end{bmatrix}\\
I_{n_{2} - m_{3}}\otimes\begin{bmatrix}
\begin{smallmatrix}
 & e^{-\imath 2\pi {\bf k}\cdot{\bf a}_{1}}I_{m_{2}}\\
I_{n_{1}- m_{2}} & 
\end{smallmatrix}
\end{bmatrix} &
\end{smallmatrix}
\end{bmatrix}.
\end{align}
\end{itemize}

\item[6.] $0 > \cos\theta_{\gamma}$, $0 \leq \cos\theta_{\beta}$, $0 > \cos\theta_{\alpha} - \cos\theta_{\gamma}\cos\theta_{\beta}$.
\begin{itemize}
\item[(a)] $m_{1} + m_{2} \leq n_{1}$
\begin{align}
&J_{3}\notag\\
=& \begin{bmatrix}
\begin{smallmatrix}
 & I_{m_{3}}\otimes\begin{bmatrix}
\begin{smallmatrix}
 & e^{-\imath 2\pi {\bf k}\cdot{\bf a}_{1}}I_{m_{2}}\\
I_{n_{1} - m_{2}} &
\end{smallmatrix}
\end{bmatrix}\\
e^{\imath 2\pi {\bf k}\cdot{\bf a}_{2}}I_{n_{2} - m_{3}}\otimes\begin{bmatrix}
\begin{smallmatrix}
 & I_{m_{1} + m_{2}}\\
e^{\imath 2\pi {\bf k}\cdot{\bf a}_{1}}I_{n_{1} - m_{1} - m_{2}} & 
\end{smallmatrix}
\end{bmatrix} &
\end{smallmatrix}
\end{bmatrix}.
\end{align}
\item[(b)] $m_{1} + m_{1}> n_{1}$
\begin{align}
&J_{3}\notag\\
=& \begin{bmatrix}
\begin{smallmatrix}
 & I_{m_{3}}\otimes\begin{bmatrix}
\begin{smallmatrix}
 & e^{-\imath 2\pi {\bf k}\cdot{\bf a}_{1}}I_{m_{2}}\\
I_{ n_{1} - m_{2}} &
\end{smallmatrix}
\end{bmatrix}\\
e^{\imath 2\pi {\bf k}\cdot{\bf a}_{2}}I_{n_{2} - m_{3}}\otimes\begin{bmatrix}
\begin{smallmatrix}
 & e^{-\imath 2\pi {\bf k}\cdot{\bf a}_{1}}I_{m_{1} + m_{2} - n_{1}}\\
I_{2n_{1}- m_{1} - m_{2}} & 
\end{smallmatrix}
\end{bmatrix} &
\end{smallmatrix}
\end{bmatrix}.
\end{align}
\end{itemize}

\item[7.] $0 > \cos\theta_{\gamma}$, $0 > \cos\theta_{\beta}$, $0 \leq \cos\theta_{\alpha} - \cos\theta_{\gamma}\cos\theta_{\beta}$.
\begin{itemize}
\item[(a)] $m_{1} \leq m_{2}$
\begin{align}
&J_{3}\notag\\
=& \begin{bmatrix}
\begin{smallmatrix}
 & e^{-\imath 2\pi {\bf k}\cdot{\bf a}_{2}}I_{m_{3}}\otimes\begin{bmatrix}
\begin{smallmatrix}
 & e^{-\imath 2\pi {\bf k}\cdot{\bf a}_{1}}I_{m_{2} - m_{1}}\\
I_{n_{1} - m_{2} + m_{1}} &
\end{smallmatrix}
\end{bmatrix}\\
I_{n_{2} - m_{3}}\otimes\begin{bmatrix}
\begin{smallmatrix}
 & I_{m_{2}}\\
e^{\imath 2\pi {\bf k}\cdot{\bf a}_{1}}I_{n_{1} - m_{2}} & 
\end{smallmatrix}
\end{bmatrix} &
\end{smallmatrix}
\end{bmatrix}.
\end{align}
\item[(b)] $m_{1} > m_{2}$
\begin{align}
&J_{3}\notag\\
=& \begin{bmatrix}
\begin{smallmatrix}
 & e^{-\imath 2\pi {\bf k}\cdot{\bf a}_{2}}I_{m_{3}}\otimes\begin{bmatrix}
\begin{smallmatrix}
 & I_{n_{1} - m_{1}+m_{2}}\\
e^{\imath 2\pi {\bf k}\cdot{\bf a}_{1}}I_{ m_{1} - m_{2}} &
\end{smallmatrix}
\end{bmatrix}\\
I_{n_{2} - m_{3}}\otimes\begin{bmatrix}
\begin{smallmatrix}
 & I_{m_{2}}\\
e^{\imath 2\pi {\bf k}\cdot{\bf a}_{1}}I_{n_{1}- m_{2}} & 
\end{smallmatrix}
\end{bmatrix} &
\end{smallmatrix}
\end{bmatrix}.
\end{align}
\end{itemize}

\item[8.] $0 > \cos\theta_{\gamma}$, $0 > \cos\theta_{\beta}$, $0 > \cos\theta_{\alpha} - \cos\theta_{\gamma}\cos\theta_{\beta}$.
\begin{itemize}
\item[(a)] $m_{1} + m_{2} \leq n_{1}$
\begin{align}
&J_{3}\notag\\
=& \begin{bmatrix}
\begin{smallmatrix}
 & I_{m_{3}}\otimes\begin{bmatrix}
\begin{smallmatrix}
 & I_{m_{2}}\\
e^{\imath 2\pi {\bf k}\cdot{\bf a}_{1}}I_{n_{1} - m_{2}} &
\end{smallmatrix}
\end{bmatrix}\\
e^{\imath 2\pi {\bf k}\cdot({\bf a}_{1} + {\bf a}_{2})}I_{n_{2} - m_{3}}\otimes\begin{bmatrix}
\begin{smallmatrix}
 & I_{m_{1} + m_{2}}\\
e^{\imath 2\pi {\bf k}\cdot{\bf a}_{1}}I_{n_{1} - m_{1} - m_{2}} & 
\end{smallmatrix}
\end{bmatrix} &
\end{smallmatrix}
\end{bmatrix}.
\end{align}
\item[(b)] $m_{1} + m_{2}> n_{1}$
\begin{align}
&J_{3}\notag\\
=& \begin{bmatrix}
\begin{smallmatrix}
 & I_{m_{3}}\otimes\begin{bmatrix}
\begin{smallmatrix}
 & I_{m_{2}}\\
e^{\imath 2\pi {\bf k}\cdot{\bf a}_{1}}I_{ n_{1} - m_{2}} &
\end{smallmatrix}
\end{bmatrix}\\
e^{\imath 2\pi {\bf k}\cdot{\bf a}_{2}}I_{n_{2} - m_{3}}\otimes\begin{bmatrix}
\begin{smallmatrix}
 & I_{m_{1} + m_{2} - n_{1}}\\
e^{\imath 2\pi {\bf k}\cdot{\bf a}_{1}}I_{2n_{1}- m_{1} - m_{2}} & 
\end{smallmatrix}
\end{bmatrix} &
\end{smallmatrix}
\end{bmatrix}.
\end{align}
\end{itemize}

\end{itemize}
\end{subequations}

We can observe that all situations can be find in the triclinic lattice, so we just need to study the triclinic lattice. These matrices seem very complicated, thus we must try to simplify these matrices. Combine \eqref{eq:acuteJ2} with \eqref{eq:obtuseJ2}, we get a general formulation:
\begin{align}\label{eq:J2}
J_{2} = e^{\imath 2\pi\rho_{1}{\bf k}\cdot{\bf a}_{1}}\begin{bmatrix}
0 & e^{-\imath 2\pi{\bf k}\cdot{\bf a}_{1}}I_{m_{1}}\\
I_{(n_{1}-m_{1})} & 0
\end{bmatrix}
\end{align}
where $\rho_{1} = \lceil\frac{2\theta_{\gamma}}{\pi} - 1\rceil$.

Similarly, we denote $\rho_{2} = \lceil\frac{2\theta_{\beta}}{\pi} - 1\rceil$, then $\rho_{2} = 0$ if $0 < \theta_{\beta} \leq \frac{\pi}{2}$ and $\rho_{2} = 1$ if $\frac{\pi}{2} < \theta_{\beta} < \pi$. Next, we define $\rho _{3} = \lceil\cos\theta_{\gamma}\cos\theta_{\beta} - \cos\theta_{\alpha}\rceil$, then $\rho_{3} = 0$ if $\cos\theta_{\alpha} - \cos\theta_{\gamma}\cos\theta_{\beta} \geq 0$ and $\rho_{3} = 1$ if $\cos\theta_{\alpha} - \cos\theta_{\gamma}\cos\theta_{\beta} < 0$. Base on these three definitions, we may set
\begin{align}
\psi_{1} = \begin{cases}
1, &\mbox{if } (\rho_{2}n_{1} - m_{2}) - (\rho_{1}n_{1} - m_{1}) > 0, \;\rho_{3} = 0,\\
0, &\mbox{if } (\rho_{2}n_{1} - m_{2}) - (\rho_{1}n_{1} - m_{1}) \leq 0, \;\rho_{3} = 0;
\end{cases}\\
\psi_{2} = \begin{cases}
1, &\mbox{if } (\rho_{1}n_{1} - m_{1}) + (\rho_{2}n_{1} - m_{2}) > 0, \;\rho_{3} = 1,\\
0, &\mbox{if } (\rho_{1}n_{1} - m_{1}) + (\rho_{2}n_{1} - m_{2}) \leq 0, \;\rho_{3} = 1.
\end{cases}
\end{align}
Moreover, let $\rho_{4} = \lceil\cos\theta_{\gamma}\cos\theta_{\beta}(\cos\theta_{\gamma}\cos\theta_{\beta} - \cos\theta_{\alpha})\rceil$, $\rho_{5} = \lceil\cos\theta_{\beta}\rceil$. In order to simplfy \eqref{eq:J3case16}, let
\begin{align}\label{eq:J1}
J_{1} = e^{\imath 2\pi\rho_{2}{\bf k}\cdot{\bf a}_{1}}\begin{bmatrix}
 & e^{-\imath 2\pi {\bf k}\cdot{\bf a}_{1}}I_{m_{2}}\\
I_{n_{1} - m_{2}} & 
\end{bmatrix},
\end{align}
we get four general categories for $J_{3}$:
\begin{itemize}
\item[(1)] $0 \leq \cos\theta_{\alpha} - \cos\theta_{\gamma}\cos\theta_{\beta}$ and $m_{1} \leq m_{2}$
\begin{subequations}\label{eq:acuteJ3}
\begin{align}
e^{\imath 2\pi\rho_{3}{\bf k}\cdot{\bf a}_{2}}\begin{bmatrix}
\begin{smallmatrix}
 & e^{-\imath 2\pi {\bf k}\cdot({\bf a}_{2} + \rho_{1}\rho_{4}{\bf a}_{1} - \psi_{1}{\bf a}_{1})}I_{m_{3}}\otimes\begin{bmatrix}
\begin{smallmatrix}
 & e^{-\imath 2\pi {\bf k}\cdot{\bf a}_{1}}I_{m_{2} - m_{1}}\\
I_{n_{1} - m_{2} + m_{1}} &
\end{smallmatrix}
\end{bmatrix}\\
I_{n_{2} - m_{3}}\otimes J_{1} &
\end{smallmatrix}
\end{bmatrix}.
\end{align}
\item[(2)] $0 \leq \cos\theta_{\alpha} - \cos\theta_{\gamma}\cos\theta_{\beta}$ and $m_{1} > m_{2}$
\begin{align}
e^{\imath 2\pi\rho_{3}{\bf k}\cdot{\bf a}_{2}}\begin{bmatrix}
\begin{smallmatrix}
 & e^{-\imath 2\pi {\bf k}\cdot({\bf a}_{2} - \rho_{2}\rho_{4}{\bf a}_{1} - \psi_{1}{\bf a}_{1})}I_{m_{3}}\otimes\begin{bmatrix}
\begin{smallmatrix}
 & e^{-\imath 2\pi {\bf k}\cdot{\bf a}_{1}}I_{n_{1} - m_{1} + m_{2}}\\
I_{ m_{1} - m_{2}} &
\end{smallmatrix}
\end{bmatrix}\\
I_{n_{2} - m_{3}}\otimes J_{1} &
\end{smallmatrix}
\end{bmatrix}.
\end{align}
\end{subequations}
\item[(3)] $0 > \cos\theta_{\alpha} - \cos\theta_{\gamma}\cos\theta_{\beta}$ and $m_{1} + m_{2} \leq n_{1}$
\begin{subequations}\label{eq:obtuseJ3}
\begin{align}
e^{\imath 2\pi\rho_{3}{\bf k}\cdot{\bf a}_{2}}\begin{bmatrix}
\begin{smallmatrix}
 & e^{-\imath 2\pi{\bf k}\cdot{\bf a}_{2}}I_{m_{3}}\otimes J_{1}\\
e^{\imath 2\pi(\rho_{1}\rho_{4} + \psi_{2}){\bf k}\cdot{\bf a}_{1}}I_{n_{2} - m_{3}}\otimes\begin{bmatrix}
\begin{smallmatrix}
 & e^{-\imath 2\pi {\bf k}\cdot{\bf a}_{1}}I_{m_{1} + m_{2}}\\
I_{n_{1} - m_{1} - m_{2}} & 
\end{smallmatrix}
\end{bmatrix} &
\end{smallmatrix}
\end{bmatrix}.
\end{align}
\item[(4)] $0 > \cos\theta_{\alpha} - \cos\theta_{\gamma}\cos\theta_{\beta}$ and $m_{1} + m_{2} > n_{1}$
\begin{align}
e^{\imath 2\pi\rho_{3}{\bf k}\cdot{\bf a}_{2}}\begin{bmatrix}
\begin{smallmatrix}
 & e^{-\imath 2\pi{\bf k}\cdot{\bf a}_{2}}I_{m_{3}}\otimes J_{1}\\
e^{-\imath 2\pi(\rho_{5}\rho_{4}-\psi_{2}){\bf k}\cdot{\bf a}_{1}}I_{n_{2} - m_{3}}\otimes\begin{bmatrix}
\begin{smallmatrix}
 & e^{-\imath 2\pi {\bf k}\cdot{\bf a}_{1}}I_{m_{1} + m_{2} - n_{1}}\\
I_{2n_{1} - m_{1} - m_{2}} & 
\end{smallmatrix}
\end{bmatrix} &
\end{smallmatrix}
\end{bmatrix}.
\end{align}
\end{subequations}
\end{itemize}

Now we begin to derive the matrix representations of the discretization for equations
\begin{align*}
\begin{cases}
\partial_{y}E_{3} - \partial_{z}E_{2} &= \imath\omega\mu_{0}H_{1},\\
\partial_{z}E_{1} - \partial_{x}E_{3} &= \imath\omega\mu_{0}H_{2},\\
\partial_{x}E_{2} - \partial_{y}E_{1} &= \imath\omega\mu_{0}H_{3},
\end{cases}
\end{align*}
and
\begin{align*}
\begin{cases}
\partial_{y}H_{3} - \partial_{z}H_{2} &= -\imath\omega\varepsilon E_{1},\\
\partial_{z}H_{1} - \partial_{x}H_{3} &= -\imath\omega\varepsilon E_{2},\\
\partial_{x}H_{2} - \partial_{y}H_{1} &= -\imath\omega\varepsilon E_{3}.
\end{cases}
\end{align*}
We remark that the Yee grids in the electric field are located on the center of edge, and the Yee grids in the magnetic field are located on the center of edge.


\subsection{Matrix representation of $\nabla \times E = \imath\omega\mu_{0} H$}\mbox{ }

\noindent{\bf{Part I. Partial derivative with respect to $x$ for $E$}.}Using the Yee's scheme, the partial derivative $\partial_{x}E_{2}$ and $\partial_{x}E_{3}$ have the discretizations
\begin{align}\label{eq:partialx_simple}
\dfrac{E_{2}(i+1, \hat{j}, k) - E_{2}(i, \hat{j}, k)}{\delta_{x}} \;\; \mbox{and}\;\; \dfrac{E_{3}(i+1, j, \hat{k}) - E_{3}(i, j, \hat{k})}{\delta_{x}}
\end{align}
for $i = 0, 1, \cdots, n_{1} - 1$, $j = 0, 1, \cdots, n_{2} - 1$, $k = 0, 1, \cdots, n_{3} - 1$. 

\begin{theorem}[Periodicity along ${\bf a}_{1}$]\label{thm:period1}
By the definition of the lattice translation vectors ${\bf a}_{1}$ in \eqref{eq:transvec} and the quasi-periodic condition \eqref{eq:quasi_periodic}, we have
\begin{align*}
E_{2}(n_{1}, \hat{j}, k) = e^{\imath 2\pi {\bf k}\cdot{\bf a}_{1}}E_{2}(0, \hat{j}, k) \hspace{1em} \mbox{and} \hspace{1em} E_{3}(n_{1}, j, \hat{k}) = e^{\imath 2\pi{\bf k}\cdot{\bf a}_{1}}E_{3}(0, j, \hat{k}),
\end{align*}
for $j = 0, 1, \cdots, n_{2} - 1$ and $k = 0, 1, \cdots, n_{3} - 1$.
\end{theorem}

\begin{proof}
By the periodic condition \eqref{eq:quasi_periodic}, an arbitrary point
\begin{align*}
{\bf x}(n_{1}, \hat{j}, k) &= [n_{1}\delta_{x}, (j + \tfrac{1}{2})\delta_{y}, k\delta_{z}]^{\top} = {\bf a}_{1} + [0, (j + \tfrac{1}{2})\delta_{y}, k\delta_{z}]^{\top}\\
&= {\bf a}_{1} + {\bf x}(0, j + \tfrac{1}{2}, k).
\end{align*}
Hence, $E_{2}(n_{1}, \hat{j}, k) = e^{\imath 2\pi{\bf k}\cdot{\bf a}_{1}}E_{2}(0, \hat{j}, k)$, and $E_{3}(n_{1}, j, \hat{k}) = e^{\imath 2\pi{\bf k}\cdot{\bf a}_{1}}E_{3}(0, j, \hat{k})$ can be obtained by the same argument.
\end{proof}

Consequently, the explicit representation of the discrete partial derivative matrix $C_{1}$ is of the form
\begin{align*}
C_{1} = I_{n_{3}} \otimes I_{n_{2}} \otimes K_{1} \in \mathbb{C}^{n \times n}
\end{align*}
and
\begin{align*}
K_{1} = \dfrac{1}{\delta_{x}}\begin{bmatrix}
-1 & 1 & & & \\
 & -1 & 1 &  & \\
 & & \ddots & \ddots & \\
 & & & -1 & 1\\
e^{\imath 2\pi {\bf k}\cdot{\bf a}_{1}} & & & & -1 
\end{bmatrix} \in \mathbb{C}^{n_{1} \times n_{1}}.
\end{align*}
The matrix representations of \eqref{eq:partialx_simple} are $C_{1}{\bf e}_{2}$ and $C_{1}{\bf e}_{3}$, respectively.

\noindent{\bf{Part II. Partial derivative with respect to $y$ for $E$}.} Using the Yee's scheme, the partial derivative $\partial_{y}E_{1}$ and $\partial_{y}E_{3}$ have the discretizations
\begin{align}\label{eq:partialy}
\dfrac{E_{1}(\hat{i}, j+1, k) - E_{1}(\hat{i}, j, k)}{\delta_{y}} \;\; \mbox{and}\;\; \dfrac{E_{3}(i, j+1, \hat{k}) - E_{3}(i, j, \hat{k})}{\delta_{y}}
\end{align}
for $i = 0, 1, \cdots, n_{1} - 1$, $j = 0, 1, \cdots, n_{2} - 1$, $k = 0, 1, \cdots, n_{3} - 1$. 

\begin{theorem}[Periodicity along ${\bf a}_{1}$ and ${\bf a}_{2}$]\label{thm:period2}
By the definition of the lattice translation vectors ${\bf a}_{1}$ and ${\bf a}_{2}$ in \eqref{eq:transvec} and the quasi-periodic condition \eqref{eq:quasi_periodic}, we have
\begin{align}\label{eq:period2E1}
E_{1}(\hat{0}:\hat{n}_{1} - 1, n_{2}, k) = e^{\imath 2\pi {\bf k}\cdot{\bf a_{2}}}J_{2}E_{1}(\hat{0}:\hat{n}_{1} - 1, 0, k)
\end{align}
and
\begin{align}\label{eq:period2E3}
E_{3}(0:n_{1} - 1, n_{2}, \hat{k}) = e^{\imath 2\pi{\bf k}\cdot{\bf a}_{2}}J_{2}E_{3}(0:n_{1} - 1, 0, \hat{k}),
\end{align}
where
\[
J_{2} = e^{\imath 2\pi \rho_{1}{\bf k}\cdot{\bf a}_{1}}\begin{bmatrix}
0 & e^{-\imath 2\pi{\bf k}\cdot{\bf a}_{1}}I_{m_{1}}\\
I_{(n_{1}-m_{1})} & 0
\end{bmatrix} \in \mathbb{C}^{n_{1}\times n_{1}}.
\]
\end{theorem}

\begin{proof}
In the begining, we remark that 
\[
m_{1} = \begin{cases}
\tfrac{\|{\bf a}_{2}\|}{\delta_{x}}\cos\theta_{\gamma} &\mbox{if } \cos\theta_{\gamma} \geq 0\\
\tfrac{\|{\bf a}_{1}\| + \|{\bf a}_{2}\|\cos\theta_{\gamma}}{\delta_{x}} &\mbox{if } \cos\theta_{\gamma} < 0
\end{cases},
\]
\[
\rho_{1} = \begin{cases}
0 &\mbox{if } 0 < \theta_{\gamma} \leq \tfrac{\pi}{2}\\
1 &\mbox{if } \tfrac{\pi}{2} < \theta_{\gamma} < \pi
\end{cases},
\]
and
\[
{\bf a}_{2} = \left[\|{\bf a}_{2}\|\!\cos\theta_{\gamma}, l_{2}, 0\right]^{\top} = \left[\|{\bf a}_{2}\|\!\cos\theta_{\gamma}, n_{2}\delta_{y}, 0\right]^{\top}.
\]
According to the angle $\theta_{\gamma}$, we have two cases as following:\\
{\bf{CaseI:}} $0 < \theta_{\gamma} \leq \tfrac{\pi}{2}$ $(\rho_{1} = 0)$. By the definition in \eqref{eq:arbpt}, we have
\begin{align}\label{eq:a2addpt1}
\begin{split}
{\bf x}(i, n_{2}, \hat{k}) &= {\bf a}_{2} + [i\delta_{x} - \|{\bf a}_{2}\|\cos\theta_{\gamma}, 0, \hat{k}\delta_{z}]^{\top}\\
&= {\bf a}_{2} + [(i - m_{1})\delta_{x}, 0, \hat{k}\delta_{z}]^{\top}.
\end{split}
\end{align}
If $0 \leq i < m_{1}$, the vector in \eqref{eq:a2addpt1} can be rewritten as
\begin{align*}
{\bf x}(i, n_{2}, \hat{k}) &= {\bf a}_{2} + \rho_{1}{\bf a}_{1} - {\bf a}_{1}+ [(i + n_{1} - m_{1})\delta_{x}, 0, \hat{k}\delta_{z}]^{\top}\\
&= {\bf a}_{2} + \rho_{1}{\bf a}_{1} - {\bf a}_{1} + {\bf x}(i + n_{1} - m_{1}, 0, \hat{k});
\end{align*}
otherwise, we have
\[
{\bf x}(i, n_{2}, \hat{k}) = {\bf a}_{2} + \rho_{1}{\bf a}_{1} + {\bf x}(i - m_{1}, 0, \hat{k}).
\]
{\bf{Case II:}} $\tfrac{\pi}{2} < \theta_{\gamma} < \pi$ $(\rho_{1} = 1)$. We can rewrite any vectors as
\begin{align}\label{eq:a2addpt2}
\begin{split}
{\bf x}(i, n_{2}, \hat{k}) &= {\bf a}_{2} + [i\delta_{x} - \|{\bf a}_{2}\|\cos\theta_{\gamma}, 0, \hat{k}\delta_{z}]^{\top}\\
&= {\bf a}_{2} +  [(i + n_{1} - m_{1})\delta_{x}, 0, \hat{k}\delta_{z}]^{\top}.
\end{split}
\end{align}
If $0 \leq i < m_{1}$, the vector in \eqref{eq:a2addpt2} can be rewritten as
\begin{align*}
{\bf x}(i, n_{2}, \hat{k}) &= {\bf a}_{2} + \rho_{1}{\bf a}_{1} - {\bf a}_{1} + [(i + n_{1} - m_{1})\delta_{x}, 0, \hat{k}\delta_{z}]^{\top}\\
&= {\bf a}_{2} + \rho_{1}{\bf a}_{1} - {\bf a}_{1} + {\bf x}(i + n_{1} - m_{1}, 0, \hat{k});
\end{align*}
otherwise, we have
\begin{align*}
{\bf x}(i, n_{2}, \hat{k}) &= {\bf a}_{2} + \rho_{1}{\bf a}_{1} + [(i - m_{1})\delta_{x}, 0, \hat{k}\delta_{z}]^{\top}\\
&= {\bf a}_{2} + \rho_{1}{\bf a}_{1} + {\bf x}(i - m_{1}, 0, \hat{k}).
\end{align*}
In both cases above, the rewritten implies that
\begin{subequations}\label{eq:period2}
\begin{align}
E_{3}(i, n_{2}, \hat{k}) = \begin{cases}
e^{\imath 2\pi {\bf k}\cdot({\bf a}_{2} + \rho_{1}{\bf a}_{1} - {\bf a}_{1})}E_{3}(i + n_{1} - m_{1}, 0, \hat{k}), &\mbox{if } 0 \leq i < m_{1},\\
e^{\imath 2\pi {\bf k}\cdot({\bf a}_{2} + \rho_{1}{\bf a}_{1})}E_{3}(i - m_{1}, 0, \hat{k}), &\mbox{if } m_{1} \leq i \leq n_{1} - 1.
\end{cases}
\end{align}
Similarly,
\begin{align}
E_{1}(\hat{i}, n_{2}, k) = \begin{cases}
e^{\imath 2\pi {\bf k}\cdot({\bf a}_{2} + \rho_{1}{\bf a}_{1} - {\bf a}_{1})}E_{1}(\hat{i} + n_{1} - m_{1}, 0, k), &\mbox{if } 0 \leq i < m_{1},\\
e^{\imath 2\pi {\bf k}\cdot({\bf a}_{2} + \rho_{1}{\bf a}_{1})}E_{1}(\hat{i} - m_{1}, 0, k), &\mbox{if } m_{1} \leq i \leq n_{1} - 1.
\end{cases}
\end{align}
\end{subequations}
The results in \eqref{eq:period2E1} and \eqref{eq:period2E3} can then be obtained by substituting $i$ from $0$ to $n_{1} - 1$ into \eqref{eq:period2}.
\end{proof}

Consequently, the explicit representation of the discrete partial derivative matrix $C_{2}$ is of the form
\begin{align*}
C_{2} = I_{n_{3}} \otimes K_{2} \in \mathbb{C}^{n\times n}
\end{align*}
and $K_{2}$ is shown in \eqref{eq:K2}. The matrix representations of \eqref{eq:partialy} are $C_{2}{\bf e}_{1}$ and $C_{2}{\bf e}_{3}$, respectively.

\noindent{\bf{Part III. Partial derivative with respect to $z$ for $E$}.} Using the Yee's scheme, the partial derivative $\partial_{z}E_{1}$ and $\partial_{z}E_{3}$ have the discretizations
\begin{align}\label{eq:partialz}
\dfrac{E_{1}(\hat{i}, j, k+1) - E_{1}(\hat{i}, j, k)}{\delta_{z}} \;\; \mbox{and}\;\; \dfrac{E_{2}(i, \hat{j}, k+1) - E_{2}(i, \hat{j}, k)}{\delta_{z}}
\end{align}
for $i = 0, 1, \cdots, n_{1} - 1$, $j = 0, 1, \cdots, n_{2} - 1$, $k = 0, 1, \cdots, n_{3} - 1$.

\begin{theorem}[Periodicity along ${\bf a}_{1}$, ${\bf a}_{2}$, and ${\bf a}_{3}$]\label{thm:period3}
By the definition of the lattice translation vectors ${\bf a}_{1}$, ${\bf a}_{2}$, and ${\bf a}_{3}$ in \eqref{eq:transvec} and the quasi-periodic condition \eqref{eq:quasi_periodic}, we have
\begin{align}\label{eq:period3E1}
\textup{vec}(E_{1}(\hat{0}:\hat{n}_{1} - 1, 0:n_{2} -1, n_{3})) = e^{\imath 2 \pi {\bf k}\cdot{\bf a}_{3}}J_{3}\textup{vec}(E_{1}(\hat{0}:\hat{n}_{1} - 1, 0:n_{2} -1, 0))
\end{align}
and
\begin{align}\label{eq:period3E2}
\textup{vec}(E_{2}(0:n_{1} - 1, \hat{0}:\hat{n}_{2} -1, n_{3})) = e^{\imath 2 \pi {\bf k}\cdot{\bf a}_{3}}J_{3}\textup{vec}(E_{1}(0:n_{1} - 1, \hat{0}:\hat{n}_{2} -1, 0))
\end{align}
where $J_{3}$ are shown in \eqref{eq:acuteJ3} and \eqref{eq:obtuseJ3}.
\end{theorem}

\begin{proof}
Before we start to prove this theorem, let us remark some notations:
\begin{align*}
m_{2} = \begin{cases}
\tfrac{\|{\bf a}_{3}\|}{\delta_{x}}\cos\theta_{\beta} & \mbox{if } 0 < \theta_{\beta} \leq \tfrac{\pi}{2}\\
\tfrac{\|{\bf a}_{1}\| + \|{\bf a}_{3}\|\cos\theta_{\beta}}{\delta_{x}} & \mbox{if } \tfrac{\pi}{2} < \theta_{\beta} < \pi
\end{cases},
\end{align*}
\begin{align*}
m_{3} = \begin{cases}
\tfrac{\|{\bf a}_{3}\|}{\delta_{y}}\tfrac{\cos\theta_{\alpha} - \cos\theta_{\gamma}\cos\theta_{\beta}}{\sin\theta_{\gamma}} & \mbox{if } \cos\theta_{\alpha} - \cos\theta_{\gamma}\cos\theta_{\beta} \geq 0\\
\tfrac{\|{\bf a}_{2}\|\sin^{2}\theta_{\gamma} + \|{\bf a}_{3}\|(\cos\theta_{\alpha} - \cos\theta_{\gamma}\cos\theta_{\beta})}{\delta_{y}\sin\theta_{\gamma}} & \mbox{if } \cos\theta_{\alpha} - \cos\theta_{\gamma}\cos\theta_{\beta} < 0
\end{cases},
\end{align*}
and
\begin{align*}
{\bf a}_{3} = \begin{bmatrix}
\|{\bf a}_{3}\|\cos\theta_{\beta} & \|{\bf a}_{3}\|\tfrac{\cos\theta_{\alpha} - \cos\theta_{\gamma}\cos\theta_{\beta}}{\sin\theta_{\gamma}} & n_{3}\delta_{z}
\end{bmatrix}^{\top}.
\end{align*}

According to these notations, we have two cases:\\
\noindent{\bf{Case I:}} $\cos\theta_{\alpha} - \cos\theta_{\beta}\cos\theta_{\gamma} \geq 0$ $(\rho_{3} = 0)$\\
Any vector can be rewitten as
\begin{align}\label{eq:arbpt_case1}
\begin{split}
{\bf x}(i, \hat{j}, n_{3}) &= [i\delta_{x}, \hat{j}\delta_{y}, n_{3}\delta_{z}]^{\top}\\
&= {\bf a}_{3} + \left[i\delta_{x} - \|{\bf a}_{3}\|\cos\theta_{\beta}, (j + \tfrac{1}{2})\delta_{y} - \|{\bf a}_{3}\|\tfrac{\cos\theta_{\alpha} - \cos\theta_{\gamma}\cos\theta_{\beta}}{\sin\theta_{\gamma}}, 0\right]^{\top}\\
&= {\bf a}_{3} + \left[i\delta_{x} - \|{\bf a}_{3}\|\cos\theta_{\beta}, (j + \tfrac{1}{2} - m_{3})\delta_{y}, 0\right]^{\top}.
\end{split}
\end{align}
First, let's consider the situation $m_{1} \leq m_{2}$, there are four possibilities:
\begin{itemize}
\item[(i)] $\cos\theta_{\gamma} \geq 0$, $\cos\theta_{\beta} \geq 0$ $(\rho_{1} = 0, \rho_{2} = 0, \rho_{4} = 0, \psi_{1} = 0)$\\
If $0 \leq j < m_{3}$ and $0 \leq i < m_{2} - m_{1}$, then
\begin{small}
\begin{align*}
&{\bf x}(i, \hat{j}, n_{3})\\
=& {\bf a}_{3}\!\! -\!\! {\bf a}_{2}\!\! +\!\! \left[(i\!\! -\!\! m_{2}\!\! +\!\! m_{1})\delta_{x}, (j\!\! +\!\! \tfrac{1}{2}\!\! +\!\! n_{2}\!\! - m_{3})\delta_{y}, 0\right]^{\top}\\
=& {\bf a}_{3}\!\! +\!\! \rho_{3}{\bf a}_{2}\!\! -\!\! {\bf a}_{2}\!\! +\!\! \left[(i\!\! -\!\! m_{2}\!\! +\!\! m_{1})\delta_{x}, (j\!\! +\!\! \tfrac{1}{2}\!\! +\!\! n_{2}\!\! -\!\! m_{3})\delta_{y}, 0\right]^{\top}\\
=& {\bf a}_{3}\!\! +\!\! \rho_{3}{\bf a}_{2}\!\! -\!\! {\bf a}_{2}\!\! -\!\! {\bf a}_{1}\!\! +\!\! \left[(i\!\! +\!\! n_{1}\!\! -\!\! m_{2}\!\! +\!\! m_{1})\delta_{x}, (j\!\! +\!\! \tfrac{1}{2}\!\! +\!\! n_{2}\!\! -\!\! m_{3})\delta_{y}, 0\right]^{\top}\\
=& {\bf a}_{3}\!\! + \!\!\rho_{3}{\bf a}_{2}\!\! - \!\!{\bf a}_{2} \!\!- \!\!\rho_{1}\rho_{4}{\bf a}_{1}\!\! +\!\! \psi_{1}{\bf a}_{1}\!\! -\!\! {\bf a}_{1}\!\! +\!\! \left[(i\!\! +\!\! n_{1}\!\! -\!\! m_{2}\!\! +\!\! m_{1})\delta_{x}, (j\!\! +\!\! \tfrac{1}{2}\!\! +\!\! n_{2}\!\! -\!\! m_{3})\delta_{y}, 0\right]^{\top}\\
=&{\bf a}_{3}\!\! +\!\! \rho_{3}{\bf a}_{2}\!\! -\!\! {\bf a}_{2}\!\! -\!\! \rho_{1}\rho_{4}{\bf a}_{1}\!\! +\!\! \psi_{1}{\bf a}_{1}\!\! -\!\! {\bf a}_{1}\!\! +\!\!{\bf x}\!\left(i\!\! +\!\! n_{1}\!\! -\!\! m_{2}\!\! +\!\! m_{1}, j\!\! +\!\! \tfrac{1}{2}\!\! +\!\! n_{2}\!\! -\!\! m_{3}, 0\right);
\end{align*}
\end{small}
if $0 \leq j < m_{3}$ and $m_{2} - m_{1} \leq i < n_{1}$, then
\begin{small}
\begin{align*}
&{\bf x}(i, \hat{j}, n_{3})\\
=& {\bf a}_{3}\! +\! \rho_{3}{\bf a}_{2}\! -\! {\bf a}_{2}\! +\! \left[(i\! -\! m_{2}\! +\! m_{1})\delta_{x}, (j\! +\! \tfrac{1}{2}\! +\! n_{2}\! -\! m_{3})\delta_{y}, 0\right]^{\top}\\
=& {\bf a}_{3}\! +\! \rho_{3}{\bf a}_{2}\! -\! {\bf a}_{2}\! -\! \rho_{1}\rho_{4}{\bf a}_{1}\! +\! \psi_{1}{\bf a}_{1}\! +\! \left[(i\! -\! m_{2}\! +\! m_{1})\delta_{x}, (j\! +\! \tfrac{1}{2}\! +\! n_{2}\! -\! m_{3})\delta_{y}, 0\right]^{\top}\\
=& {\bf a}_{3}\! +\! \rho_{3}{\bf a}_{2}\! -\! {\bf a}_{2}\! -\! \rho_{1}\rho_{4}{\bf a}_{1}\! +\! \psi_{1}{\bf a}_{1}\! +\! {\bf x}\!\left(i\! -\! m_{2}\! +\! m_{1}, j\! +\! \tfrac{1}{2}\! +\! n_{2}\! -\! m_{3}, 0\right);
\end{align*}
\end{small}
if $m_{3} \leq j < n_{2}$ and $0 \leq i < m_{2}$, then
\begin{small}
\begin{align*}
{\bf x}(i, \hat{j}, n_{3}) &= {\bf a}_{3} - {\bf a}_{1} + \left[(i\! +\! n_{1}\! -\! m_{2})\delta_{x}, (j\! +\! \tfrac{1}{2}\! -\! m_{3})\delta_{y}, 0\right]^{\top}\\
&= {\bf a}_{3} + \rho_{3}{\bf a}_{2} + \rho_{2}{\bf a}_{1} - {\bf a}_{1} + \left[(i\! +\! n_{1}\! -\! m_{2})\delta_{x}, (j\! +\! \tfrac{1}{2}\! -\! m_{3})\delta_{y}, 0\right]^{\top}\\
&= {\bf a}_{3} + \rho_{3}{\bf a}_{2} + \rho_{2}{\bf a}_{1} - {\bf a}_{1} + {\bf x}\!\left(i\! +\! n_{1}\! -\! m_{2}, j\! +\! \tfrac{1}{2}\! -\! m_{3}, 0\right);
\end{align*}
\end{small}
otherwise, i.e. $m_{3} \leq j < n_{2}$ and $m_{2} \leq i < n_{1}$, then
\begin{align*}
{\bf x}(i, \hat{j}, n_{3}) = {\bf a}_{3} + \rho_{3}{\bf a}_{2} + \rho_{2}{\bf a}_{1} + {\bf x}\left(i - m_{2}, j + \tfrac{1}{2} - m_{3}, 0\right).
\end{align*}
\item[(ii)] $\cos\theta_{\gamma} \geq 0$, $\cos\theta_{\beta} < 0$ ($\rho_{1} = 0$, $\rho_{2} = 1$, $\rho_{4} = 1$, $\psi_{1} = 1$)\\
If $0 \leq j < m_{3}$ and $0 \leq i < m_{2} - m_{1}$, then
\begin{small}
\begin{align*}
&{\bf x}(i, \hat{j}, n_{3})\\
=& {\bf a}_{3}\!\! +\!\! \rho_{3}{\bf a}_{2}\!\! -\!\!{\bf a}_{2}\!\! +\!\! \left[(i\!\! +\!\! n_{1}\!\! -\!\! m_{2}\!\! +\!\! m_{1})\delta_{x}, (j\!\! +\!\! \tfrac{1}{2}\!\! +\!\! n_{2}\!\! -\!\! m_{3})\delta_{y}, 0\right]^{\top}\\
=& {\bf a}_{3}\!\! +\!\! \rho_{3}{\bf a}_{2}\!\! -\!\! {\bf a}_{2}\!\! -\!\! \rho_{1}\rho_{4}{\bf a}_{1}\!\! +\!\! \psi_{1}{\bf a}_{1}\!\! -\!\! {\bf a}_{1}\!\! +\!\! \left[(i\!\! +\!\! n_{1}\!\! -\!\! m_{2}\!\! +\!\! m_{1})\delta_{x}, (j\!\! +\!\! \tfrac{1}{2}\!\! +\!\! n_{2}\!\! -\!\! m_{3})\delta_{y}, 0\right]^{\top}\\
=& {\bf a}_{3}\!\! +\!\! \rho_{3}{\bf a}_{2}\!\! -\!\! {\bf a}_{2}\!\! -\!\! \rho_{1}\rho_{4}{\bf a}_{1}\!\! +\!\! \psi_{1}{\bf a}_{1}\!\! -\!\! {\bf a}_{1}\!\! + {\bf x}\!\left(i\!\! +\!\! n_{1}\!\! -\!\! m_{2}\!\! +\!\! m_{1}, j\!\! +\!\! \tfrac{1}{2}\!\! +\!\! n_{2}\!\! -\!\! m_{3}, 0\right);
\end{align*}
\end{small}
if $0 \leq j < m_{3}$ and $m_{2} - m_{1} \leq i < n_{1}$, then
\begin{small}
\begin{align*}
&{\bf x}(i, \hat{j}, n_{3})\\
=& {\bf a}_{3}\! +\! \rho_{3}{\bf a}_{2}\! -\! {\bf a}_{2}\! -\! \rho_{1}\rho_{4}{\bf a}_{1}\! +\! \psi_{1}{\bf a}_{1}\! +\! \left[(i\! -\! m_{2}\! +\! m_{1})\delta_{x}, (j\! +\! \tfrac{1}{2}\! +\! n_{2}\! -\! m_{3})\delta_{y}, 0\right]^{\top}\\
=& {\bf a}_{3}\! +\! \rho_{3}{\bf a}_{2}\! -\! {\bf a}_{2}\! -\! \rho_{1}\rho_{4}{\bf a}_{1}\! +\! \psi_{1}{\bf a}_{1}\! +\! {\bf x}\!\left(i\! -\! m_{2}\! +\! m_{1}, j\! +\! \tfrac{1}{2}\! +\! n_{2}\! -\! m_{3}, 0\right);
\end{align*} 
\end{small}
if $m_{3} \leq j < n_{2}$ and $0 \leq i < m_{2}$, then
\begin{align*}
{\bf x}(i, \hat{j}, n_{3}) &= {\bf a}_{3} + \left[(i\! +\! n_{1}\! -\! m_{2})\delta_{x}, (j\! +\! \tfrac{1}{2}\! -\! m_{3})\delta_{y}, 0\right]^{\top}\\
&= {\bf a}_{3} + \rho_{3}{\bf a}_{2} + \rho_{2}{\bf a}_{1} - {\bf a}_{1} + \left[(i\! +\! n_{1}\! -\! m_{2})\delta_{x}, (j\! +\! \tfrac{1}{2}\! -\! m_{3})\delta_{y}, 0\right]^{\top}\\
&= {\bf a}_{3} + \rho_{3}{\bf a}_{2} + \rho_{2}{\bf a}_{1} - {\bf a}_{1} + {\bf x}\!\left(i\! +\! n_{1}\! -\! m_{2}, j\! +\! \tfrac{1}{2}\! -\! m_{3}, 0\right);
\end{align*}
otherwise, i.e. $m_{3} \leq j < n_{2}$ and $m_{2} \leq i < n_{1}$, then
\begin{align*}
{\bf x}(i, \hat{j}, n_{3}) &= {\bf a}_{3} + \rho_{3}{\bf a}_{2} + {\bf a}_{1} + \left[(i\! -\! m_{2})\delta_{x}, (j\! +\! \tfrac{1}{2}\! -\! m_{3})\delta_{y}, 0\right]^{\top}\\
&= {\bf a}_{3} + \rho_{3}{\bf a}_{2} + \rho_{2}{\bf a}_{1} + {\bf x}\!\left(i\! -\! m_{2}, j\! +\! \tfrac{1}{2}\! -\! m_{3}, 0\right).
\end{align*}
\item[(iii)] $\cos\theta_{\gamma} < 0$, $\cos\theta_{\beta} \geq 0$ ($\rho_{1} = 1$, $\rho_{2} = 0$, $\rho_{4} = 1$, $\psi_{1} = 0$)\\
If $0 \leq j < m_{3}$ and $0 \leq i < m_{2} - m_{1}$, then
\begin{small}
\begin{align*}
&{\bf x}(i, \hat{j}, n_{3})\\
=& {\bf a}_{3}\! +\! \rho_{3}{\bf a}_{2}\! -\! {\bf a}_{2}\! +\! \left[(i\! -\! m_{2}\! +\! m_{1}\! -\! n_{1})\delta_{x}, (j\! +\! \tfrac{1}{2}\! +\! n_{2}\! -\! m_{3})\delta_{y}, 0\right]^{\top}\\
=& {\bf a}_{3}\! +\! \rho_{3}{\bf a}_{2}\! -\! {\bf a}_{2}\! -\! 2{\bf a}_{1}\! +\! \left[(i\! +\! n_{1}\! -\! m_{2}\! +\! m_{1})\delta_{x}, (j\! +\! \tfrac{1}{2}\! +\! n_{2}\! -\! m_{3})\delta_{y}, 0\right]^{\top}\\
=& {\bf a}_{3}\! +\! \rho_{3}{\bf a}_{2}\! -\! {\bf a}_{2}\! -\! \rho_{1}\rho_{4}{\bf a}_{1}\! +\! \psi_{1}{\bf a}_{1}\! -\! {\bf a}_{1}\! +\!\! \left[(i\!\! +\!\! n_{1}\!\! -\!\! m_{2}\!\! +\!\! m_{1})\delta_{x}, (j\!\! +\!\! \tfrac{1}{2}\!\! +\!\! n_{2}\!\! -\!\! m_{3})\delta_{y}, 0\right]^{\top}\\
=& {\bf a}_{3}\! +\! \rho_{3}{\bf a}_{2}\! -\! {\bf a}_{2}\! -\! \rho_{1}\rho_{4}{\bf a}_{1}\! +\! \psi_{1}{\bf a}_{1}\! -\! {\bf a}_{1}\! +\! {\bf x}\!\left(i\!\! +\!\! n_{1}\!\! -\!\! m_{2}\!\! +\!\! m_{1}, j\!\! +\!\! \tfrac{1}{2}\!\! +\!\! n_{2}\!\! -\!\! m_{3}, 0\right);
\end{align*}
\end{small}
if $0 \leq j < m_{3}$ and $m_{2} - m_{1} \leq i < n_{1}$, then
\begin{small}
\begin{align*}
&{\bf x}(i, \hat{j}, n_{3})\\
=& {\bf a}_{3}\! +\! \rho_{3}{\bf a}_{2}\! -\! {\bf a}_{2}\! +\! \left[(i\! -\! m_{2}\! +\! m_{1}\! -\! n_{1})\delta_{x}, (j\! +\! \tfrac{1}{2}\! +\! n_{2}\! -\! m_{3})\delta_{y}, 0\right]^{\top}\\
=& {\bf a}_{3}\! +\! \rho_{3}{\bf a}_{2}\! -\! {\bf a}_{2}\! -\! {\bf a}_{1}\! +\! \left[(i\! -\! m_{2}\! +\! m_{1}\! -\! n_{1})\delta_{x}, (j\! +\! \tfrac{1}{2}\! +\! n_{2}\! -\! m_{3})\delta_{y}, 0\right]^{\top}\\
=& {\bf a}_{3}\! +\! \rho_{3}{\bf a}_{2}\! -\! {\bf a}_{2}\! -\! \rho_{1}\rho_{4}{\bf a}_{1}\! +\! \psi_{1}{\bf a}_{1}\! +\! {\bf x}\!\left(i\! -\! m_{2}\! +\! m_{1}\! -\! n_{1}, j\! +\! \tfrac{1}{2}\! +\! n_{2}\! -\! m_{3}, 0\right);
\end{align*}
\end{small}
if $m_{3} \leq j < n_{2}$ and $0 \leq i < m_{2}$, then
\begin{align*}
{\bf x}(i, \hat{j}, n_{3}) &= {\bf a}_{3} - {\bf a}_{1} + \left[(i + n_{1} - m_{2})\delta_{x}, (j + \tfrac{1}{2} - m_{3})\delta_{y}, 0\right]^{\top}\\
&= {\bf a}_{3} + \rho_{3}{\b a}_{2} + \rho_{2}{\bf a}_{1} - {\bf a}_{1} + {\bf x}\!\left(i + n_{1} - m_{2}, j + \tfrac{1}{2} - m_{3}, 0\right);
\end{align*}
otherwise, i.e. $m_{3} \leq j < n_{2}$ and $m_{2} \leq i < n_{1}$, then
\begin{align*}
{\bf x}(i, \hat{j}, n_{3}) = {\bf a}_{3} + \rho_{3}{\bf a}_{2} + \rho_{2}{\bf a}_{1} + {\bf x}\!\left(i - m_{2}, j + \tfrac{1}{2} - m_{3}, 0\right).
\end{align*}
\item[(iv)] $\cos\theta_{\gamma} < 0$, $\cos\theta_{\beta} < 0$ ($\rho_{1} = 1$, $\rho_{2} = 1$, $\rho_{4} = 0$, $\psi_{1} =0$)\\
If $0 \leq j < m_{3}$ and $0 \leq i < m_{2} - m_{1}$, then
\begin{small}
\begin{align*}
{\bf x}(i, \hat{j}, n_{3}) &= {\bf a}_{3}\!\! -\!\! {\bf a}_{2}\!\! +\!\! \left[(i\!\! -\!\! m_{2}\!\! +\!\! m_{1})\delta_{x}, (j\!\! +\!\! \tfrac{1}{2}\!\! +\!\! n_{2}\!\! -\!\! m_{3})\delta_{y}, 0\right]^{\top}\\
&= {\bf a}_{3}\!\! +\!\! \rho_{3}{\bf a}_{2}\!\! -\!\! {\bf a}_{2}\!\! -\!\! {\bf a}_{1}\!\! + \left[(i\!\! +\!\! n_{1}\!\! -\!\! m_{2}\!\! +\!\! m_{1})\delta_{x}, (j\!\! +\!\! \tfrac{1}{2}\!\! +\!\! n_{2}\!\! -\!\! m_{3})\delta_{y}, 0\right]^{\top}\\
&= {\bf a}_{3}\!\! +\!\! \rho_{3}{\bf a}_{2}\!\! -\!\! {\bf a}_{2}\!\! -\!\! \rho_{1}\rho_{4}{\bf a}_{1}\!\! +\!\! \psi_{1}{\bf a}_{1}\!\! -\!\! {\bf a}_{1}\!\! +\!\! {\bf x}\!\left(i\!\! +\!\! n_{1}\!\! -\!\! m_{2}\!\! +\!\! m_{1}, j\!\! +\!\! \tfrac{1}{2}\!\! +\!\! n_{2}\!\! -\!\! m_{3}, 0\right);
\end{align*}
\end{small}
if $0 \leq j < m_{3}$ and $m_{2} - m_{1} \leq i < n_{1}$, then
\begin{small}
\begin{align*}
{\bf x}(i, \hat{j}, n_{3}) &= {\bf a}_{3}\! +\! \rho_{3}{\bf a}_{2}\! -\! {\bf a}_{2}\! +\! \left[(i\! -\! m_{2}\! +\! m_{1})\delta_{x}, (j\! +\! \tfrac{1}{2}\! +\! n_{2}\! -\! m_{3})\delta_{y}, 0\right]^{\top}\\
&= {\bf a}_{3}\! +\! \rho_{3}{\bf a}_{2}\! -\! {\bf a}_{2}\! -\! \rho_{1}\rho_{4}{\bf a}_{1}\! +\! \psi_{1}{\bf a}_{1}\! +\! {\bf x}\!\left(i\! -\! m_{2}\! +\! m_{1}, j\! +\! \tfrac{1}{2}\! +\! n_{2}\! -\! m_{3}, 0\right);
\end{align*}
\end{small}
if $m_{3} \leq j < n_{2}$ and $0 \leq i < m_{2}$, then
\begin{align*}
{\bf x}(i, \hat{j}, n_{3}) &= {\bf a}_{3} + \left[(i + n_{1} - m_{2})\delta_{x}, (j + \tfrac{1}{2} - m_{3})\delta_{y}, 0\right]^{\top}\\
&= {\bf a}_{3} + \rho_{3}{\bf a}_{2} + \rho_{2}{\bf a}_{1} - {\bf a}_{1} + {\bf x}\!\left(i + n_{1} - m_{2}, j + \tfrac{1}{2} - m_{3}, 0\right);
\end{align*}
otherwise, i.e. $m_{3} \leq j < n_{2}$ and $m_{2} \leq i < n_{1}$, then
\begin{align*}
{\bf x}(i, \hat{j}, n_{3}) &= {\bf a}_{3} + {\bf a}_{1} + \left[(i - m_{2})\delta_{x}, (j + \tfrac{1}{2} - m_{3})\delta_{y}, 0\right]^{\top}\\
&= {\bf a}_{3} + \rho_{3}{\bf a}_{2} + \rho_{2}{\bf a}_{1} + {\bf x}\!\left(i - m_{2}, j + \tfrac{1}{2} - m_{3}, 0\right).
\end{align*}
\end{itemize}
The above four rewritten imply that
\begin{small}
\begin{align*}
E_{2}(i, \hat{j}, n_{3}) = \begin{cases}
e^{\imath 2\pi{\bf k}\cdot\left[{\bf a}_{3} + \rho_{3}{\bf a}_{2} - ({\bf a}_{2} + \rho_{1}\rho_{4}{\bf a}_{1} - \psi_{1}{\bf a}_{1}) - {\bf a}_{1}\right]}E_{2}(i\! +\! n_{1}\! -\! m_{2}\! +\! m_{1}, \hat{j}\! +\! n_{2}\! -\! m_{3}, 0)\\
\hspace{6cm}\mbox{if }  0 \leq i < m_{2}\! -\! m_{1}, 0 \leq j < m_{3};\\
e^{\imath 2\pi{\bf k}\cdot\left[{\bf a}_{3} + \rho_{3}{\bf a}_{2} - ({\bf a}_{2} + \rho_{1}\rho_{4}{\bf a}_{1} - \psi_{1}{\bf a}_{1})\right]}E_{2}(i\! -\! m_{2}\! +\! m_{1}, \hat{j}\! +\! n_{2}\! -\! m_{3}, 0)\\
\hspace{6cm}\mbox{if }  m_{2}\! -\! m_{1} \leq i < n_{1}, 0 \leq j < m_{3};\\
e^{\imath 2\pi{\bf k}\cdot\left[{\bf a}_{3} + \rho_{3}{\bf a}_{2} + \rho_{2}{\bf a}_{1} - {\bf a}_{1}\right]}E_{2}(i + n_{1} - m_{2}, \hat{j} - m_{3}, 0)\\
\hspace{6.7cm}\mbox{if } 0 \leq i < m_{2}, m_{3} \leq j < n_{2};\\
e^{\imath 2\pi{\bf k}\cdot\left[{\bf a}_{3} + \rho_{3}{\bf a}_{2} + \rho_{2}{\bf a}_{1}\right]}E_{2}(i - m_{2}, \hat{j} - m_{3}, 0)\\
\hspace{6.7cm} \mbox{if }  m_{2} \leq i < n_{1}, m_{3} \leq j < n_{2}.
\end{cases}
\end{align*}
\end{small}

Second, we consider the situation $m_{1} > m_{2}$, there are alos four possibilities:
\begin{itemize}
\item[(i)] $\cos\theta_{\gamma} \geq 0$, $\cos\theta_{\beta} \geq 0$ ($\rho_{1} = 0$, $\rho_{2} = 0$, $\rho_{4}$ = 0, $\psi_{1} = 1$)\\
If $0 \leq j < m_{3}$ and $0 \leq i < n_{1} - m_{1} + m_{2}$, then
\begin{small}
\begin{align*}
{\bf x}(i, \hat{j}, n_{3}) &= {\bf a}_{3}\!\! -\!\! {\bf a}_{2}\!\! +\!\! \left[(i\!\! +\!\! m_{1}\!\! -\!\! m_{2})\delta_{x}, (j\!\! +\!\! \tfrac{1}{2}\!\! +\!\! n_{2}\!\! -\!\! m_{3})\delta_{y}, 0\right]^{\top}\\
&= {\bf a}_{3}\!\! +\!\! \rho_{3}{\bf a}_{2}\!\! -\!\! {\bf a}_{2}\!\! +\!\! \rho_{2}\rho_{4}{\bf a}_{1}\!\! +\!\! \psi_{1}{\bf a}_{1}\!\! -\!\! {\bf a}_{1}\!\! +\!\! {\bf x}\!\left(i\!\! +\!\! m_{1}\!\! -\!\! m_{2}, j\!\! +\!\! \tfrac{1}{2}\!\! +\!\! n_{2}\!\! - m_{3}, 0\right);
\end{align*}
\end{small}
if $0 \leq j < m_{3}$ and $n_{1} - m_{1} + m_{2} \leq i < n_{1}$, then
\begin{small}
\begin{align*}
{\bf x}(i, \hat{j}, n_{3}) &= {\bf a}_{3}\!\! +\!\! \rho_{3}{\bf a}_{2}\!\! -\!\! {\bf a}_{2}\!\! +\!\! \left[(i\!\! +\!\! m_{1}\!\! -\!\! m_{2})\delta_{x}, (j\!\! +\!\! \tfrac{1}{2}\!\! +\!\! n_{2}\!\! -\!\! m_{3})\delta_{y}, 0\right]^{\top}\\
&= {\bf a}_{3}\!\! +\!\! \rho_{3}{\bf a}_{2}\!\! -\!\! {\bf a}_{2}\!\! +\!\! {\bf a}_{1}\!\! +\!\! \left[(i\!\! -\!\! n_{1}\!\! +\!\! m_{1}\!\! -\!\! m_{2})\delta_{x}, (j\!\! +\!\! \tfrac{1}{2}\!\! +\!\! n_{2}\!\! -\!\! m_{3})\delta_{y}, 0\right]^{\top}\\
&= {\bf a}_{3}\!\! +\!\! \rho_{3}{\bf a}_{2}\!\! -\!\! {\bf a}_{2}\!\! +\!\! \rho_{2}\rho_{4}{\bf a}_{1}\!\! +\!\! \psi_{1}{\bf a}_{1}\!\! +\!\! {\bf x}\!\left(i\!\! -\!\! n_{1}\!\! +\!\! m_{1}\!\! -\!\! m_{2}, j\!\! +\!\! \tfrac{1}{2}\!\! +\!\! n_{2}\!\! -\!\! m_{3}, 0\right);
\end{align*}
\end{small}
if $m_{3} \leq j < n_{2}$ and $0 \leq i < m_{2}$, then
\begin{align*}
{\bf x}(i, \hat{j}, n_{3}) &= {\bf a}_{3} - {\bf a}_{1} + \left[(i\! +\! n_{1}\! -\! m_{2})\delta_{x}, (j\! +\! \tfrac{1}{2}\! -\! m_{3})\delta_{y}, 0\right]^{\top}\\
&= {\bf a}_{3} + \rho_{3}{\bf a}_{2} + \rho_{2}{\bf a}_{1} - {\bf a}_{1} + {\bf x}\!\left(i\! +\! n_{1}\! -\! m_{2}, j\! +\! \tfrac{1}{2}\! -\! m_{3}, 0\right);
\end{align*}
otherwise, i.e. $m_{3} \leq j < n_{2}$ and $m_{2} \leq i < n_{1}$, then
\begin{align*}
{\bf x}(i, \hat{j}, n_{3}) = {\bf a}_{3} + \rho_{3}{\bf a}_{2} + \rho_{2}{\bf a}_{1} + {\bf x}\!(i - m_{2}, j + \tfrac{1}{2} - m_{3}, 0).
\end{align*}
\item[(ii)] $\cos\theta_{\gamma} \geq 0$, $\cos\theta_{\beta} < 0$ ($\rho_{1} = 0$, $\rho_{2} = 1$, $\rho_{4} = 1$, $\psi_{1} = 1$)\\
If $0 \leq j < m_{3}$ and $0 \leq i < n_{1} - m_{1} + m_{2}$, then
\begin{small}
\begin{align*}
{\bf x}(i, \hat{j}, n_{3}) &= {\bf a}_{3}\!\! -\!\! {\bf a}_{2}\!\! +\!\! \left[(i\!\! +\!\! n_{1}\!\! +\!\! m_{1}\!\! -\!\! m_{2})\delta_{x}, (j\!\! +\!\! \tfrac{1}{2}\!\! +\!\! n_{2}\!\! -\!\! m_{3})\delta_{y}, 0\right]^{\top}\\
&= {\bf a}_{3}\!\! -\!\! {\bf a}_{2}\!\! +\!\! {\bf a}_{1}\!\! +\!\! \left[(i\!\! +\!\! m_{1}\!\! -\!\! m_{2})\delta_{x}, (j\!\! +\!\! \tfrac{1}{2}\!\! +\!\! n_{2}\!\! -\!\! m_{3})\delta_{y}, 0\right]^{\top}\\
&= {\bf a}_{3}\!\! +\!\! \rho_{3}{\bf a}_{2}\!\! -\!\! {\bf a}_{2}\!\! +\!\! \rho_{2}\rho_{4}{\bf a}_{1}\!\!  +\!\! \psi_{1}{\bf a}_{1}\!\! -\!\! {\bf a}_{1}\!\! +\!\! {\bf x}\!\left(i\!\! +\!\! m_{1}\!\! -\!\! m_{2}, j\!\! +\!\! \tfrac{1}{2}\!\! +\!\! n_{2}\!\! -\!\! m_{3}, 0\right);
\end{align*}
\end{small}
if $0 \leq j < m_{3}$ and $n_{1} - m_{1} + m_{2} \leq i < n_{1}$, then
\begin{small}
\begin{align*}
{\bf x}(i, \hat{j}, n_{3}) &= {\bf a}_{3}\!\! -\!\! {\bf a}_{2}\!\! + 2{\bf a}_{1}\!\! +\!\! \left[(i\!\! -\!\! n_{1}\!\! +\!\! m_{1}\!\! -\!\! m_{2})\delta_{x}, (j\!\! +\!\! \tfrac{1}{2}\!\! +\!\! n_{2}\!\! -\!\! m_{3})\delta_{y}, 0\right]^{\top}\\
&= {\bf a}_{3}\!\! +\!\! \rho_{3}{\bf a}_{2}\!\! -\!\! {\bf a}_{2}\!\! +\!\! \rho_{2}\rho_{4}{\bf a}_{1}\!\!  +\!\! \psi_{1}{\bf a}_{1}\!\! +\!\! {\bf x}\!\left(i\!\! -\!\! n_{1}\!\! +\!\! m_{1}\!\! -\!\! m_{2}, j\!\! +\!\! \tfrac{1}{2}\!\! +\!\! n_{2}\!\! -\!\! m_{3}, 0\right);
\end{align*}
\end{small}
if $m_{3} \leq j < n_{2}$ and $0 \leq i < m_{2}$, then
\begin{align*}
{\bf x}(i, \hat{j}, n_{3}) &= {\bf a}_{3} + \left[(i\! +\! n_{1}\! -\! m_{2})\delta_{x}, (j\! +\! \tfrac{1}{2}\! -\! m_{3})\delta_{y}, 0\right]^{\top}\\
&= {\bf a}_{3} + \rho_{3}{\bf a}_{2}  + \rho_{2}{\bf a}_{1}  - {\bf a}_{1} + {\bf x}\!\left(i\! +\! n_{1}\! -\! m_{2}, j\! +\! \tfrac{1}{2}\! - m_{3}, 0\right);
\end{align*}
otherwise, i.e. $m_{3} \leq j < n_{2}$ and $m_{2} \leq i < n_{1}$, then
\begin{align*}
{\bf x}(i, \hat{j}, n_{3}) &= {\bf a}_{3} + {\bf a}_{1} + \left[(i\! -\! m_{2})\delta_{x}, (j\! +\! \tfrac{1}{2}\! -\! m_{3})\delta_{y}, 0\right]^{\top}\\
&= {\bf a}_{3} + \rho_{3}{\bf a}_{2} + \rho_{2}{\bf a}_{1} + {\bf x}\!\left(i\! -\! m_{2}, j\! +\! \tfrac{1}{2}\! -\! m_{3}, 0\right).
\end{align*}
\item[(iii)] $\cos\theta_{\gamma} < 0$, $\cos\theta_{\beta} \geq 0$ ($\rho_{1} = 1$, $\rho_{2} = 0$, $\rho_{4} = 1$, $\psi_{1} = 0$)\\
If $0 \leq j < m_{3}$ and $0 \leq i < n_{1} - m_{1} + m_{2}$, then
\begin{small}
\begin{align*}
{\bf x}(i, \hat{j}, n_{3}) &= {\bf a}_{3}\!\! -\!\! {\bf a}_{2}\!\! +\!\! \left[(i\!\! -\!\! m_{2}\!\! +\!\! m_{1}\!\! -\!\! n_{1})\delta_{x}, (j\!\! +\!\! \tfrac{1}{2}\!\! +\!\! n_{2}\!\! -\!\! m_{3})\delta_{y}, 0\right]^{\top}\\
&= {\bf a}_{3}\!\! -\!\! {\bf a}_{2}\!\! -\!\! {\bf a}_{1}\!\! +\!\! \left[(i\!\! -\!\! m_{2}\!\! +\!\! m_{1})\delta_{x}, (j\!\! +\!\! \tfrac{1}{2}\!\! +\!\! n_{2}\!\! -\!\! m_{3})\delta_{y}, 0\right]^{\top}\\
&= {\bf a}_{3}\!\! +\!\! \rho_{3}{\bf a}_{2}\!\! -\!\! {\bf a}_{2}\!\! +\!\! \rho_{2}\rho_{4}{\bf a}_{1}\!\! +\!\! \psi_{1}{\bf a}_{1}\!\! -\!\! {\bf a}_{1}\!\! +\!\! {\bf x}\!\left(i\!\! -\!\! m_{2}\!\! +\!\! m_{1}, j\!\! +\!\! \tfrac{1}{2}\!\! +\!\! n_{2}\!\! -\!\! m_{3}, 0\right);
\end{align*}
\end{small}
if $0 \leq j < m_{3}$ and $n_{1} - m_{1} + m_{2} \leq i < n_{1}$, then
\begin{small}
\begin{align*}
{\bf x}(i, \hat{j}, n_{3}) &= {\bf a}_{3}\!\! -\!\! {\bf a}_{2}\!\! +\!\! \left[(i\!\! -\!\! m_{2}\!\! +\!\! m_{1}\!\! -\!\! n_{1})\delta_{x}, (j\!\! +\!\! \tfrac{1}{2}\!\! +\!\! n_{2}\!\! -\!\! m_{3})\delta_{y}, 0\right]^{\top}\\
&= {\bf a}_{3}\!\! +\!\! \rho_{3}{\bf a}_{2}\!\! -\!\! {\bf a}_{2}\!\! +\!\! \rho_{2}\rho_{4}{\bf a}_{1}\!\! +\!\! \psi_{1}{\bf a}_{1}\!\! +\!\! {\bf x}\!\left(i\!\! -\!\! n_{1}\!\! +\!\! m_{1}\!\! -\!\! m_{2}, j\!\! +\!\! \tfrac{1}{2}\!\! +\!\! n_{2}\!\! -\!\! m_{3}, 0\right);
\end{align*}
\end{small}
if $m_{3} \leq j < n_{2}$ and $0 \leq i < m_{2}$, then
\begin{align*}
{\bf x}(i, \hat{j}, n_{3}) &= {\bf a}_{3}\! -\! {\bf a}_{1}\! +\! \left[(i\! +\! n_{1}\! -\! m_{2})\delta_{x}, (j\! +\! \tfrac{1}{2}\! -\! m_{3})\delta_{y}, 0\right]^{\top}\\
&= {\bf a}_{3}\! +\! \rho_{3}{\bf a}_{2}\! +\! \rho_{2}{\bf a}_{1}\! -\! {\bf a}_{1}\! +\! {\bf x}\!\left(i\! +\! n_{1}\! -\! m_{2}, j\! +\! \tfrac{1}{2}\! -\! m_{3}, 0\right);
\end{align*}
otherwise, i.e. $m_{3} \leq j < n_{2}$ and $m_{2} \leq i < n_{1}$, then
\begin{align*}
{\bf x}(i, \hat{j}, n_{3}) &= {\bf a}_{3} + \left[(i\! -\! m_{2})\delta_{x}, (j\! +\! \tfrac{1}{2}\! -\! m_{3})\delta_{y}, 0\right]^{\top}\\
&= {\bf a}_{3} + \rho_{3}{\bf a}_{2} + \rho_{2}{\bf a}_{1} + {\bf x}\!\left(i\! -\! m_{2}, j\! +\! \tfrac{1}{2}\! -\! m_{3}, 0\right).
\end{align*}
\item[(iv)] $\cos\theta_{\gamma} < 0$, $\cos\theta_{\beta} < 0$ ($\rho_{1} = 1$, $\rho_{2} = 1$, $\rho_{4} = 0$, $\psi_{1} = 1$)\\
If $0 \leq j < m_{3}$ and $0 \leq i < n_{1} - m_{1} + m_{2}$, then
\begin{small}
\begin{align*}
{\bf x}(i, \hat{j}, n_{3}) &= {\bf a}_{3}\!\! -\!\! {\bf a}_{2}\!\! +\!\! \left[(i\!\! -\!\! (m_{2}\!\! -\!\! n_{1})\!\! +\!\! (m_{1}\!\! -\!\! n_{1}))\delta_{x}, (j\!\! +\!\! \tfrac{1}{2}\!\! +\!\! n_{2}\!\! -\!\! m_{3})\delta_{y}, 0\right]^{\top}\\
&= {\bf a}_{3}\!\! +\!\! \rho_{3}{\bf a}_{2}\!\! -\!\! {\bf a}_{2}\!\! +\!\! \rho_{2}\rho_{4}{\bf a}_{1}\!\! +\!\! \psi_{1}{\bf a}_{1}\!\! -\!\! {\bf a}_{1}\!\! +\!\! {\bf x}\!\left(i\!\! +\!\! m_{1}\!\! -\!\! m_{2}, j\!\! +\!\! \tfrac{1}{2}\!\! +\!\! n_{2}\!\! -\!\! m_{3}, 0\right);
\end{align*}
\end{small}
if $0 \leq j < m_{3}$ and $n_{1} - m_{1} + m_{2} \leq i < n_{1}$, then
\begin{small}
\begin{align*}
{\bf x}(i, \hat{j}, n_{3}) &= {\bf a}_{3}\!\! -\!\! {\bf a}_{2}\!\! +\!\! {\bf a}_{1}\!\! +\!\! \left[(i\!\! -\!\! n_{1}\!\! +\!\! m_{1}\!\! -\!\! m_{2})\delta_{x}, (j\!\! +\!\! \tfrac{1}{2}\!\! +\!\! n_{2}\!\! -\!\! m_{3})\delta_{y}, 0\right]^{\top}\\
&= {\bf a}_{3}\!\! +\!\! \rho_{3}{\bf a}_{2}\!\! -\!\! {\bf a}_{2}\!\! +\!\! \rho_{2}\rho_{4}{\bf a}_{1}\!\! +\!\! \psi_{1}{\bf a}_{1}\!\! +\!\! {\bf x}\!\left(i\!\! -\!\! n_{1}\!\! +\!\! m_{1}\!\! -\!\! m_{2}, j\!\! +\!\! \tfrac{1}{2}\!\! +\!\! n_{2}\!\! -\!\! m_{3}, 0\right);
\end{align*}
\end{small}
if $m_{3} \leq j < n_{2}$ and $0 \leq i < m_{2}$, then
\begin{align*}
{\bf x}(i, \hat{j}, n_{3}) &= {\bf a}_{3}\! +\! \left[(i\! -\! m_{2}\! +\! n_{1})\delta_{x}, (j\! +\! \tfrac{1}{2}\! -\! m_{3})\delta_{y}, 0\right]^{\top}\\
&= {\bf a}_{3}\! +\! \rho_{3}{\bf a}_{2}\! +\! \rho_{2}{\bf a}_{1}\! -\! {\bf a}_{1}\! +\! {\bf x}\!\left(i\! +\! n_{1}\! -\! m_{2}, j\! +\! \tfrac{1}{2}\! -\! m_{3}, 0\right);
\end{align*}
otherwise, i.e. $m_{3} \leq j < n_{2}$ and $m_{2} \leq i < n_{1}$, then
\begin{align*}
{\bf x}(i, \hat{j}, n_{3}) &= {\bf a}_{3}\! +\! {\bf a}_{1}\! +\! \left[(i\! -\! m_{2})\delta_{x}, (j\! +\! \tfrac{1}{2}\! -\! m_{3})\delta_{y}, 0\right]^{\top}\\
&= {\bf a}_{3}\! + \rho_{3}{\bf a}_{2}\! +\! \rho_{2}{\bf a}_{1}\! +\! {\bf x}\!\left(i\! -\! m_{2}, j\! +\! \tfrac{1}{2}\! -\! m_{3}, 0\right).
\end{align*}
\end{itemize}
The above four rewritten imply that
\begin{small}
\begin{align*}
E_{2}(i, \hat{j}, n_{3}) = \begin{cases}
e^{\imath 2\pi{\bf k}\cdot\left[{\bf a}_{3} + \rho_{3}{\bf a}_{2} - ({\bf a}_{2} - \rho_{2}\rho_{4}{\bf a}_{1} - \psi_{1}{\bf a}_{1}) - {\bf a}_{1}\right]}E_{2}(i\! +\! m_{1}\! -\! m_{2}, \hat{j}\! +\! n_{2}\! -\! m_{3}, 0)\\
\hspace{4.8cm} \mbox{if } 0 \leq i < n_{1}\! -\! m_{1}\! +\! m_{2}, 0 \leq j < m_{3};\\
e^{\imath 2\pi{\bf k}\cdot\left[{\bf a}_{3} + \rho_{3}{\bf a}_{2} - ({\bf a}_{2} - \rho_{2}\rho_{4}{\bf a}_{1} - \psi_{1}{\bf a}_{1})\right]}E_{2}(i\! -\! n_{1}\! +\! m_{1}\! -\! m_{2}, \hat{j}\! +\! n_{2}\! -\! m_{3}, 0)\\
\hspace{4.8cm} \mbox{if }  n_{1}\! -\! m_{1}\! +\! m_{2} \leq i < n_{1}, 0 \leq j < m_{3};\\
e^{\imath 2\pi{\bf k}\cdot\left[{\bf a}_{3} + \rho_{3}{\bf a}_{2} + \rho_{2}{\bf a}_{1} - {\bf a}_{1}\right]}E_{2}(i + n_{1} - m_{2}, \hat{j} - m_{3}, 0)\\
\hspace{5.5cm} \mbox{if } 0 \leq i < m_{2}, m_{3} \leq j < n_{2};\\
e^{\imath 2\pi{\bf k}\cdot\left[{\bf a}_{3} + \rho_{3}{\bf a}_{2} + \rho_{2}{\bf a}_{1}\right]}E_{2}(i - m_{2}, \hat{j} - m_{3}, 0)\\
\hspace{5.5cm} \mbox{if } m_{2} \leq i < n_{1}, m_{3} \leq j < n_{2}.
\end{cases}
\end{align*}
\end{small}

\noindent{\bf{Case II:}} $\cos\theta_{\alpha} - \cos\theta_{\gamma}\cos\theta_{\beta} < 0$ $(\rho_{3} = 1)$\\
Different from the above case, any vector can be rewitten as
\begin{align}\label{eq:arbpt_case2}
{\bf x}(i, \hat{j}, n_{3}) = {\bf a}_{3} + \left[i\delta_{x} - \|{\bf a}_{3}\|\!\cos\theta_{\beta}, (j + \tfrac{1}{2} + n_{2} - m_{3})\delta_{y}, 0\right]^{\top}.
\end{align}
First, let us consider $m_{1} + m_{2} \leq n_{1}$. In this situation, if we want to rewrite \eqref{eq:arbpt_case2}, there are four possibilities:
\begin{itemize}
\item[(i)] $\cos\theta_{\gamma} \geq 0$, $\cos\theta_{\beta} \geq 0$ ($\rho_{1} = 0$, $\rho_{2} = 0$, $\rho_{4} = 1$, $\psi_{2} = 0$)\\
If $0 \leq j < m_{3}$ and $0 \leq i < m_{2}$, then
\begin{align*}
{\bf x}(i, \hat{j}, n_{3}) &= {\bf a}_{3}\! -\! {\bf a}_{1}\! +\! \left[(i\! +\! n_{1}\! -\! m_{2})\delta_{x}, (j\! +\! \tfrac{1}{2}\! +\! n_{2}\! -\! m_{3})\delta_{y}, 0\right]^{\top}\\
&= {\bf a}_{3}\! +\! \rho_{3}{\bf a}_{2}\! -\! {\bf a}_{2}\! +\! \rho_{2}{\bf a}_{1}\! -\! {\bf a}_{1}\! +\! {\bf x}\!\left(i\! +\! n_{1}\! -\! m_{2}, j\! +\! \tfrac{1}{2}\! +\! n_{2}\! -\! m_{3}, 0\right);
\end{align*}
if $0 \leq j < m_{3}$ and $m_{2} \leq i < n_{1}$, then
\begin{align*}
{\bf x}(i, \hat{j}, n_{3}) &= {\bf a}_{3}\! +\! \left[(i\! -\! m_{2})\delta_{x}, (j\! +\! \tfrac{1}{2}\! +\! n_{2}\! -\! m_{3})\delta_{y}, 0\right]^{\top}\\
&= {\bf a}_{3}\! +\! \rho_{3}{\bf a}_{2}\! -\! {\bf a}_{2}\! +\! \rho_{2}{\bf a}_{1}\! +\! {\bf x}\!\left(i\! -\! m_{2}, j\! +\! \tfrac{1}{2}\! +\! n_{2}\! -\! m_{3}, 0\right);
\end{align*}
if $m_{3} \leq j < n_{2}$ and $0 \leq i < m_{1} + m_{2}$, then
\begin{small}
\begin{align*}
{\bf x}(i, \hat{j}, n_{3}) &= {\bf a}_{3}\!\! +\!\! {\bf a}_{2}\!\! +\!\! \left[(i\!\! -\!\! m_{2}\!\! -\!\! m_{1})\delta_{x}, (j\!\! +\!\! \tfrac{1}{2}\!\! -\!\! m_{3})\delta_{y}, 0\right]^{\top}\\
&= {\bf a}_{3}\!\! +\!\! {\bf a}_{2}\!\! -\!\! {\bf a}_{1}\!\! +\!\! \left[(i\!\! +\!\! n_{1}\!\! -\!\! m_{1}\!\! -\!\! m_{2})\delta_{x}, (j\!\! +\!\! \tfrac{1}{2}\!\! -\!\! m_{3})\delta_{y}, 0\right]^{\top}\\
&= {\bf a}_{3}\!\! +\!\! \rho_{3}{\bf a}_{2}\!\! +\!\! \rho_{1}\rho_{4}{\bf a}_{1}\!\! +\!\! \psi_{2}{\bf a}_{1}\!\! -\!\! {\bf a}_{1}\!\! +\!\! {\bf x}\!\left(i\!\! +\!\! n_{1}\!\! -\!\! m_{1}\!\! -\!\! m_{2}, j\!\! +\!\! \tfrac{1}{2}\!\! -\!\! m_{3}, 0\right);
\end{align*}
\end{small}
otherwise, i.e. $m_{3} \leq j < n_{2}$ and $m_{1} + m_{2} \leq i < n_{1}$, then
\begin{small}
\begin{align*}
{\bf x}(i, \hat{j}, n_{3}) &= {\bf a}_{3}\!\! +\!\! {\bf a}_{2}\!\! +\!\! \left[(i\!\! -\!\! m_{2}\!\! -\!\! m_{1})\delta_{x}, (j\!\! +\!\! \tfrac{1}{2}\!\! -\!\! m_{3})\delta_{y}, 0\right]^{\top}\\
&= {\bf a}_{3}\!\! +\!\! \rho_{3}{\bf a}_{2}\!\! +\!\! \rho_{1}\rho_{4}{\bf a}_{1}\!\! +\!\! \psi_{2}{\bf a}_{1}\!\! +\!\! {\bf x}\!\left(i\!\! -\!\! m_{1}\!\! -\!\! m_{2}, j\!\! +\!\! \tfrac{1}{2}\!\! -\!\! m_{3}, 0\right).
\end{align*}
\end{small}
\item[(ii)] $\cos\theta_{\gamma} \geq 0$, $\cos\theta_{\beta} < 0$ ($\rho_{1} = 0$, $\rho_{2} = 1$, $\rho_{4} = 0$, $\psi_{2} = 1$)\\
If $0 \leq j < m_{3}$ and $0 \leq i < m_{2}$, then
\begin{align*}
{\bf x}(i, \hat{j}, n_{3}) &= {\bf a}_{3}\!\! +\!\! \left[(i\!\! +\!\! n_{1}\!\! -\!\! m_{2})\delta_{x}, (j\!\! +\!\! \tfrac{1}{2}\!\! +\!\! n_{2}\!\! -\!\! m_{3})\delta_{y}, 0\right]^{\top}\\
&= {\bf a}_{3}\!\! +\!\! \rho_{3}{\bf a}_{2}\!\! -\!\! {\bf a}_{2}\!\! +\!\! \rho_{2}{\bf a}_{1}\!\! -\!\! {\bf a}_{1}\!\! +\!\! {\bf x}\!\left(i\!\! +\!\! n_{1}\!\! -\!\! m_{2}, j\!\! +\!\! \tfrac{1}{2}\!\! +\!\! n_{2}\!\! -\!\! m_{3}, 0\right);
\end{align*}
if $0 \leq j < m_{3}$ and $m_{2} \leq i < n_{1}$, then
\begin{align*}
{\bf x}(i, \hat{j}, n_{3}) &= {\bf a}_{3}\!\! +\!\! {\bf a}_{1}\!\! +\!\! \left[(i\!\! -\!\! m_{2})\delta_{x}, (j\!\! +\!\! \tfrac{1}{2}\!\! +\!\! n_{2}\!\! -\!\! m_{3})\delta_{y}, 0\right]^{\top}\\
&= {\bf a}_{3}\!\! +\!\! \rho_{3}{\bf a}_{2}\!\! -\!\! {\bf a}_{2}\!\! +\!\! \rho_{2}{\bf a}_{1}\!\! +\!\! {\bf x}\!\left(i\!\! -\!\! m_{2}, j\!\! +\!\! \tfrac{1}{2}\!\! +\!\! n_{2}\!\! -\!\! m_{3}, 0\right);
\end{align*}
if $m_{3} \leq j < n_{2}$ and $0 \leq i < m_{1} + m_{2}$, then
\begin{small}
\begin{align*}
{\bf x}(i, \hat{j}, n_{3}) &= {\bf a}_{3}\!\! +\!\! {\bf a}_{2}\!\! +\!\! \left[(i\!\! +\!\! n_{1}\!\! -\!\! m_{2}\!\! -\!\! m_{1})\delta_{x}, (j\!\! +\!\! \tfrac{1}{2}\!\! -\!\! m_{3})\delta_{y}, 0\right]^{\top}\\
&= {\bf a}_{3}\!\! +\!\! \rho_{3}{\bf a}_{2}\!\! +\!\! \rho_{1}\rho_{4}{\bf a}_{1}\!\! +\!\! \psi_{2}{\bf a}_{1}\!\! -\!\! {\bf a}_{1}\!\! +\!\! {\bf x}\!\left(i\!\! +\!\! n_{1}\!\! -\!\! m_{1}\!\! -\!\! m_{2}, j\!\! +\!\! \tfrac{1}{2}\!\! -\!\! m_{3}, 0\right);
\end{align*}
\end{small}
otherwise, i.e. $m_{3} \leq j < n_{2}$ and $m_{1} + m_{2} \leq i < n_{1}$, then
\begin{small}
\begin{align*}
{\bf x}(i, \hat{j}, n_{3}) &= {\bf a}_{3}\!\! +\!\! {\bf a}_{2}\!\! +\!\! {\bf a}_{1}\!\! +\!\! \left[(i\!\! -\!\! m_{1}\!\! -\!\! m_{2})\delta_{x}, (j\!\! +\!\! \tfrac{1}{2}\!\! -\!\! m_{3})\delta_{y}, 0\right]^{\top}\\
&= {\bf a}_{3}\!\! +\!\! \rho_{3}{\bf a}_{2}\!\! +\!\! \rho_{1}\rho_{4}{\bf a}_{1}\!\! +\!\! \psi_{2}{\bf a}_{1}\!\! +\!\! {\bf x}\!\left(i\!\! -\!\! m_{1}\!\! -\!\! m_{2}, j\!\! +\!\! \tfrac{1}{2}\!\! -\!\! m_{3}, 0\right).
\end{align*}
\end{small}
\item[(iii)] $\cos\theta_{\gamma} < 0$, $\cos\theta_{\beta} \geq 0$ ($\rho_{1} = 1$, $\rho_{2} = 0$, $\rho_{4} = 0$, $\psi_{2} = 1$)\\
If $0 \leq j < m_{3}$ and $0 \leq i < m_{2}$, then
\begin{small}
\begin{align*}
{\bf x}(i, \hat{j}, n_{3}) &= {\bf a}_{3}\!\! -\!\! {\bf a}_{1}\!\! +\!\! \left[(i\!\! +\!\! n_{1}\!\! -\!\! m_{2})\delta_{x}, (j\!\! +\!\! \tfrac{1}{2}\!\! +\!\! n_{2}\!\! -\!\! m_{3})\delta_{y}, 0\right]^{\top}\\
&= {\bf a}_{3}\!\! +\!\! \rho_{3}{\bf a}_{2}\!\! -\!\! {\bf a}_{2}\!\! +\!\! \rho_{2}{\bf a}_{1}\!\! -\!\! {\bf a}_{1}\!\! +\!\! {\bf x}\!\left(i\!\! +\!\! n_{1}\!\! -\!\! m_{2}, j\!\! +\!\! \tfrac{1}{2}\!\! +\!\! n_{2}\!\! -\!\! m_{3}, 0\right);
\end{align*}
\end{small}
if $0 \leq j < m_{3}$ and $m_{2} \leq i < n_{1}$, then
\begin{align*}
{\bf x}(i, \hat{j}, n_{3}) &= {\bf a}_{3}\! +\! \left[(i\! -\! m_{2})\delta_{x}, (j\! +\! \tfrac{1}{2}\! +\! n_{2}\! -\! m_{3})\delta_{y}, 0\right]^{\top}\\
&= {\bf a}_{3}\! +\! \rho_{3}{\bf a}_{2}\! -\! {\bf a}_{2}\! +\! \rho_{2}{\bf a}_{1}\! +\! {\bf x}\!\left(i\! -\! m_{2}, j\! +\! \tfrac{1}{2}\! +\! n_{2}\! -\! m_{3}, 0\right);
\end{align*}
if $m_{3} \leq j < n_{2}$ and $0 \leq i < m_{1} + m_{2}$, then
\begin{small}
\begin{align*}
{\bf x}(i, \hat{j}, n_{3}) &= {\bf a}_{3}\!\! +\!\! {\bf a}_{2}\!\! +\!\! \left[(i\!\! -\!\! m_{2}\!\! -\!\! m_{1}\!\! +\!\! n_{1})\delta_{x}, (j\!\! +\!\! \tfrac{1}{2}\!\! -\!\! m_{3})\delta_{y}, 0\right]^{\top}\\
&= {\bf a}_{3}\!\! +\!\! \rho_{3}{\bf a}_{2}\!\! +\!\! \rho_{1}\rho_{4}{\bf a}_{1}\!\! +\!\! \psi_{2}{\bf a}_{1}\!\! -\!\! {\bf a}_{1}\!\! +\!\! {\bf x}\!\left(i\!\! +\!\! n_{1}\!\! -\!\! m_{1}\!\! -\!\! m_{2}, j\!\! +\!\! \tfrac{1}{2}\!\! -\!\! m_{3}, 0\right);
\end{align*}
\end{small}
otherwise, i.e. $m_{3} \leq j < n_{2}$ and $m_{1} + m_{2} \leq i < n_{1}$, then
\begin{small}
\begin{align*}
{\bf x}(i, \hat{j}, n_{3}) &= {\bf a}_{3}\! +\! {\bf a}_{2}\! +\! {\bf a}_{1}\! +\! \left[(i\! -\! m_{2}\! -\! m_{1})\delta_{x}, (j\! +\! \tfrac{1}{2}\! -\! m_{3})\delta_{y}, 0\right]^{\top}\\
&= {\bf a}_{3}\! +\! \rho_{2}{\bf a}_{2}\! +\! \rho_{1}\rho_{4}{\bf a}_{1}\! +\! \psi_{2}{\bf a}_{1}\! +\! {\bf x}\!\left(i\! -\! m_{1}\! -\! m_{2}, j\! +\! \tfrac{1}{2}\! -\! m_{3}, 0\right).
\end{align*}
\end{small}
\item[(iv)] $\cos\theta_{\gamma} < 0$, $\cos\theta_{\beta} < 0$ ($\rho_{1} = 1$, $\rho_{2} = 1$, $\rho_{4} = 1$, $\psi_{2} = 1$)\\
If $0 \leq j < m_{3}$ and $0 \leq i < m_{2}$, then
\begin{small}
\begin{align*}
{\bf x}(i, \hat{j}, n_{3}) &= {\bf a}_{3}\! +\! \left[(i\! +\! n_{1}\! -\! m_{2})\delta_{x}, (j\! +\! \tfrac{1}{2}\! +\! n_{2}\! -\! m_{3})\delta_{y}, 0\right]^{\top}\\
&= {\bf a}_{3}\! +\! \rho_{3}{\bf a}_{2}\! -\! {\bf a}_{2}\! +\! \rho_{2}{\bf a}_{1}\! -\! {\bf a}_{1}\! +\! {\bf x}\!\left(i\! +\! n_{1}\! -\! m_{2}, j\! +\! \tfrac{1}{2}\! +\! n_{2}\! -\! m_{3}, 0\right);
\end{align*}
\end{small}
if $0 \leq j < m_{3}$ and $m_{2} \leq i < n_{1}$, then
\begin{align*}
{\bf x}(i, \hat{j}, n_{3}) &= {\bf a}_{3}\! +\! {\bf a}_{1}\! +\! \left[(i\! -\! m_{2})\delta_{x}, (j\! +\! \tfrac{1}{2}\! +\! n_{2}\! -\! m_{3})\delta_{y}, 0\right]^{\top}\\
&= {\bf a}_{3}\! +\! \rho_{3}{\bf a}_{2}\! -\! {\bf a}_{2}\! +\! \rho_{2}{\bf a}_{1}\! +\! {\bf x}\!\left(i\! -\! m_{2}, j\! +\! \tfrac{1}{2}\! +\! n_{2}\! -\! m_{3}, 0\right);
\end{align*}
if $m_{3} \leq j < n_{2}$ and $0 \leq i < m_{1} + m_{2}$, then
\begin{small}
\begin{align*}
{\bf x}(i, \hat{j}, n_{3}) &= {\bf a}_{3}\!\! +\!\! {\bf a}_{2}\!\! +\!\! \left[(i\!\! -\!\! (m_{2}\!\! -\!\! n_{1})\!\! -\!\! (m_{1}\!\! -\!\! n_{1}))\delta_{x}, (j\!\! +\!\! \tfrac{1}{2}\!\! -\!\! m_{3})\delta_{y}, 0\right]^{\top}\\
&= {\bf a}_{3}\!\! +\!\! {\bf a}_{2}\!\! +\!\! {\bf a}_{1}\!\! +\!\! \left[(i\!\! +\!\! n_{1}\!\! -\!\! m_{1}\!\! -\!\! m_{2})\delta_{x}, (j\!\! +\!\! \tfrac{1}{2}\!\! -\!\! m_{3})\delta_{y}, 0\right]^{\top}\\
&= {\bf a}_{3}\!\! +\!\! \rho_{3}{\bf a}_{2}\!\! +\!\! \rho_{1}\rho_{4}{\bf a}_{1}\!\! +\!\! \psi_{2}{\bf a}_{1}\!\! -\!\! {\bf a}_{1}\!\! +\!\! {\bf x}\!\left(i\!\! +\!\! n_{1}\!\! -\!\! m_{1}\!\! -\!\! m_{2}, j\!\! +\!\! \tfrac{1}{2}\!\! -\!\! m_{3}, 0\right);
\end{align*}
\end{small}
otherwise, i.e. $m_{3} \leq j < n_{2}$ and $m_{1} + m_{2} \leq i < n_{1}$, then
\begin{small}
\begin{align*}
{\bf x}(i, \hat{j}, n_{3}) &= {\bf a}_{3}\! +\! {\bf a}_{2}\! +\! 2{\bf a}_{1}\! +\! \left[(i\! -\! m_{1}\! -\! m_{2})\delta_{x}, (j\! +\! \tfrac{1}{2}\! -\! m_{3})\delta_{y}, 0\right]^{\top}\\
&= {\bf a}_{3}\! +\! \rho_{3}{\bf a}_{2}\! +\! \rho_{1}\rho_{4}{\bf a}_{1}\! +\! \psi_{2}{\bf a}_{1}\! +\! {\bf x}\!\left(i\! -\! m_{1}\! -\! m_{2}, j\! +\! \tfrac{1}{2}\! -\! m_{3}, 0\right).
\end{align*}
\end{small}
\end{itemize}
The above four rewritten imply that
\begin{small}
\begin{align*}
E_{2}(i, \hat{j}, n_{3}) = \begin{cases}
e^{\imath 2\pi{\bf k}\cdot\left[{\bf a}_{3} + \rho_{3}{\bf a}_{2} - {\bf a}_{2} + \rho_{2}{\bf a}_{1} - {\bf a}_{1}\right]}E_{2}(i\! +\! n_{1}\! -\! m_{2}, \hat{j}\! +\! n_{2}\! -\! m_{3}, 0)\\
{\hspace{6cm}} \mbox{if } 0 \leq i < m_{2}, 0 \leq j < m_{3};\\
e^{\imath 2\pi{\bf k}\cdot\left[{\bf a}_{3} + \rho_{3}{\bf a}_{2} - {\bf a}_{2} + \rho_{2}{\bf a}_{1}\right]}E_{2}(i\! -\! m_{2}, \hat{j}\! +\! n_{2}\! -\! m_{3}, 0)\\
{\hspace{6cm}} \mbox{if } m_{2} \leq i < n_{1}, 0 \leq j < m_{3};\\
e^{\imath 2\pi{\bf k}\cdot\left[{\bf a}_{3} + \rho_{3}{\bf a}_{2} + (\rho_{1}\rho_{4} + \psi_{2}){\bf a}_{1} - {\bf a}_{1}\right]}E_{2}(i\! +\! n_{1}\! -\! m_{1}\! -\! m_{2}, \hat{j}\! -\! m_{3}, 0)\\
{\hspace{5cm}} \mbox{if } 0 \leq i < m_{1}\! +\! m_{2}, m_{3} \leq j < n_{2};\\
e^{\imath 2\pi{\bf k}\cdot\left[{\bf a}_{3} + \rho_{3}{\bf a}_{2}+(\rho_{1}\rho_{4} + \psi_{2}){\bf a}_{1}\right]}E_{2}(i\! -\! m_{1}\! -\! m_{2}, \hat{j}\! -\! m_{3}, 0)\\
{\hspace{5cm}} \mbox{if } m_{1}\! +\! m_{2} \leq i < n_{1}, m_{3} \leq j < n_{2}.
\end{cases}
\end{align*}
\end{small}
Second, let us consider $m_{1} + m_{2} > n_{1}$. In this situation, if we want to rewrite \eqref{eq:arbpt_case2}, there are also four possibilities:
\begin{itemize}
\item[(i)] $\cos\theta_{\gamma} \geq 0$, $\cos\theta_{\beta} \geq 0$ ($\rho_{2} = 0$, $\rho_{4} = 1$, $\rho_{5} = 1$, $\psi_{2} = 0$)\\
If $0 \leq j < m_{3}$ and $0 \leq i < m_{2}$, then
\begin{align*}
{\bf x}(i, \hat{j}, n_{3}) = {\bf a}_{3}\! +\! \rho_{3}{\bf a}_{2}\! -\! {\bf a}_{2}\! +\! \rho_{2}{\bf a}_{1}\! -\! {\bf a}_{1}\! +\! {\bf x}\!\left(i\! +\! n_{1}\! -\! m_{2}, j\! +\! \tfrac{1}{2}\! +\! n_{2}\! -\! m_{3}\right);
\end{align*}
if $0 \leq j < m_{3}$ and $m_{2} \leq i < n_{1}$, then
\begin{align*}
{\bf x}(i, \hat{j}, n_{3}) = {\bf a}_{3}\! +\! \rho_{3}{\bf a}_{2}\! -\! {\bf a}_{2}\! +\! \rho_{2}{\bf a}_{1}\! +\! {\bf x}\!\left(i\! -\! m_{2}, j\! +\! \tfrac{1}{2}\! +\! n_{2}\! -\! m_{3}, 0\right);
\end{align*}
if $m_{3} \leq j < n_{2}$ and $0 \leq i < m_{1} + m_{2} - n_{1}$, then
\begin{small}
\begin{align*}
{\bf x}(i, \hat{j}, n_{3}) &= {\bf a}_{3}\! +\! {\bf a}_{2}\! -\! 2{\bf a}_{1}\! +\! \left[(i\! +\! 2n_{1}\! -\! m_{1}\! -\! m_{2})\delta_{x}, (j\! +\! \tfrac{1}{2}\! -\! m_{3})\delta_{y}, 0\right]^{\top}\\
&= {\bf a}_{3}\! +\! \rho_{3}{\bf a}_{2}\! -\! \rho_{5}\rho_{4}{\bf a}_{1}\! +\! \psi_{2}{\bf a}_{1}\! -\! {\bf a}_{1}\! +\! {\bf x}\!\left(i\! +\! 2n_{1}\! -\! m_{1}\! -\! m_{2}, j\! +\! \tfrac{1}{2}\! -\! m_{3}, 0\right);
\end{align*}
\end{small}
otherwise, i.e. $m_{3} \leq j < n_{2}$ and $m_{1} + m_{2} - n_{1} \leq i < n_{1}$, then
\begin{small}
\begin{align*}
{\bf x}(i, \hat{j}, n_{3}) &= {\bf a}_{3}\! +\! {\bf a}_{2}\! -\! {\bf a}_{1}\! +\! \left[(i\! +\! n_{1}\! -\! m_{1}\! -\! m_{2})\delta_{x}, (j\! +\! \tfrac{1}{2}\! -\! m_{3})\delta_{y}, 0\right]^{\top}\\
&= {\bf a}_{3}\! +\! \rho_{3}{\bf a}_{2}\! -\! \rho_{5}\rho_{4}{\bf a}_{1}\! +\! \psi_{2}{\bf a}_{1}\! +\! {\bf x}\!\left(i\! +\! n_{1}\! -\! m_{1}\! -\! m_{2}, j\! +\! \tfrac{1}{2}\! -\! m_{3}, 0\right).
\end{align*}
\end{small}
\item[(ii)] $\cos\theta_{\gamma} \geq 0$, $\cos\theta_{\beta} < 0$ ($\rho_{2} = 1$, $\rho_{4} = 0$, $\rho_{5} = 0$, $\psi_{2} = 0$)\\
If $0 \leq j < m_{3}$ and $0 \leq i < m_{2}$, then
\begin{align*}
{\bf x}(i, \hat{j}, n_{3}) &= {\bf a}_{3}\! +\! \left[(i\! +\! n_{1}\! -\! m_{2})\delta_{x}, (j\! +\! \tfrac{1}{2}\! +\! n_{2}\! -\! m_{3})\delta_{y}, 0\right]^{\top}\\
&= {\bf a}_{3}\! +\! \rho_{3}{\bf a}_{2}\! -\! {\bf a}_{2}\! +\! \rho_{2}{\bf a}_{1}\! -\! {\bf a}_{1}\! +\! {\bf x}\!\left(i\! +\! n_{1}\! -\! m_{2}, j\! +\! \tfrac{1}{2}\! +\! n_{2}\! -\! m_{3}, 0\right);
\end{align*}
if $0 \leq j < m_{3}$ and $m_{2} \leq i < n_{1}$, then
\begin{align*}
{\bf x}(i, \hat{j}, n_{3}) &= {\bf a}_{3}\! +\! {\bf a}_{1}\! +\! \left[(i\! -\! m_{2})\delta_{x}, (j\! +\! \tfrac{1}{2}\! +\! n_{2}\! -\! m_{3})\delta_{y}, 0\right]^{\top}\\
&= {\bf a}_{3}\! +\! \rho_{3}{\bf a}_{2}\! -\! {\bf a}_{2}\! +\! \rho_{2}{\bf a}_{1}\! +\! {\bf x}\!\left(i\! -\! m_{2}, j\! +\! \tfrac{1}{2}\! +\! n_{2}\! -\! m_{3}, 0\right);
\end{align*}
if $m_{3} \leq j < n_{2}$ and $0 \leq i < m_{1} + m_{2} - n_{1}$, then
\begin{small}
\begin{align*}
{\bf x}(i, \hat{j}, n_{3}) &= {\bf a}_{3}\! +\! {\bf a}_{2}\! -\! {\bf a}_{1}\! +\! \left[(i\! +\! 2n_{1}\! -\! m_{2}\! -\! m_{1})\delta_{x}, (j\! +\! \tfrac{1}{2}\! -\! m_{3})\delta_{y}, 0\right]^{\top}\\
&= {\bf a}_{3}\! +\! \rho_{3}{\bf a}_{2}\! -\! \rho_{5}\rho_{4}{\bf a}_{1}\! +\! \psi_{2}{\bf a}_{1}\! -\! {\bf a}_{1}\! +\! {\bf x}\!\left(i\! +\! 2n_{1}\! -\! m_{1}\! -\! m_{2}, j\! +\! \tfrac{1}{2}\! -\! m_{3}, 0\right);
\end{align*}
\end{small}
otherwise, i.e. $m_{3} \leq j < n_{2}$ and $m_{1} + m_{2} - n_{1} \leq i < n_{1}$, then
\begin{small}
\begin{align*}
{\bf x}(i, \hat{j}, n_{3}) &= {\bf a}_{3}\! +\! {\bf a}_{2}\! +\! \left[(i\! +\! n_{1}\! -\! m_{1}\! -\! m_{2})\delta_{x}, (j\! +\! \tfrac{1}{2}\! -\! m_{3})\delta_{y}, 0\right]^{\top}\\
&= {\bf a}_{3}\! +\! \rho_{3}{\bf a}_{2}\! -\! \rho_{5}\rho_{4}{\bf a}_{1}\! +\! \psi_{2}{\bf a}_{1}\! +\! {\bf x}\!\left(i\! +\! n_{1}\! -\! m_{1}\! -\! m_{2}, j\! +\! \tfrac{1}{2}\! -\! m_{3}, 0\right).
\end{align*}
\end{small}
\item[(iii)] $\cos\theta_{\gamma} < 0$, $\cos\theta_{\beta} \geq 0$ ($\rho_{2} = 0$, $\rho_{4} = 0$, $\rho_{5} = 1$, $\psi_{2} = 0$)\\
If $0 \leq j < m_{3}$ and $0 \leq i < m_{2}$, then
\begin{align*}
{\bf x}(i, \hat{j}, n_{3}) &= {\bf a}_{3}\! -\! {\bf a}_{1}\! +\! \left[(i\! +\! n_{1}\! -\! m_{2})\delta_{x}, (j\! +\! \tfrac{1}{2}\! +\! n_{2}\! -\! m_{3})\delta_{y}, 0\right]^{\top}\\
&= {\bf a}_{3}\! +\! \rho_{3}{\bf a}_{2}\! -\! {\bf a}_{2}\! +\! \rho_{2}{\bf a}_{1}\! -\! {\bf a}_{1}\! +\! {\bf x}\!\left(i\! +\! n_{1}\! -\! m_{2}, j\! +\! \tfrac{1}{2}\! +\! n_{2}\! -\! m_{3}, 0\right);
\end{align*}
if $0 \leq j < m_{3}$ and $m_{2} \leq i < n_{1}$, then
\begin{align*}
{\bf x}(i, \hat{j}, n_{3}) &= {\bf a}_{3}\! +\! \left[(i\! -\! m_{2})\delta_{x}, (j\! +\! \tfrac{1}{2}\! +\! n_{2}\! -\! m_{3})\delta_{y}, 0\right]^{\top}\\
&= {\bf a}_{3}\! +\! \rho_{3}{\bf a}_{2}\! -\! {\bf a}_{2}\! +\! \rho_{2}{\bf a}_{1}\! +\! {\bf x}\!\left(i\! -\! m_{2}, j\! +\! \tfrac{1}{2}\! +\! n_{2}\! -\! m_{3}, 0\right);
\end{align*}
if $m_{3} \leq < n_{2}$ and $0 \leq i < m_{1} + m_{2} - n_{1}$, then
\begin{small}
\begin{align*}
{\bf x}(i, \hat{j}, n_{3}) &= {\bf a}_{3}\! +\! {\bf a}_{2}\! +\! \left[(i\! -\! m_{2}\! -\! (m_{1}\! -\! n_{1}))\delta_{x}, (j\! +\! \tfrac{1}{2}\! -\! m_{3})\delta_{y}, 0\right]^{\top}\\
&= {\bf a}_{3}\! +\! {\bf a}_{2}\! -\! {\bf a}_{1}\! +\! \left[(i\! +\! 2n_{1}\! -\! m_{1}\! -\! m_{2})\delta_{x}, (j\! +\! \tfrac{1}{2}\! -\! m_{3})\delta_{y}, 0\right]^{\top}\\
&= {\bf a}_{3}\! +\! \rho_{3}{\bf a}_{2}\! -\! \rho_{5}\rho_{4}{\bf a}_{1}\! +\! \psi_{2}{\bf a}_{1}\! -\! {\bf a}_{1}\! +\! {\bf x}\!\left(i\! +\! 2n_{1}\! -\! m_{1}\! -\! m_{2}, j\! +\! \tfrac{1}{2}\! -\! m_{3}, 0\right);
\end{align*}
\end{small}
otherwise, i.e. $m_{3} \leq j < n_{2}$ and $m_{1} + m_{2} - n_{1} \leq i < n_{1}$, then
\begin{small}
\begin{align*}
{\bf x}(i, \hat{j}, n_{3}) &= {\bf a}_{3}\! +\! {\bf a}_{2}\! +\! \left[(i\! +\! n_{1}\! -\! m_{1}\! -\! m_{2})\delta_{x}, (j\! +\! \tfrac{1}{2}\! -\! m_{3})\delta_{y}, 0\right]^{\top}\\
&= {\bf a}_{3}\! +\! \rho_{3}{\bf a}_{2}\! -\! \rho_{5}\rho_{4}{\bf a}_{1}\! +\! \psi_{2}{\bf a}_{1}\! +\! {\bf x}\!\left(i\! +\! n_{1}\! -\! m_{1}\! -\! m_{2}, j\! +\! \tfrac{1}{2}\! -\! m_{3}, 0\right).
\end{align*}
\end{small}
\item[(iv)] $\cos\theta_{\gamma} < 0$, $\cos\theta_{\beta} < 0$ ($\rho_{2} = 1$, $\rho_{4} = 1$, $\rho_{5} = 0$, $\psi_{2} = 1$)\\
If $0 \leq j < m_{3}$ and $0 \leq i < m_{2}$, then
\begin{align*}
{\bf x}(i, \hat{j}, n_{3}) &= {\bf a}_{3}\! +\! \left[(i\! -\! m_{2}\! +\! n_{1})\delta_{x}, (j\! +\! \tfrac{1}{2}\! +\! n_{2}\! -\! m_{3})\delta_{y}, 0\right]^{\top}\\
&= {\bf a}_{3}\! +\! \rho_{3}{\bf a}_{2}\! -\! {\bf a}_{2}\! +\! \rho_{2}{\bf a}_{1}\! -\! {\bf a}_{1}\! +\! {\bf x}\!\left(i\! +\! n_{1}\! -\! m_{2}, j\! +\! \tfrac{1}{2}\! +\! n_{2}\! -\! m_{3}, 0\right);
\end{align*}
if $0 \leq j < m_{3}$ and $m_{2} \leq i < n_{1}$, then
\begin{align*}
{\bf x}(i, \hat{j}, n_{3}) &= {\bf a}_{3}\! +\! {\bf a}_{1}\! +\! \left[(i\! -\! m_{2})\delta_{x}, (j\! +\! \tfrac{1}{2}\! +\! n_{2}\! -\! m_{3})\delta_{y}, 0\right]^{\top}\\
&= {\bf a}_{3}\! +\! \rho_{3}{\bf a}_{2}\! -\! {\bf a}_{2}\! +\! \rho_{2}{\bf a}_{1}\! +\! {\bf x}\!\left(i\! -\! m_{2}, j\! +\! \tfrac{1}{2}\! +\! n_{2}\! -\! m_{3}, 0\right);
\end{align*}
if $m_{3} \leq j < n_{2}$ and $0 \leq i < m_{1} + m_{2} - n_{1}$, then
\begin{small}
\begin{align*}
{\bf x}(i, \hat{j}, n_{3}) &= {\bf a}_{3}\! +\! {\bf a}_{2}\! +\! \left[(i\! -\! (m_{2}\! -\! n_{1})\! -\! (m_{1}\! -\! n_{1}))\delta_{x}, (j\! +\! \tfrac{1}{2}\! -\! m_{3})\delta_{y}, 0\right]^{\top}\\
&= {\bf a}_{3}\! +\! \rho_{3}{\bf a}_{2}\! -\! \rho_{5}\rho_{4}{\bf a}_{1}\! +\! \psi_{2}{\bf a}_{1}\! -\! {\bf a}_{1}\! +\! {\bf x}\!\left(i\! +\! 2n_{1}\! -\! m_{1}\! -\! m_{2}, j\! +\! \tfrac{1}{2}\! -\! m_{3}, 0\right);
\end{align*}
\end{small}
otherwise, i.e. $m_{3} \leq j < n_{2}$ and $m_{1} + m_{2} - n_{1} \leq i < n_{1}$, then
\begin{small}
\begin{align*}
{\bf x}(i, \hat{j}, n_{3}) &= {\bf a}_{3}\! +\! {\bf a}_{2}\! +\! {\bf a}_{1}\! +\! \left[(i\! +\! n_{1}\! -\! m_{1}\! -\! m_{2})\delta_{x}, (j\! +\! \tfrac{1}{2}\! -\! m_{3})\delta_{y}, 0\right]^{\top}\\
&= {\bf a}_{3}\! +\! \rho_{3}{\bf a}_{2}\! -\! \rho_{5}\rho_{4}{\bf a}_{1}\! +\! \psi_{2}{\bf a}_{1}\! +\! {\bf x}\!\left(i\! +\! n_{1}\! -\! m_{1}\! -\! m_{2}, j\! +\! \tfrac{1}{2}\! -\! m_{3}, 0\right).
\end{align*}
\end{small}
\end{itemize}
The above four rewritten imply that
\begin{small}
\begin{align*}
E_{2}(i, \hat{j}, n_{3}) = \begin{cases}
e^{\imath 2\pi{\bf k}\cdot\left[{\bf a}_{3} + \rho_{3}{\bf a}_{2} - {\bf a}_{2} + \rho_{2}{\bf a}_{1} - {\bf a}_{1}\right]}E_{2}(i\! +\! n_{1}\! -\! m_{2}, \hat{j}\! +\! n_{2}\! -\! m_{3}, 0)\\
\hspace{6cm} \mbox{if } 0 \leq i < m_{2}, 0 \leq j < m_{3};\\
e^{\imath 2\pi{\bf k}\cdot\left[{\bf a}_{3} + \rho_{3}{\bf a}_{2} - {\bf a}_{2} + \rho_{2}{\bf a}_{1}\right]}E_{2}(i\! -\! m_{2}, \hat{j}\! +\! n_{2}\! -\! m_{3}, 0)\\
\hspace{6cm} \mbox{if } m_{2} \leq i < n_{1}, 0 \leq j < m_{3};\\
e^{\imath 2\pi{\bf k}\cdot\left[{\bf a}_{3} + \rho_{3}{\bf a}_{2} - (\rho_{5}\rho_{4} - \psi_{2}){\bf a}_{1} - {\bf a}_{1}\right]}E_{2}(i\! +\! 2n_{1}\! -\! m_{1}\! -\! m_{2}, \hat{j}\! -\! m_{3}, 0)\\
\hspace{4.5cm} \mbox{if } 0 \leq i < m_{1}\! +\! m_{2}\! -\! n_{1}, m_{3} \leq j < n_{2};\\
e^{\imath 2\pi{\bf k}\cdot\left[{\bf a}_{3} + \rho_{3}{\bf a}_{2} - (\rho_{5}\rho_{4} - \psi_{2}){\bf a}_{1}\right]}E_{2}(i\! +\! n_{1}\! -\! m_{1}\! -\! m_{2}, \hat{j}\! -\! m_{3}, 0)\\
\hspace{4.5cm} \mbox{if } m_{1}\! +\! m_{2}\! -\! n_{1} \leq i < n_{1}, m_{3} \leq j < n_{2}.
\end{cases}
\end{align*}
\end{small}
No matter in which case above, the periodicity shown in \eqref{eq:period3E2} can be obtained by consider $i = 0, 1, \cdots, n_{1} - 1$ and $j = 0, 1, \cdots, n_{2} - 1$. Similarly, we can prove the periodicity \eqref{eq:period3E1} using the same process.

\end{proof}

Consequently, the explicit representation of the discrete partial derivative matrix $C_{3}$ is of the form
\begin{align*}
C_{3} \equiv K_{3} \in \mathbb{C}^{n\times n}
\end{align*}
Which is shown in \eqref{eq:K3}. The matrix representations of \eqref{eq:partialz} are $C_{3}{\bf e}_{1}$ and $C_{3}{\bf e}_{2}$, respectively.

So far, we discretize $\nabla \times E = \imath\omega H$ to the matrix representation $C{\bf e} = \imath\omega{\bf h}$ where
\begin{align*}
C = \begin{bmatrix}
0 & -C_{3} & C_{2}\\
C_{3} & 0 & -C_{1}\\
-C_{2} & C_{1} & 0
\end{bmatrix}\in \mathbb{C}^{3n \times 3n}.
\end{align*}

\subsection{Matrix representation of $\nabla \times H = -\imath\omega \varepsilon E$}

In this subsection, we will continue to derive the matrix representation of $\nabla \times H = -\imath\omega\varepsilon E$. Similar to last subsection, we also devide this subsections to three parts. Since the idea of proof is similar to that of the previous subsection, we will omit all the proofs in this subsection.

\noindent{\bf{Part IV. Partial derivative with respect to $x$ for $H$}.} Using the Yee's scheme, the partial derivative $\partial_{x}H_{2}$ and $\partial_{x}H_{3}$ have the discretizations
\begin{align}\label{eq:partial_x_H}
\dfrac{H_{2}(\hat{i}, j, \hat{k}) - H_{2}(\hat{i}-1, j, \hat{k})}{\delta_{x}} \;\; \mbox{and}\;\; \dfrac{H_{3}(\hat{i}, \hat{j}, k) - H_{3}(\hat{i}-1, \hat{j}, k)}{\delta_{x}}
\end{align}
for $i = 0, 1, \cdots, n_{1} - 1$, $j = 0, 1, \cdots, n_{2} - 1$, $k = 0, 1, \cdots, n_{3} - 1$. 

\begin{theorem}
By the quasi-periodic condition \eqref{eq:quasi_periodic}, we have
\begin{align}\label{eq:partial1H2}
H_{2}(\widehat{-1}, j, \hat{k}) = \imath\omega\mu_{0} e^{-\imath2\pi{\bf k}\cdot{\bf a}_{1}}H_{2}(\hat{n}_{1}-1, j, \hat{k}) 
\end{align}
and
\begin{align}\label{eq:partial1H3}
H_{3}(\widehat{-1}, \hat{j}, k) = \imath\omega\mu_{0} e^{-\imath2\pi{\bf k}\cdot{\bf a}_{1}}H_{3}(\hat{n}_{1}-1, \hat{j}, k),
\end{align}
for $j = 0, 1, \cdots, n_{2}-1$ and $k = 0, 1, \cdots, n_{3}-1$.
\end{theorem}
Therefore, the matrix representations of \eqref{eq:partial_x_H} are $-C_{1}^{\ast}{\bf h}_{2}$ and $-C_{1}^{\ast}{\bf h}_{3}$.

\noindent{\bf{Part V. Partial derivative with respect to $y$ for $H$}.} Using the Yee's scheme, the partial derivative $\partial_{y}H_{1}$ and $\partial_{y}H_{3}$ have the discretizations
\begin{align}\label{eq:partial_y_H}
\dfrac{H_{1}(i, \hat{j}, \hat{k}) - H_{1}(i, \hat{j}-1, \hat{k})}{\delta_{y}} \;\; \mbox{and}\;\; \dfrac{H_{3}(\hat{i}, \hat{j}, k) - H_{3}(\hat{i}, \hat{j}-1, k)}{\delta_{y}}
\end{align}
for $i = 0, 1, \cdots, n_{1} - 1$, $j = 0, 1, \cdots, n_{2} - 1$, $k = 0, 1, \cdots, n_{3} - 1$. 

\begin{theorem}
By the quasi-periodic condition \eqref{eq:quasi_periodic}, we have
\begin{align}\label{eq:partial2H1}
H_{1}(0:n_{1}-1, \widehat{-1}, \hat{k}) = \imath\omega\mu_{0} e^{-\imath2\pi{\bf k}\cdot{\bf a}_{2}}J_{2}^{\ast}H_{1}(0:n_{1}-1, \hat{n}_{2}-1, \hat{k}) 
\end{align}
and
\begin{align}\label{eq:partial2H3}
H_{3}(\hat{0}:\hat{n}_{1}-1, \widehat{-1}, k) = \imath\omega\mu_{0} e^{-\imath2\pi{\bf k}\cdot{\bf a}_{2}}J_{2}^{\ast}H_{3}(\hat{0}:\hat{n}_{1}-1, \hat{n}_{2}-1, k),
\end{align}
for $k = 0, 1, \cdots, n_{3}-1$.
\end{theorem}
Therefore, the matrix representations of \eqref{eq:partial_y_H} are $-C_{2}^{\ast}{\bf h}_{1}$ and $-C_{2}^{\ast}{\bf h}_{3}$.

\noindent{\bf{Part VI. Partial derivative with respect to $z$ for $H$}.} Using the Yee's scheme, the partial derivative $\partial_{z}H_{1}$ and $\partial_{z}H_{2}$ have the discretizations
\begin{align}\label{eq:partial_z_H}
\dfrac{H_{1}(i, \hat{j}, \hat{k}) - H_{1}(i, \hat{j}, \hat{k}-1)}{\delta_{z}} \;\; \mbox{and}\;\; \dfrac{H_{2}(\hat{i}, j, \hat{k}) - H_{2}(\hat{i}, j, \hat{k}-1)}{\delta_{z}}
\end{align}
for $i = 0, 1, \cdots, n_{1} - 1$, $j = 0, 1, \cdots, n_{2} - 1$, $k = 0, 1, \cdots, n_{3} - 1$. 

\begin{theorem}
By the quasi-periodic condition \eqref{eq:quasi_periodic}, we have
\begin{small}
\begin{align}\label{eq:partial3H1}
\textup{vec}\big(H_{1}(0:n_{1}\!\!-\!\!1, \hat{0}:\hat{n}_{2}\!\!-\!\!1, \widehat{-1})\big) = \imath\omega\mu_{0} e^{-\imath2\pi{\bf k}\cdot{\bf a}_{3}}J_{3}^{\ast}\textup{vec}\big(H_{1}(0:n_{1}\!\!-\!\!1, \hat{0}:\hat{n}_{2}\!\!-\!\!1, \hat{n}_{3}\!\!-\!\!1)\big) 
\end{align}
\end{small}
and
\begin{small}
\begin{align}\label{eq:partial3H2}
\textup{vec}\big(H_{2}(\hat{0}:\hat{n}_{1}\!\!-\!\!1, 0:n_{2}\!\!-\!\!1, \widehat{-1})\big) = \imath\omega\mu_{0} e^{-\imath2\pi{\bf k}\cdot{\bf a}_{3}}J_{3}^{\ast}\textup{vec}\big(H_{2}(\hat{0}:\hat{n}_{1}\!\!-\!\!1, 0:n_{2}\!\!-\!\!1, \hat{n}_{3}\!\!-\!\!1)\big).
\end{align}
\end{small}
\end{theorem}
Therefore, the matrix representations of \eqref{eq:partial_z_H} are $-C_{3}^{\ast}{\bf h}_{1}$ and $-C_{3}^{\ast}{\bf h}_{2}$.

Hence we discretize $\nabla \times H = -\imath\omega\varepsilon E$ to the matrix representation\\ $C^{\ast}{\bf h} = -\imath\omega B{\bf e}$. Until now, we propose the discretizations of curl operators for various lattice structures and list systematically 4 general forms that can be suitable for any situations. From now on, the governing equation \eqref{eq:double_curl} becomes a generalized eigenvalue problem
\begin{align*}
C^{\ast}C{\bf e} \equiv A{\bf e} = \lambda B{\bf e},
\end{align*}
where $\lambda = \mu_{0}\omega^{2}$, our main goal is to develop a fast algorithm to solve this problem.


\section{Eigen-decomposition of the partial derivative operators}\label{section4}

Before we compute the generalized eigenvalue problem $A{\bf e} = \lambda B{\bf e}$, we have to derive eigen-decompositions of the partial derivative operators $C_{1}$, $C_{2}$, and $C_{3}$. This section will be categorized into two parts: first we derive the eigen-decompositions of the differential operators $K_{1}$, $K_{2}$, and $K_{3}$, then use these results to get the eigen-decompositions of $C_{1}$, $C_{2}$, and $C_{3}$.


\subsection{Eigen-decomposition of the differential operators}

From the last section we know that the matrix representations of $C_{1}$ is
\begin{align*}
C_{1} = I_{n_{3}} \otimes I_{n_{2}} \otimes K_{1}.
\end{align*}
Here,
\begin{align*}
K_{1} = \dfrac{1}{\delta_{x}}\begin{bmatrix}
-1 & 1 & & & \\
 & -1 & 1 &  & \\
 & & \ddots & \ddots & \\
 & & & -1 & 1\\
e^{\imath 2\pi {\bf k}\cdot{\bf a}_{1}} & & & & -1 
\end{bmatrix}
\end{align*}
is a discrete matrix of differential operator with quasi-periodic boundary condition in one-dimensional space.

\begin{theorem}\label{thm:eigK1}
Let $K_{1}$ be a differential matrix with quasi-periodic boundary consition, then the $i$-th eigenvalue of $K_{1}$ is 
\begin{align}
\lambda_{i} = \dfrac{e^{\theta_{i}}-1}{\delta_{x}} \hspace{1em}\mbox{with}\hspace{1em}\theta_{i} = \dfrac{\imath 2\pi(i + {\bf k}\cdot{\bf a}_{1})}{n_{1}},\label{eq:ewK1}
\end{align}
and the corresponding $i$-th eigenvector is
\begin{align}
\mbox{\bf x}_{i} &= \begin{bmatrix}
1 & e^{\theta_{i}} & \cdots & e^{(n_{1} - 1)\theta_{i}}
\end{bmatrix}^{\top},\label{eq:evK1}
\end{align}
for $i = 1, \cdots, n_{1}$.
\end{theorem}

\begin{proof}
The proof can be done straightforwardly by observation.
\end{proof}

\begin{theorem}[eigenpairs of $K_2$]\label{thm:eigK2}
The eigenpairs of $K_{2}$ which is mentioned in \eqref{eq:K2} can be written as $(\delta_{y}^{-1}(e^{\theta_{i, j}}-1), \mbox{\bf y}_{i, j}\otimes\mbox{\bf x}_{i})$, where ${\bf x}_{i}$ is of the form as \eqref{eq:evK1} and
\begin{align}
\theta_{i, j} &= \dfrac{\imath 2\pi\left(j - \frac{m_{1}}{n_{1}}i + {\bf  k}\cdot \hat{{\bf a}}_{2}\right)}{n_{2}} \hspace{1em} with \hspace{1em} \hat{{\bf a}}_{2} = {\bf a}_{2} - \frac{m_{1}}{n_{1}}{\bf a}_{1} + \rho_{1}{\bf a}_{1},\label{eq:ewK2}\\
\mbox{\bf y}_{i, j} &= \begin{bmatrix}
1 & e^{\theta_{i, j}} & \cdots &  e^{(n_{2} - 1)\theta_{i, j}}
\end{bmatrix}^{\top}\label{eq:evK2},
\end{align}
for $i = 1, \cdots, n_{1}$, and $j = 1, \cdots, n_{2}$.
\end{theorem}

\begin{proof}
Assume that $(\lambda, [y_{1}, \cdots, y_{n_{2}}]^{\top}\otimes {\bf x}_{i})$ is one of the eigenpairs of $K_{2}$. It satisfies that
\begin{subequations}
\begin{align}
y_{2} - y_{1} &= \lambda\delta_{y}y_{1},\label{eq:y1}\\
&\vdots\notag \\
y_{n_{2}} - y_{n_{2}-1} &= \lambda\delta_{y}y_{n_{2}-1},\label{eq:y2}\\
y_{1}e^{\imath 2\pi {\bf k}\cdot{\bf a}_{2}}J_{2}{\bf x}_{i} - y_{n_{2}}{\bf x}_{i} &= \lambda\delta_{y}y_{n_{2}}{\bf x}_{i}.\label{eq:y3}
\end{align}
\end{subequations}
With the definition of $J_{2}$ which is given in \eqref{eq:J2}, we know that equation \eqref{eq:y3} implies that
\begin{subequations}\label{eq:y30}
\begin{align}
y_{1}e^{\imath 2\pi{\bf k}\cdot{\bf a}_{2}}e^{\imath 2\pi\rho_{1}{\bf k}\cdot{\bf a}_{1}}e^{-\imath 2\pi{\bf k}\cdot{\bf a}_{1}}e^{(n_{1} - m_{1})\theta_{i}} - y_{n_{2}} &= \lambda\delta_{y}y_{n_{2}},\\
y_{1}e^{\imath 2\pi{\bf k}\cdot{\bf a}_{2}}e^{\imath 2\pi\rho_{1}{\bf k}\cdot{\bf a}_{1}}e^{-m_{1}\theta_{i}} - y_{n_{2}} &= \lambda\delta_{y}y_{n_{2}}.
\end{align}
\end{subequations}
Substitute $\theta_{i}$ which is mentioned in \eqref{eq:ewK1} into \eqref{eq:y30}, it show that both of the  equations of \eqref{eq:y30} will become
\begin{align}\label{eq:y300}
y_{1}e^{\imath 2\pi({\bf k}\cdot\hat{\bf a}_{2} - \frac{m_{1}}{n_{1}}i)} - y_{n_{2}} = \lambda\delta_{y}y_{n_{2}}.
\end{align}
Applying Theorem \ref{thm:eigK1} and combining the results in \eqref{eq:y1}, \eqref{eq:y2}, \eqref{eq:y300}, we obtain $\lambda = \delta_{y}^{-1}(e^{\theta_{i, j}} - 1)$ and $y_{s+1} = e^{s\theta_{i, j}}$ for $s = 0, 1, \cdots, n_{2} - 1$.
\end{proof}

\begin{theorem}[eigenpairs of $K_3$ if $\cos\theta_{\alpha} - \cos\theta_{\beta}\cos\theta_{\gamma} \geq 0$]\label{thm:eigK3_acute}
The eigenpairs of $K_{3}$ which is mentioned in \eqref{eq:K3} can be written as $(\delta_{z}^{-1}(e^{\theta_{i, j, k}}-1), \mbox{\bf z}_{i, j, k} \otimes \mbox{\bf y}_{i, j}\otimes\mbox{\bf x}_{i})$, where ${\bf x}_{i}$ and ${\bf y}_{i, j}$ are of the form as \eqref{eq:evK1} and \eqref{eq:evK2}, respectively, and
\begin{align}
\theta_{i, j, k} &= \dfrac{\imath 2\pi(k - \frac{m_{3}}{n_{2}}j - \frac{n_{2} - m_{3}}{n_{2}}\frac{m_{1}}{n_{1}}i + \frac{m_{1}-m_{2}}{n_{1}}i + {\bf  k}\cdot \hat{{\bf a}}_{3})}{n_{3}},\label{eq:acute_ewK3}\\
\mbox{\bf z}_{i, j, k} &= \begin{bmatrix}
1 & e^{\theta_{i, j, k}} & \cdots &  e^{(n_{2} - 1)\theta_{i, j, k}}
\end{bmatrix}^{\top}\label{eq:acute_evK3},
\end{align}
with $\hat{{\bf a}}_{3} = {\bf a}_{3} - \frac{m_{3}}{n_{2}}{\bf a}_{2} - \frac{n_{2} - m_{3}}{n_{2}}\frac{m_{1}}{n_{1}}{\bf a}_{1} + \frac{m_{1} - m_{2}}{n_{1}}{\bf a}_{1} + \rho_{3}{\bf a}_{2} - \frac{m_{3}}{n_{2}}\rho_{1}{\bf a}_{1} + \rho_{2}{\bf a}_{1}$ for $i = 1, \cdots, n_{1}$, $j = 1, \cdots, n_{2}$, and $k = 1, \cdots, n_{3}$.
\end{theorem}

\begin{proof}
As before theorem, we may assume that the eigenpair of $K_{3}$ is of the form $(\lambda, [z_{1}, \cdots, z_{n_{3}}]^{\top}\otimes {\bf y}_{i, j} \otimes {\bf x}_{i})$. By the definition of eigenpair, we can get
\begin{subequations}
\begin{align}
z_{2} - z_{1} &= \lambda\delta_{z}z_{1},\label{eq:acute_z1}\\
&\vdots\notag\\
z_{n_{3}} - z_{n_{3}-1} &= \lambda\delta_{z}z_{n_{3} - 1},\label{eq:acute_z2}\\
z_{1}e^{\imath 2\pi{\bf k}\cdot{\bf a}_{3}}J_{3}({\bf y}_{i, j}\otimes{\bf x}_{i}) - z_{n_{3}}({\bf y}_{i, j}\otimes{\bf x}_{i}) &= \lambda\delta_{z}z_{n_{3}}({\bf y}_{i, j}\otimes{\bf x}_{i}).\label{eq:acute_z3}
\end{align}
\end{subequations}
Base on the formulation of $J_{3}$ in \eqref{eq:acuteJ3} and $y_{i, j}$ in \eqref{eq:evK2}, \eqref{eq:acute_z3} has two situations as below:\\
(1) $m_{1} < m_{2}$
\begin{subequations}
\begin{small}
\begin{align}
z_{1}e^{\imath 2\pi{\bf k}\cdot{\bf a}_{3}}e^{\imath 2\pi\rho_{3}{\bf k}\cdot{\bf a}_{2}}e^{-\imath 2\pi{\bf k}\cdot({\bf a}_{2}\! +\! \rho_{1}\rho_{4}{\bf a}_{1}\! -\! \psi_{1}{\bf a}_{1})}e^{(n_{2} - m_{3})\theta_{i, j}}\!\!\begin{bmatrix}\begin{smallmatrix}
 & e^{-\imath 2\pi{\bf k}\cdot{\bf a}_{1}}I_{m_{2}\! -\! m_{1}}\\
I_{n_{1}\! -\! m_{2}\! +\! m_{1}} &
\end{smallmatrix}\end{bmatrix}\!\!{\bf x}_{i}\notag\\
\hspace{6cm} - z_{n_{3}}{\bf x}_{i} = \lambda\delta_{z}z_{n_{3}}{\bf x}_{i}, \label{eq:acute_z31}\\
z_{1}e^{\imath 2\pi{\bf k}\cdot{\bf a}_{3}}e^{\imath 2\pi\rho_{3}{\bf k}\cdot{\bf a}_{2}}e^{-m_{3}\theta_{i, j}}J_{1}{\bf x}_{i} - z_{n_{3}}{\bf x}_{i} = \lambda\delta_{z}z_{n_{3}}{\bf x}_{i}.\label{eq:acute_z32}
\end{align}
\end{small}
\end{subequations}
Moreover, \eqref{eq:acute_z31} implies that
\begin{subequations}
\begin{align}
z_{1}e^{\imath 2\pi{\bf k}\cdot{\bf a}_{3}}e^{\imath 2\pi\rho_{3}{\bf k}\cdot{\bf a}_{2}}e^{-\imath 2\pi{\bf k}\cdot({\bf a}_{2} + \rho_{1}\rho_{4}{\bf a}_{1} - \psi_{1}{\bf a}_{1})}e^{(n_{2} - m_{3})\theta_{i, j}}e^{-\imath 2\pi{\bf k}\cdot{\bf a}_{1}}e^{(n_{1} - m_{2} + m_{1})\theta_{i}}\notag\\
- z_{n_{3}} = \lambda\delta_{z}z_{n_{3}},\label{eq:acute_z311}\\
z_{1}e^{\imath 2\pi{\bf k}\cdot{\bf a}_{3}}e^{\imath 2\pi\rho_{3}{\bf k}\cdot{\bf a}_{2}}e^{-\imath 2\pi{\bf k}\cdot({\bf a}_{2} + \rho_{1}\rho_{4}{\bf a}_{1} - \psi_{1}{\bf a}_{1})}e^{(n_{2} - m_{3})\theta_{i, j}}e^{(m_{1} - m_{2})\theta_{i}}\notag\\
- z_{n_{3}} = \lambda\delta_{z}z_{n_{3}}.\label{eq:acute_z312}
\end{align}
\end{subequations}
With the definition of $J_{1}$ which is mentioned in \eqref{eq:J1}, \eqref{eq:acute_z32} can divide into
\begin{subequations}
\begin{align}
z_{1}e^{\imath 2\pi{\bf k}\cdot{\bf a}_{3}}e^{\imath 2\pi\rho_{3}{\bf k}\cdot{\bf a}_{2}}e^{-m_{3}\theta_{i, j}}e^{\imath 2\pi\rho_{2}{\bf k}\cdot{\bf a}_{1}}e^{-\imath 2\pi{\bf k}\cdot{\bf a}_{1}}e^{(n_{1} - m_{2})\theta_{i}} - z_{n_{3}} &= \lambda\delta_{z}z_{n_{3}}, \label{eq:acute_z321}\\
z_{1}e^{\imath 2\pi{\bf k}\cdot{\bf a}_{3}}e^{\imath 2\pi\rho_{3}{\bf k}\cdot{\bf a}_{2}}e^{-m_{3}\theta_{i, j}}e^{\imath 2\pi\rho_{2}{\bf k}\cdot{\bf a}_{1}}e^{-m_{2}\theta_{i}} - z_{n_{3}} &= \lambda\delta_{z}z_{n_{3}}.\label{eq:acute_z322}
\end{align}
\end{subequations}
Using the definitions of $\theta_{i}$ and $\theta_{i, j}$ in \eqref{eq:ewK1} and \eqref{eq:ewK2}, respectively, the exponents in \eqref{eq:acute_z311}, \eqref{eq:acute_z312}, \eqref{eq:acute_z321}, and \eqref{eq:acute_z322} satisfy
\begin{subequations}
\begin{align}
&\imath2\pi\{{\bf k}\cdot{\bf a}_{3} + \rho_{3}{\bf k}\cdot{\bf a}_{2} - {\bf k}\cdot({\bf a}_{2} + \rho_{1}\rho_{4}{\bf a}_{1} - \psi_{1}{\bf a}_{1}) - {\bf k}\cdot{\bf a}_{1}\notag\\
&\hspace{3cm} + \tfrac{n_{2} - m_{3}}{n_{2}}(j - \tfrac{m_{1}}{n_{1}}i + {\bf k}\cdot({\bf a}_{2} - \tfrac{m_{1}}{n_{1}}{\bf a}_{1} + \rho_{1}{\bf a}_{1})) + \tfrac{n_{1} - m_{2} + m_{1}}{n_{1}}(i + {\bf k}\cdot{\bf a}_{1})\}\notag\\
=& \imath 2\pi\{{\bf k}\cdot[{\bf a}_{3}\!\! +\!\! \rho_{3}{\bf a}_{2}\!\! -\!\! {\bf a}_{2}\!\! -\!\! \rho_{1}\rho_{4}{\bf a}_{1}\!\! +\!\! \psi_{1}{\bf a}_{1}\!\! -\!\! {\bf a}_{1}\!\! +\!\! \tfrac{n_{2} - m_{3}}{n_{2}}{\bf a}_{2}\!\! -\!\! \tfrac{n_{2} - m_{3}}{n_{2}}\tfrac{m_{1}}{n_{1}}{\bf a}_{1}\!\! +\!\! \tfrac{n_{2} - m_{3}}{n_{2}}\rho_{1}{\bf a}_{1}\!\! +\!\! \tfrac{n_{1} - m_{2} + m_{1}}{n_{1}}{\bf a}_{1}]\notag\\
& \hspace{3cm}+ \tfrac{n_{2} - m_{3}}{n_{2}}j - \tfrac{n_{2} - m_{3}}{n_{2}}\tfrac{m_{1}}{n_{1}}i + \tfrac{n_{1} - m_{2} + m_{1}}{n_{1}}i\}\notag\\
=& \imath2\pi\{{\bf k}\cdot[{\bf a}_{3} - \tfrac{m_{3}}{n_{2}}{\bf a}_{2} + \tfrac{m_{1} - m_{2}}{n_{1}}{\bf a}_{1} - \tfrac{n_{2} - m_{3}}{n_{2}}\tfrac{m_{1}}{n_{1}}{\bf a}_{1} + \rho_{3}{\bf a}_{2} + \tfrac{n_{2}-m_{3}}{n_{2}}\rho_{1}{\bf a}_{1} - \rho_{1}\rho_{4}{\bf a}_{1} + \psi_{1}{\bf a}_{1}]\notag\\
& \hspace{3cm} -\tfrac{m_{3}}{n_{2}}j - \tfrac{n_{2} - m_{3}}{n_{2}}\tfrac{m_{1}}{n_{1}}i + \tfrac{m_{1} - m_{2}}{n_{1}}i\} + \imath 2\pi(i + j),\label{eq:acute_z311e}\\
&\imath 2\pi\{{\bf k}\cdot{\bf a}_{3} + \rho_{3}{\bf k}\cdot{\bf a}_{2} - {\bf k}\cdot({\bf a}_{2} + \rho_{1}\rho_{4}{\bf a}_{1} - \psi_{1}{\bf a}_{1}) \notag\\
&\hspace{3cm}+ \tfrac{n_{2} - m_{3}}{n_{2}}(j - \tfrac{m_{1}}{n_{1}}i + {\bf k}\cdot({\bf a}_{2} - \tfrac{m_{1}}{n_{1}}{\bf a}_{1} + \rho_{1}{\bf a}_{1})) + \tfrac{m_{1}- m_{2}}{n_{1}}(i + {\bf k}\cdot{\bf a}_{1})\}\notag\\
=& \imath 2\pi\{{\bf k}\cdot[{\bf a}_{3}\!\! +\!\! \rho_{3}{\bf a}_{2}\!\! -\!\! {\bf a}_{2}\!\! -\!\! \rho_{1}\rho_{4}{\bf a}_{1}\!\! +\!\! \psi_{1}{\bf a}_{1}\!\! +\!\! \tfrac{n_{2} - m_{3}}{n_{2}}{\bf a}_{2}\!\! -\!\! \tfrac{n_{2} - m_{3}}{n_{2}}\tfrac{m_{1}}{n_{1}}{\bf a}_{1}\!\! +\!\! \tfrac{n_{2} - m_{3}}{n_{2}}\rho_{1}{\bf a}_{1}\!\! +\!\! \tfrac{m_{1} - m_{2}}{n_{1}}{\bf a}_{1}]\notag\\
&\hspace{3cm} + \tfrac{n_{2}- m_{3}}{n_{2}}j - \tfrac{n_{2}-m_{3}}{n_{2}}\tfrac{m_{1}}{n_{1}}i + \tfrac{m_{1}-m_{2}}{n_{1}}i\}\notag\\
=& \imath 2\pi\{{\bf k}\cdot[{\bf a}_{3} - \tfrac{m_{3}}{n_{2}}{\bf a}_{2} - \tfrac{n_{2}-m_{3}}{n_{2}}\tfrac{m_{1}}{n_{1}}{\bf a}_{1} + \tfrac{m_{1}-m_{2}}{n_{1}}{\bf a}_{1} + \rho_{3}{\bf a}_{2} - \rho_{1}\rho_{4}{\bf a}_{1} + \psi_{1}{\bf a}_{1} + \tfrac{n_{2} - m_{3}}{n_{2}}\rho_{1}{\bf a}_{1}]\notag\\
&\hspace{3cm} -\tfrac{m_{3}}{n_{2}}j - \tfrac{n_{2}-m_{3}}{n_{2}}\tfrac{m_{1}}{n_{1}}i + \tfrac{m_{1}-m_{2}}{n_{1}}i\} + \imath 2\pi j,\label{eq:acute_z312e}\\
& \imath 2\pi\{{\bf k}\cdot{\bf a}_{3}\!\! +\!\! \rho_{3}{\bf k}\cdot{\bf a}_{2}\!\! -\!\! \tfrac{m_{3}}{n_{2}}(j\!\! -\!\! \tfrac{m_{1}}{n_{1}}i\!\! +\!\! {\bf k}\cdot({\bf a}_{2}\!\! -\!\! \tfrac{m_{1}}{n_{1}}{\bf a}_{1}\!\! +\!\! \rho_{1}{\bf a}_{1}))\!\! +\!\! \rho_{2}{\bf k}\cdot{\bf a}_{1}\!\! -\!\! {\bf k}\cdot{\bf a}_{1}\!\! +\!\! \tfrac{n_{1}-m_{2}}{n_{1}}(i\!\! +\!\! {\bf k}\cdot{\bf a}_{1})\}\notag\\
=& \imath 2\pi\{{\bf k}\cdot[{\bf a}_{3}\!\! +\!\! \rho_{3}{\bf a}_{2}\!\! -\!\! \tfrac{m_{3}}{n_{2}}{\bf a}_{2}\!\! +\!\! \tfrac{m_{3}}{n_{2}}\tfrac{m_{1}}{n_{1}}{\bf a}_{1}\!\! -\!\! \tfrac{m_{3}}{n_{2}}\rho_{1}{\bf a}_{1}\!\! +\!\! \rho_{2}{\bf a}_{1}\!\! -\!\! {\bf a}_{1}\!\! +\!\! \tfrac{n_{1} - m_{2}}{n_{1}}{\bf a}_{1}]\!\! -\!\! \tfrac{m_{3}}{n_{2}}j\!\! +\!\! \tfrac{m_{3}}{n_{2}}\tfrac{m_{1}}{n_{1}}i\!\! +\!\! \tfrac{n_{1}-m_{2}}{n_{1}}i\}\notag\\
=& \imath 2\pi\{{\bf k}\cdot[{\bf a}_{3} - \tfrac{m_{3}}{n_{2}}{\bf a}_{2} - \tfrac{n_{2}-m_{3}}{n_{2}}\tfrac{m_{1}}{n_{1}}{\bf a}_{1} + \tfrac{m_{1}-m_{2}}{n_{1}}{\bf a}_{1} + \rho_{3}{\bf a}_{2} - \tfrac{m_{3}}{n_{2}}\rho_{1}{\bf a}_{1} + \rho_{2}{\bf a}_{1}]\notag\\
&\hspace{3cm} - \tfrac{m_{3}}{n_{2}}j - \tfrac{n_{2}-m_{3}}{n_{2}}\tfrac{m_{1}}{n_{1}}i + \tfrac{m_{1}-m_{2}}{n_{1}}i\} + \imath 2\pi i,\label{eq:acute_z321e}\\
& \imath 2\pi\{{\bf k}\cdot{\bf a}_{3}\! +\! \rho_{3}{\bf k}\cdot{\bf a}_{2}\! -\! \tfrac{m_{3}}{n_{2}}(j\! -\! \tfrac{m_{1}}{n_{1}}i\! +\! {\bf k}\cdot({\bf a}_{2}\! -\! \tfrac{m_{1}}{n_{1}}{\bf a}_{1}\! +\! \rho_{1}{\bf a}_{1}))\! +\! \rho_{2}{\bf k}\cdot{\bf a}_{1}\! -\! \tfrac{m_{2}}{n_{1}}(i\! +\! {\bf k}\cdot{\bf a}_{1})\}\notag\\
=& \imath 2\pi\{{\bf k}\cdot[{\bf a}_{3}\! +\! \rho_{3}{\bf a}_{2}\! -\! \tfrac{m_{3}}{n_{2}}{\bf a}_{2}\! +\! \tfrac{m_{3}}{n_{2}}\tfrac{m_{1}}{n_{1}}{\bf a}_{1}\!-\!\tfrac{m_{3}}{n_{2}}\rho_{1}{\bf a}_{1}\! +\! \rho_{2}{\bf a}_{1}\! -\! \tfrac{m_{2}}{n_{1}}{\bf a}_{1}]\! -\! \tfrac{m_{3}}{n_{2}}j\! +\! \tfrac{m_{3}}{n_{2}}\tfrac{m_{1}}{n_{1}}i\! -\! \tfrac{m_{2}}{n_{1}}i\}\notag\\
=& \imath 2\pi\{{\bf k}\cdot[{\bf a}_{3}\!\! -\!\! \tfrac{m_{3}}{n_{2}}{\bf a}_{2}\!\! -\!\! \tfrac{n_{2}-m_{3}}{n_{2}}\tfrac{m_{1}}{n_{1}}{\bf a}_{1}\!\! +\!\! \tfrac{m_{1}-m_{2}}{n_{1}}{\bf a}_{1}\!\! +\!\! \rho_{3}{\bf a}_{2}\!\! -\!\! \tfrac{m_{3}}{n_{2}}\rho_{1}{\bf a}_{1}\!\! +\!\! \rho_{2}{\bf a}_{1}]\!\! -\!\! \tfrac{m_{3}}{n_{2}}j\!\! -\!\! \tfrac{n_{2}-m_{3}}{n_{2}}\tfrac{m_{1}}{n_{1}}i\!\! +\!\! \tfrac{m_{1}-m_{2}}{n_{1}}i\}.\label{eq:acute_z322e}
\end{align}
\end{subequations}
We have to check whether $-\rho_{1}\rho_{4} + \psi_{1} + \rho_{1} = \rho_{2}$:
\begin{itemize}
\item If $\rho_{1} = 0$ and $\rho_{2} = 0$, since $\psi_{1} = 0$, we have $-\rho_{1}\rho_{4} + \psi_{1} + \rho_{1} = 0 = \rho_{2}$.
\item If $\rho_{1} = 0$ and $\rho_{2} = 1$, since $\psi_{1} = 1$, we have $-\rho_{1}\rho_{4} + \psi_{1} + \rho_{1} = 1 = \rho_{2}$.
\item If $\rho_{1} = 1$ and $\rho_{2} = 0$, since $\psi_{1} = 0$, $\rho_{4} = 1$, we have $-\rho_{1}\rho_{4} + \psi_{1} + \rho_{1} = 0 = \rho_{2}$.
\item If $\rho_{1} = 1$ and $\rho_{2} = 1$, since $\psi_{1} = 0$, $\rho_{4} = 0$, we have $-\rho_{1}\rho_{4} + \psi_{1} + \rho_{1} = 1 = \rho_{2}$.
\end{itemize}
Substitute \eqref{eq:acute_z311e}, \eqref{eq:acute_z312e}, \eqref{eq:acute_z321e}, and \eqref{eq:acute_z322e} into \eqref{eq:acute_z311}, \eqref{eq:acute_z312}, \eqref{eq:acute_z321}, and \eqref{eq:acute_z322}, respectively, we reduce \eqref{eq:acute_z3} to be
\begin{align}\label{eq:acute_z3_m2m1}
z_{1}e^{\imath 2\pi\left({\bf k}\cdot\hat{\bf a}_{3} - \frac{m_{3}}{n_{2}}j - \frac{n_{2} - m_{3}}{n_{2}}\frac{m_{1}}{n_{1}}i + \tfrac{m_{1}-m_{2}}{n_{1}}i\right)} - z_{n_{3}} = \lambda\delta_{z}z_{n_{3}}.
\end{align}
(2) $m_{1} > m_{2}$
\begin{subequations}
\begin{align}
z_{1}e^{\imath 2\pi{\bf k}\cdot{\bf a}_{3}}e^{\imath 2\pi\rho_{3}{\bf k}\cdot{\bf a}_{2}}e^{-\imath 2\pi{\bf k}\cdot({\bf a}_{2} - \rho_{2}\rho_{4}{\bf a}_{1} - \psi_{1}{\bf a}_{1})}\begin{bmatrix}\begin{smallmatrix}
& e^{-\imath 2\pi{\bf k}\cdot{\bf a}_{1}}I_{n_{1} - m_{1} + m_{2}}\\
I_{m_{1} - m_{2}} & 
\end{smallmatrix}\end{bmatrix}e^{(n_{2}-m_{3})\theta_{i, j}}{\bf x}_{i}\notag\\
 - z_{n_{3}}{\bf x}_{i} = \lambda\delta_{z}z_{n_{3}}{\bf x}_{i},\label{eq:acute_z33}\\
z_{1}e^{\imath 2\pi{\bf k}\cdot{\bf a}_{3}}e^{\imath 2\pi\rho_{3}{\bf k}\cdot{\bf a}_{2}}e^{-m_{3}\theta_{i, j}}J_{1}{\bf x}_{i} - z_{n_{3}}{\bf x}_{i} = \lambda\delta_{z}z_{n_{3}}{\bf x}_{i}.\label{eq:acute_z34}
\end{align}
\end{subequations}
Moreover, \eqref{eq:acute_z33} implies that
\begin{subequations}
\begin{align}
z_{1}e^{\imath 2\pi{\bf k}\cdot{\bf a}_{3}}e^{\imath 2\pi\rho_{3}{\bf k}\cdot{\bf a}_{2}}e^{-\imath 2\pi{\bf k}\cdot({\bf a}_{2} - \rho_{2}\rho_{4}{\bf a}_{1} - \psi_{1}{\bf a}_{1})}e^{(n_{2} - m_{3})\theta_{i, j}}e^{-\imath 2\pi{\bf k}\cdot{\bf a}_{1}}e^{(m_{1} - m_{2})\theta_{i}}\notag\\
- z_{n_{3}} = \lambda\delta_{z}z_{n_{3}},\label{eq:acute_z331}\\
z_{1}e^{\imath 2\pi{\bf k}\cdot{\bf a}_{3}}e^{\imath 2\pi\rho_{3}{\bf k}\cdot{\bf a}_{2}}e^{-\imath 2\pi{\bf k}\cdot({\bf a}_{2} - \rho_{2}\rho_{4}{\bf a}_{1} - \psi_{1}{\bf a}_{1})}e^{(n_{2} - m_{3})\theta_{i, j}}e^{-(n_{1} - m_{1} + m_{2})\theta_{i}}\notag\\
- z_{n_{3}} = \lambda\delta_{z}z_{n_{3}}.\label{eq:acute_z332}
\end{align}
\end{subequations}
From the definition of $J_{1}$ in \eqref{eq:J1}, \eqref{eq:acute_z34} implies that \eqref{eq:acute_z321} and \eqref{eq:acute_z322}. Using the definitions of $\theta_{i}$ and $\theta_{i, j}$ in \eqref{eq:ewK1} and \eqref{eq:ewK2}, respectively, the exponents in \eqref{eq:acute_z331} and \eqref{eq:acute_z332} satisfy
\begin{subequations}
\begin{align}
&\imath 2\pi\{{\bf k}\cdot{\bf a}_{3}\! +\! \rho_{3}{\bf k}\cdot{\bf a}_{2}\! -\! {\bf k}\cdot({\bf a}_{2}\! -\! \rho_{2}\rho_{4}{\bf a}_{1}\! -\! \psi_{1}{\bf a}_{1})\! +\! \tfrac{n_{2} - m_{3}}{n_{2}}(j\! -\! \tfrac{m_{1}}{n_{1}}i\! +\! {\bf k}\cdot({\bf a}_{2}\! -\! \tfrac{m_{1}}{n_{1}}{\bf a}_{1}\! +\! \rho_{1}{\bf a}_{1}))\notag\\
&\hspace{3cm} - {\bf k}\cdot{\bf a}_{1} + \tfrac{m_{1} - m_{2}}{n_{1}}(i + {\bf k}\cdot{\bf a}_{1})\}\notag\\
=& \imath 2\pi\{{\bf k}\cdot[{\bf a}_{3}\!\! +\!\! \rho_{3}{\bf a}_{2}\!\! -\!\! {\bf a}_{2}\!\! +\!\! \rho_{2}\rho_{4}{\bf a}_{1}\!\! +\!\! \psi_{1}{\bf a}_{1}\!\! +\!\! \tfrac{n_{2} - m_{3}}{n_{2}}{\bf a}_{2}\!\! -\!\! \tfrac{n_{2} - m_{3}}{n_{2}}\tfrac{m_{1}}{n_{1}}{\bf a}_{1}\!\! +\!\! \tfrac{n_{2} - m_{3}}{n_{2}}\rho_{1}{\bf a}_{1}\!\! -\!\! {\bf a}_{1}\!\! +\!\! \tfrac{m_{1} - m_{2}}{n_{1}}{\bf a}_{1}]\notag\\
&\hspace{3cm}+ \tfrac{n_{2} - m_{3}}{n_{2}}j - \tfrac{n_{2} - m_{3}}{n_{2}}\tfrac{m_{1}}{n_{1}}i + \tfrac{m_{1} - m_{2}}{n_{1}}i\}\notag\\
=& \imath 2\pi\{{\bf k}\cdot[{\bf a}_{3}\! -\! \tfrac{m_{3}}{n_{2}}{\bf a}_{2}\! -\! \tfrac{n_{2} - m_{3}}{n_{2}}\tfrac{m_{1}}{n_{1}}{\bf a}_{1}\! +\! \tfrac{m_{1} - m_{2}}{n_{1}}{\bf a}_{1}\! -\! {\bf a}_{1}\! +\! \rho_{3}{\bf a}_{2}\! +\! \rho_{2}\rho_{4}{\bf a}_{1}\! +\! \psi_{1}{\bf a}_{1}\! +\! \tfrac{n_{2} - m_{3}}{n_{2}}\rho_{1}{\bf a}_{1}]\notag\\
&\hspace{3cm} - \tfrac{m_{3}}{n_{2}}j - \tfrac{n_{2} - m_{3}}{n_{2}}\tfrac{m_{1}}{n_{1}}i + \tfrac{m_{2} - m_{1}}{n_{1}}i\} + \imath 2\pi j,\label{eq:acute_z331e}
\end{align}
\begin{align}
&\imath 2\pi\{{\bf k}\cdot{\bf a}_{3}\! +\! \rho_{3}{\bf k}\cdot{\bf a}_{2}\! -\! {\bf k}\cdot({\bf a}_{2}\! -\! \rho_{2}\rho_{4}{\bf a}_{1}\! -\! \psi_{1}{\bf a}_{1})\! +\! \tfrac{n_{2} - m_{3}}{n_{2}}(j\! -\! \tfrac{m_{1}}{n_{1}}i\! +\! {\bf k}\cdot({\bf a}_{2}\! -\! \tfrac{m_{1}}{n_{1}}{\bf a}_{1}\! +\! \rho_{1}{\bf a}_{1}))\notag\\
&\hspace{3cm} - \tfrac{n_{1} - m_{1} + m_{2}}{n_{1}}(i + {\bf k}\cdot{\bf a}_{1})\}\notag\\
=& \imath 2\pi\{{\bf k}\cdot[{\bf a}_{3}\!\! +\!\! \rho_{3}{\bf a}_{2}\!\! -\!\! {\bf a}_{2}\!\! +\!\! \rho_{2}\rho_{4}{\bf a}_{1}\!\! +\!\! \psi_{1}{\bf a}_{1}\!\! +\!\! \tfrac{n_{2}-m_{3}}{n_{2}}{\bf a}_{2}\!\! -\!\! \tfrac{n_{2} - m_{3}}{n_{2}}\tfrac{m_{1}}{n_{1}}{\bf a}_{1}\!\! +\!\! \tfrac{n_{2} - m_{3}}{n_{2}}\rho_{1}{\bf a}_{1}\!\! -\!\! \tfrac{n_{1} - m_{1} + m_{2}}{n_{1}}{\bf a}_{1}]\notag\\
&\hspace{3cm} + \tfrac{n_{2} - m_{3}}{n_{2}}j - \tfrac{n_{2}-m_{3}}{n_{2}}\tfrac{m_{1}}{n_{1}}i - \tfrac{n_{1} - m_{1} + m_{2}}{n_{1}}i\}\notag\\
=& \imath 2\pi\{{\bf k}\cdot[{\bf a}_{3}\! -\! \tfrac{m_{3}}{n_{2}}{\bf a}_{2}\! -\! \tfrac{n_{2}-m_{3}}{n_{2}}\tfrac{m_{1}}{n_{1}}{\bf a}_{1}\! +\! \tfrac{m_{1}-m_{2}}{n_{1}}{\bf a}_{1}\! -\! {\bf a}_{1}\! +\! \rho_{3}{\bf a}_{2}\! +\! \rho_{2}\rho_{4}{\bf a}_{1}\! +\! \psi_{1}{\bf a}_{1}\! +\! \tfrac{n_{2}-m_{3}}{n_{2}}\rho_{1}{\bf a}_{1}]\notag\\
&\hspace{3cm} - \tfrac{m_{3}}{n_{2}}j - \tfrac{n_{2} - m_{3}}{n_{2}}\tfrac{m_{1}}{n_{1}}i + \tfrac{m_{1} - m_{2}}{n_{1}}i\} + \imath 2\pi(j - i).\label{eq:acute_z332e}
\end{align}
\end{subequations}
We have to check whether $\rho_{1} - 1 + \rho_{2}\rho_{4} + \psi_{1} = \rho_{2}$:
\begin{itemize}
\item If $\rho_{1} = 0$ and $\rho_{2} = 0$, since $\psi_{1} = 1$, we have $\rho_{1} - 1 + \rho_{2}\rho_{4} + \psi_{1} = 0 = \rho_{2}$.
\item If $\rho_{1} = 0$ and $\rho_{2} = 1$, since $\psi_{1} = 1$, $\rho_{4} = 1$, we have $\rho_{1} - 1 + \rho_{2}\rho_{4} + \psi_{1} = 1 = \rho_{2}$.
\item If $\rho_{1} = 1$ and $\rho_{2} = 0$, since $\psi_{1} = 0$, we have $\rho_{1} - 1 + \rho_{2}\rho_{4} + \psi_{1} = 0 = \rho_{2}$.
\item If $\rho_{1} = 1$ and $\rho_{2} = 1$, since $\psi_{1} = 1$, $\rho_{4} = 0$, we have $\rho_{1} - 1 + \rho_{2}\rho_{4} + \psi_{1} = 1 = \rho_{2}$.
\end{itemize}
Substituting \eqref{eq:acute_z331e}, \eqref{eq:acute_z332e}, \eqref{eq:acute_z321e}, and \eqref{eq:acute_z322e} into \eqref{eq:acute_z331}, \eqref{eq:acute_z332}, \eqref{eq:acute_z321}, and \eqref{eq:acute_z322}, respectively, we reduce \eqref{eq:acute_z3} to
\begin{align}\label{eq:acute_z3_m1m2}
z_{1}e^{\imath 2\pi\left({\bf k}\cdot\hat{\bf a}_{3} - \frac{m_{3}}{n_{2}}j - \frac{n_{2} - m_{3}}{n_{2}}\frac{m_{1}}{n_{1}}i + \tfrac{m_{1}-m_{2}}{n_{1}}i\right)} - z_{n_{3}} = \lambda\delta_{z}z_{n_{3}}.
\end{align}
Applying the Theorem \ref{thm:eigK1} and combining the results in \eqref{eq:acute_z1}, \eqref{eq:acute_z2} with \eqref{eq:acute_z3_m2m1}, \eqref{eq:acute_z3_m1m2}, we obtain $\lambda = \delta_{z}^{-1}(e^{\theta_{i, j, k}} - 1)$ and $z_{s+1} = e^{s\theta_{i, j, k}}$ for $s = 0, 1, \cdots, n_{3}$.
\end{proof}

\begin{theorem}[eigenpairs of $K_3$ if $\cos\theta_{\alpha} - \cos\theta_{\beta}\cos\theta_{\gamma} < 0$]\label{thm:eigK3_obtuse}
The eigenpairs of $K_{3}$ which is mentioned in \eqref{eq:K3} can be written as $(\delta_{z}^{-1}(e^{\theta_{i, j, k}}-1), \mbox{\bf z}_{i, j, k} \otimes \mbox{\bf y}_{i, j}\otimes\mbox{\bf x}_{i})$, where ${\bf x}_{i}$ and ${\bf y}_{i, j}$ are of the form as \eqref{eq:evK1} and \eqref{eq:evK2}, respectively, and
\begin{align}
\theta_{i, j, k} &= \dfrac{\imath 2\pi(k - \frac{m_{3}}{n_{2}}j - \frac{n_{2} - m_{3}}{n_{2}}\frac{m_{1}}{n_{1}}i - \frac{m_{2}}{n_{1}}i + {\bf  k}\cdot \hat{{\bf a}}_{3})}{n_{3}},\label{eq:obtuse_ewK3}\\
\mbox{\bf z}_{i, j, k} &= \begin{bmatrix}
1 & e^{\theta_{i, j, k}} & \cdots &  e^{(n_{2} - 1)\theta_{i, j, k}}
\end{bmatrix}^{\top}\label{eq:obtuse_evK3},
\end{align}
with $\hat{{\bf a}}_{3} = {\bf a}_{3} - \frac{m_{3}}{n_{2}}{\bf a}_{2} - \frac{n_{2} - m_{3}}{n_{2}}\frac{m_{1}}{n_{1}}{\bf a}_{1} - \frac{m_{2}}{n_{1}}{\bf a}_{1} + \rho_{3}{\bf a}_{2} - \frac{m_{3}}{n_{2}}\rho_{1}{\bf a}_{1} + \rho_{1}{\bf a}_{1} + \rho_{2}{\bf a}_{1}$ for $i = 1, \cdots, n_{1}$, $j = 1, \cdots, n_{2}$, and $k = 1, \cdots, n_{3}$.
\end{theorem}

\begin{proof}
Assume that the eigenpair of $K_{3}$ is of the form $(\lambda, [z_{1}, \cdots, z_{n_{3}}]^{\top}\otimes {\bf y}_{i, j} \otimes {\bf x}_{i})$. It satisties that
\begin{subequations}
\begin{align}
z_{2} - z_{1} &= \lambda\delta_{z}z_{1},\label{eq:obtuse_z1}\\
&\vdots\notag\\
z_{n_{3}} - z_{n_{3}-1} &= \lambda\delta_{z}z_{n_{3} - 1},\label{eq:obtuse_z2}\\
z_{1}e^{\imath 2\pi{\bf k}\cdot{\bf a}_{3}}J_{3}({\bf y}_{i, j}\otimes{\bf x}_{i}) - z_{n_{3}}({\bf y}_{i, j}\otimes{\bf x}_{i}) &= \lambda\delta_{z}z_{n_{3}}({\bf y}_{i, j}\otimes{\bf x}_{i}).\label{eq:obtuse_z3}
\end{align}
\end{subequations}
With the definitions of $J_{3}$ which is mentioned in \eqref{eq:obtuseJ3} and $y_{i, j}$ in \eqref{eq:evK2}, \eqref{eq:obtuse_z3} has two situations as below:\\
(1) $m_{1} + m_{2} \leq n_{1}$
\begin{subequations}
\begin{small}
\begin{align}
z_{1}e^{\imath 2\pi{\bf k}\cdot{\bf a}_{3}}e^{\imath 2\pi\rho_{3}{\bf k}\cdot{\bf a}_{2}}e^{-\imath 2\pi{\bf k}\cdot{\bf a}_{2}}e^{(n_{2} - m_{3})\theta_{i, j}}J_{1}{\bf x}_{i} - z_{n_{3}}{\bf x}_{i} = \lambda\delta_{z}z_{n_{3}}{\bf x}_{i}, \label{eq:obtuse_z31}\\
z_{1}e^{\imath 2\pi{\bf k}\cdot{\bf a}_{3}}e^{\imath 2\pi\rho_{3}{\bf k}\cdot{\bf a}_{2}}e^{\imath 2\pi(\rho_{1}\rho_{4} + \psi_{2}){\bf k}\cdot{\bf a}_{1}}e^{-m_{3}\theta_{i, j}}\begin{bmatrix}\begin{smallmatrix}
 & e^{-\imath 2\pi{\bf k}\cdot{\bf a}_{1}I_{m_{1} + m_{2}}}\\
I_{n_{1} - m_{1} - m_{2}} &
\end{smallmatrix}\end{bmatrix}{\bf x}_{i}\notag\\
 - z_{n_{3}}{\bf x}_{i} = \lambda\delta_{z}z_{n_{3}}{\bf x}_{i}.\label{eq:obtuse_z32}
\end{align}
\end{small}
\end{subequations}
\eqref{eq:obtuse_z31} implies that
\begin{subequations}
\begin{align}
z_{1}e^{\imath 2\pi{\bf k}\cdot{\bf a}_{3}}e^{\imath 2\pi\rho_{3}{\bf k}\cdot{\bf a}_{2}}e^{-\imath 2\pi{\bf k}\cdot{\bf a}_{2}}e^{(n_{2} - m_{3})\theta_{i, j}}e^{\imath 2\pi\rho_{2}{\bf k}\cdot{\bf a}_{1}}e^{-\imath 2\pi{\bf k}\cdot{\bf a}_{1}}e^{(n_{1} - m_{2})\theta_{i}} - z_{n_{3}} &= \lambda\delta_{z}z_{n_{3}},\label{eq:obtuse_z311}\\
z_{1}e^{\imath 2\pi{\bf k}\cdot{\bf a}_{3}}e^{\imath 2\pi\rho_{3}{\bf k}\cdot{\bf a}_{2}}e^{-\imath 2\pi{\bf k}\cdot{\bf a}_{2}}e^{(n_{2} - m_{3})\theta_{i, j}}e^{\imath 2\pi\rho_{2}{\bf k}\cdot{\bf a}_{1}}e^{-m_{2}\theta_{i}} - z_{n_{3}} &= \lambda\delta_{z}z_{n_{3}}.\label{eq:obtuse_z312}
\end{align}
\end{subequations}
With the definition of $J_{1}$ which is given in \eqref{eq:J1}, we know that \eqref{eq:obtuse_z32} can devide into
\begin{subequations}
\begin{align}
z_{1}e^{\imath 2\pi{\bf k}\cdot{\bf a}_{3}}e^{\imath 2\pi\rho_{3}{\bf k}\cdot{\bf a}_{2}}e^{\imath 2\pi(\rho_{1}\rho_{4} + \psi_{2}){\bf k}\cdot{\bf a}_{1}}e^{-m_{3}\theta_{i, j}}e^{-\imath 2\pi{\bf k}\cdot{\bf a}_{1}}e^{(n_{1} - m_{1} - m_{2})\theta_{i}} - z_{n_{3}} &= \lambda\delta_{z}z_{n_{3}}, \label{eq:obtuse_z321}\\
z_{1}e^{\imath 2\pi{\bf k}\cdot{\bf a}_{3}}e^{\imath 2\pi\rho_{3}{\bf k}\cdot{\bf a}_{2}}e^{\imath 2\pi(\rho_{1}\rho_{4} + \psi_{2}){\bf k}\cdot{\bf a}_{1}}e^{-m_{3}\theta_{i, j}}e^{-(m_{1} + m_{2})\theta_{i}} - z_{n_{3}} &= \lambda\delta_{z}z_{n_{3}}.\label{eq:obtuse_z322}
\end{align}
\end{subequations}
Using the definitions of $\theta_{i}$ and $\theta_{i, j}$ in \eqref{eq:ewK1} and \eqref{eq:ewK2}, respectively, the exponents in \eqref{eq:obtuse_z311}, \eqref{eq:obtuse_z312}, \eqref{eq:obtuse_z321}, and \eqref{eq:obtuse_z322} satisfy
\begin{subequations}
\begin{align}
&\imath2\pi\{{\bf k}\cdot{\bf a}_{3} + \rho_{3}{\bf k}\cdot{\bf a}_{2} - {\bf k}\cdot{\bf a}_{2} + \rho_{2}{\bf k}\cdot{\bf a}_{1} - {\bf k}\cdot{\bf a}_{1}\notag\\
&\hspace{3cm} + \tfrac{n_{2} - m_{3}}{n_{2}}(j\! -\! \tfrac{m_{1}}{n_{1}}i\! +\! {\bf k}\cdot({\bf a}_{2}\! -\! \tfrac{m_{1}}{n_{1}}{\bf a}_{1}\! +\! \rho_{1}{\bf a}_{1}))\! +\! \tfrac{n_{1} - m_{2}}{n_{1}}(i\! +\! {\bf k}\cdot{\bf a}_{1})\}\notag\\
=& \imath 2\pi\{{\bf k}\cdot[{\bf a}_{3}\! +\! \rho_{3}{\bf a}_{2}\! -\! {\bf a}_{2}\! +\! \tfrac{n_{2} - m_{3}}{n_{2}}{\bf a}_{2}\! -\! \tfrac{n_{2} - m_{3}}{n_{2}}\tfrac{m_{1}}{n_{1}}{\bf a}_{1}\! +\! \tfrac{n_{2} - m_{3}}{n_{2}}\rho_{1}{\bf a}_{1}\! +\! \rho_{2}{\bf a}_{1}\! -\! {\bf a}_{1}\! +\! \tfrac{n_{1} - m_{2}}{n_{1}}{\bf a}_{1}]\notag\\
& \hspace{3cm}+ \tfrac{n_{2} - m_{3}}{n_{2}}j - \tfrac{n_{2} - m_{3}}{n_{2}}\tfrac{m_{1}}{n_{1}}i + \tfrac{n_{1} - m_{2}}{n_{1}}i\}\notag\\
=& \imath2\pi\{{\bf k}\cdot[{\bf a}_{3} - \tfrac{m_{3}}{n_{2}}{\bf a}_{2} - \tfrac{n_{2} - m_{3}}{n_{2}}\tfrac{m_{1}}{n_{1}}{\bf a}_{1}  - \tfrac{m_{2}}{n_{1}}{\bf a}_{1}+ \rho_{3}{\bf a}_{2} - \tfrac{m_{3}}{n_{2}}\rho_{1}{\bf a}_{1} + \rho_{1}{\bf a}_{1} + \rho_{2}{\bf a}_{1}]\notag\\
& \hspace{3cm} -\tfrac{m_{3}}{n_{2}}j - \tfrac{n_{2} - m_{3}}{n_{2}}\tfrac{m_{1}}{n_{1}}i - \tfrac{m_{2}}{n_{1}}i\} + \imath 2\pi(i + j),\label{eq:obtuse_z311e}\\
&\imath 2\pi\{{\bf k}\cdot{\bf a}_{3}\!\! +\!\! \rho_{3}{\bf k}\cdot{\bf a}_{2}\!\! -\!\! {\bf k}\cdot{\bf a}_{2}\!\! +\!\! \rho_{2}{\bf k}\cdot{\bf a}_{1}\!\!+\!\! \tfrac{n_{2} - m_{3}}{n_{2}}(j\!\! -\!\! \tfrac{m_{1}}{n_{1}}i\!\! +\!\! {\bf k}\cdot({\bf a}_{2}\!\! -\!\! \tfrac{m_{1}}{n_{1}}{\bf a}_{1}\!\! +\!\! \rho_{1}{\bf a}_{1}))\!\! -\!\! \tfrac{m_{2}}{n_{1}}(i\!\! +\!\! {\bf k}\cdot{\bf a}_{1})\}\notag\\
=& \imath 2\pi\{{\bf k}\cdot[{\bf a}_{3} + \rho_{3}{\bf a}_{2} - {\bf a}_{2} + \tfrac{n_{2} - m_{3}}{n_{2}}{\bf a}_{2} - \tfrac{n_{2} - m_{3}}{n_{2}}\tfrac{m_{1}}{n_{1}}{\bf a}_{1} + \tfrac{n_{2} - m_{3}}{n_{2}}\rho_{1}{\bf a}_{1}  + \rho_{2}{\bf a}_{1} - \tfrac{m_{2}}{n_{1}}{\bf a}_{1}]\notag\\
&\hspace{3cm} + \tfrac{n_{2}- m_{3}}{n_{2}}j - \tfrac{n_{2}-m_{3}}{n_{2}}\tfrac{m_{1}}{n_{1}}i - \tfrac{m_{2}}{n_{1}}i\}\notag\\
=& \imath 2\pi\{{\bf k}\cdot[{\bf a}_{3} - \tfrac{m_{3}}{n_{2}}{\bf a}_{2} - \tfrac{n_{2}-m_{3}}{n_{2}}\tfrac{m_{1}}{n_{1}}{\bf a}_{1} - \tfrac{m_{2}}{n_{1}}{\bf a}_{1} + \rho_{3}{\bf a}_{2} - \tfrac{m_{3}}{n_{2}}\rho_{1}{\bf a}_{1}+ \rho_{1}{\bf a}_{1} + \rho_{2}{\bf a}_{1} ]\notag\\
&\hspace{3cm} -\tfrac{m_{3}}{n_{2}}j - \tfrac{n_{2}-m_{3}}{n_{2}}\tfrac{m_{1}}{n_{1}}i - \tfrac{m_{2}}{n_{1}}i\} + \imath 2\pi j,\label{eq:obtuse_z312e}\\
& \imath 2\pi\{{\bf k}\cdot{\bf a}_{3} + \rho_{3}{\bf k}\cdot{\bf a}_{2}  + (\rho_{1}\rho_{4} + \psi_{2}){\bf k}\cdot{\bf a}_{1} - {\bf k}\cdot{\bf a}_{1}\notag\\
&\hspace{3cm} -\! \tfrac{m_{3}}{n_{2}}(j\! -\! \tfrac{m_{1}}{n_{1}}i\! +\! {\bf k}\cdot({\bf a}_{2}\! -\! \tfrac{m_{1}}{n_{1}}{\bf a}_{1}\! +\! \rho_{1}{\bf a}_{1}))\! +\! \tfrac{n_{1}-m_{1}-m_{2}}{n_{1}}(i\! +\! {\bf k}\cdot{\bf a}_{1})\}\notag\\
=& \imath 2\pi\{{\bf k}\cdot[{\bf a}_{3}\! +\! \rho_{3}{\bf a}_{2}\! +\! (\rho_{1}\rho_{4}\! +\! \psi_{2}){\bf a}_{1}\! -\! \tfrac{m_{3}}{n_{2}}{\bf a}_{2}\! +\! \tfrac{m_{3}}{n_{2}}\tfrac{m_{1}}{n_{1}}{\bf a}_{1}\! -\! \tfrac{m_{3}}{n_{2}}\rho_{1}{\bf a}_{1}\! -\! {\bf a}_{1}\! +\! \tfrac{n_{1} - m_{1} - m_{2}}{n_{1}}{\bf a}_{1}]\notag\\
&\hspace{3cm} - \tfrac{m_{3}}{n_{2}}j +\tfrac{m_{3}}{n_{2}}\tfrac{m_{1}}{n_{1}}i + \tfrac{n_{1}-m_{1}-m_{2}}{n_{1}}i\}\notag\\
=& \imath 2\pi\{{\bf k}\cdot[{\bf a}_{3} - \tfrac{m_{3}}{n_{2}}{\bf a}_{2} - \tfrac{n_{2}-m_{3}}{n_{2}}\tfrac{m_{1}}{n_{1}}{\bf a}_{1} - \tfrac{m_{2}}{n_{1}}{\bf a}_{1} + \rho_{3}{\bf a}_{2} - \tfrac{m_{3}}{n_{2}}\rho_{1}{\bf a}_{1} + (\rho_{1}\rho_{4} + \psi_{2}){\bf a}_{1}]\notag\\
&\hspace{3cm} - \tfrac{m_{3}}{n_{2}}j - \tfrac{n_{2}-m_{3}}{n_{2}}\tfrac{m_{1}}{n_{1}}i - \tfrac{m_{2}}{n_{1}}i\} + \imath 2\pi i,\label{eq:obtuse_z321e}\\
& \imath 2\pi\{{\bf k}\cdot{\bf a}_{3} + \rho_{3}{\bf k}\cdot{\bf a}_{2} + (\rho_{1}\rho_{4} + \psi_{2}){\bf k}\cdot{\bf a}_{1} - \tfrac{m_{1} + m_{2}}{n_{1}}(i + {\bf k}\cdot{\bf a}_{1})\}\notag\\
&\hspace{3cm} - \tfrac{m_{3}}{n_{2}}(j - \tfrac{m_{1}}{n_{1}}i + {\bf k}\cdot({\bf a}_{2} - \tfrac{m_{1}}{n_{1}}{\bf a}_{1} + \rho_{1}{\bf a}_{1}))\notag\\
=& \imath 2\pi\{{\bf k}\cdot[{\bf a}_{3} + \rho_{3}{\bf a}_{2} + (\rho_{1}\rho_{4} + \psi_{2}){\bf a}_{1} - \tfrac{m_{3}}{n_{2}}{\bf a}_{2} + \tfrac{m_{3}}{n_{2}}\tfrac{m_{1}}{n_{1}}{\bf a}_{1}-\tfrac{m_{3}}{n_{2}}\rho_{1}{\bf a}_{1} - \tfrac{m_{1} + m_{2}}{n_{1}}{\bf a}_{1}]\notag\\
&\hspace{3cm} - \tfrac{m_{3}}{n_{2}}j + \tfrac{m_{3}}{n_{2}}\tfrac{m_{1}}{n_{1}}i - \tfrac{m_{1} + m_{2}}{n_{1}}i\}\notag\\
=& \imath 2\pi\{{\bf k}\cdot[{\bf a}_{3} - \tfrac{m_{3}}{n_{2}}{\bf a}_{2} - \tfrac{n_{2}-m_{3}}{n_{2}}\tfrac{m_{1}}{n_{1}}{\bf a}_{1} - \tfrac{m_{2}}{n_{1}}{\bf a}_{1} + \rho_{3}{\bf a}_{2} - \tfrac{m_{3}}{n_{2}}\rho_{1}{\bf a}_{1} + (\rho_{1}\rho_{4} + \psi_{2}){\bf a}_{1}]\notag\\
&\hspace{3cm} - \tfrac{m_{3}}{n_{2}}j - \tfrac{n_{2}-m_{3}}{n_{2}}\tfrac{m_{1}}{n_{1}}i - \tfrac{m_{2}}{n_{1}}i\}.\label{eq:obtuse_z322e}
\end{align}
\end{subequations}
We have to check whether $\rho_{1} + \rho_{2} = \rho_{1}\rho_{4} + \psi_{2}$:
\begin{itemize}
\item If $\rho_{1} = 0$ and $\rho_{2} = 0$, since $\psi_{2} = 0$, we have $\rho_{1} + \rho_{2} = 0 = \rho_{1}\rho_{4} + \psi_{2}$.
\item If $\rho_{1} = 0$ and $\rho_{2} = 1$, since $\psi_{2} = 1$, we have $\rho_{1} + \rho_{2} = 1 = \rho_{1}\rho_{4} + \psi_{2}$.
\item If $\rho_{1} = 1$ and $\rho_{2} = 0$, since $\psi_{2} = 1$, $\rho_{4} = 0$, we have $\rho_{1} + \rho_{2} = 1 = \rho_{1}\rho_{4} + \psi_{2}$.
\item If $\rho_{1} = 1$ and $\rho_{2} = 1$, since $\psi_{2} = 1$, $\rho_{4} = 1$, we have $\rho_{1} + \rho_{2} = 2 = \rho_{1}\rho_{4} + \psi_{2}$.
\end{itemize}
Substituting \eqref{eq:obtuse_z311e}, \eqref{eq:obtuse_z312e}, \eqref{eq:obtuse_z321e}, and \eqref{eq:obtuse_z322e} into \eqref{eq:obtuse_z311}, \eqref{eq:obtuse_z312}, \eqref{eq:obtuse_z321}, and \eqref{eq:obtuse_z322}, respectively, we reduce \eqref{eq:obtuse_z3} to be
\begin{align}\label{eq:obtuse_z3_n1}
z_{1}e^{\imath 2\pi\left({\bf k}\cdot\hat{\bf a}_{3} - \frac{m_{3}}{n_{2}}j - \frac{n_{2} - m_{3}}{n_{2}}\frac{m_{1}}{n_{1}}i - \tfrac{m_{2}}{n_{1}}i\right)} - z_{n_{3}} = \lambda\delta_{z}z_{n_{3}}.
\end{align}
(2) $m_{1} + m_{2} > n_{1}$
\begin{subequations}
\begin{align}
z_{1}e^{\imath 2\pi{\bf k}\cdot{\bf a}_{3}}e^{\imath 2\pi\rho_{3}{\bf k}\cdot{\bf a}_{2}}e^{-\imath 2\pi{\bf k}\cdot{\bf a}_{2}}e^{(n_{2}-m_{3})\theta_{i, j}}J_{1}{\bf x}_{i} - z_{n_{3}}{\bf x}_{i} = \lambda\delta_{z}z_{n_{3}}{\bf x}_{i},\label{eq:obtuse_z33}\\
z_{1}e^{\imath 2\pi{\bf k}\cdot{\bf a}_{3}}e^{\imath 2\pi\rho_{3}{\bf k}\cdot{\bf a}_{2}}e^{-\imath 2\pi(\rho_{5}\rho_{4} - \psi_{2}){\bf k}\cdot{\bf a}_{1}}e^{-m_{3}\theta_{i, j}}\begin{bmatrix}\begin{smallmatrix}
& e^{-\imath 2\pi{\bf k}\cdot{\bf a}_{1}}I_{m_{1} + m_{2} - n_{1}}\\
I_{2n_{1} - m_{1} - m_{2}} & 
\end{smallmatrix}\end{bmatrix}{\bf x}_{i}\notag\\
 - z_{n_{3}}{\bf x}_{i} = \lambda\delta_{z}z_{n_{3}}{\bf x}_{i}.\label{eq:obtuse_z34}
\end{align}
\end{subequations}
With the definition of $J_{1}$ which is mentioned in \eqref{eq:J1}, \eqref{eq:obtuse_z33} implies that \eqref{eq:obtuse_z311} and \eqref{eq:obtuse_z312}. \eqref{eq:obtuse_z34} can be devided into
\begin{subequations}
\begin{align}
z_{1}e^{\imath 2\pi{\bf k}\cdot{\bf a}_{3}}e^{\imath 2\pi\rho_{3}{\bf k}\cdot{\bf a}_{2}}e^{-\imath 2\pi(\rho_{5}\rho_{4} - \psi_{2}){\bf k}\cdot{\bf a}_{1}}e^{- m_{3}\theta_{i, j}}e^{-\imath 2\pi{\bf k}\cdot{\bf a}_{1}}e^{(2n_{1} - m_{1} - m_{2})\theta_{i}} - z_{n_{3}} &= \lambda\delta_{z}z_{n_{3}},\label{eq:obtuse_z341}\\
z_{1}e^{\imath 2\pi{\bf k}\cdot{\bf a}_{3}}e^{\imath 2\pi\rho_{3}{\bf k}\cdot{\bf a}_{2}}e^{-\imath 2\pi(\rho_{5}\rho_{4} - \psi_{2}){\bf k}\cdot{\bf a}_{1}}e^{- m_{3}\theta_{i, j}}e^{-(m_{1} + m_{2} - n_{1})\theta_{i}} - z_{n_{3}} &= \lambda\delta_{z}z_{n_{3}}.\label{eq:obtuse_z342}
\end{align}
\end{subequations}
Using the definitions of $\theta_{i}$ and $\theta_{i, j}$ in \eqref{eq:ewK1} and \eqref{eq:ewK2}, respectively, the exponents in \eqref{eq:obtuse_z341} and \eqref{eq:obtuse_z342} satisfy
\begin{subequations}
\begin{align}
&\imath 2\pi\{{\bf k}\cdot{\bf a}_{3} + \rho_{3}{\bf k}\cdot{\bf a}_{2} - (\rho_{5}\rho_{4} - \psi_{2}){\bf k}\cdot{\bf a}_{1} - \tfrac{m_{3}}{n_{2}}(j - \tfrac{m_{1}}{n_{1}}i + {\bf k}\cdot({\bf a}_{2} - \tfrac{m_{1}}{n_{1}}{\bf a}_{1} + \rho_{1}{\bf a}_{1}))\notag\\
&\hspace{3cm} - {\bf k}\cdot{\bf a}_{1} + \tfrac{2n_{1} - m_{1} - m_{2}}{n_{1}}(i + {\bf k}\cdot{\bf a}_{1})\}\notag\\
=& \imath 2\pi\{{\bf k}\cdot[{\bf a}_{3}\! +\! \rho_{3}{\bf a}_{2}\! -\! (\rho_{5}\rho_{4}\! -\! \psi_{2}){\bf a}_{1}\! -\! \tfrac{m_{3}}{n_{2}}{\bf a}_{2}\! +\! \tfrac{m_{3}}{n_{2}}\tfrac{m_{1}}{n_{1}}{\bf a}_{1}\! -\!  \tfrac{m_{3}}{n_{2}}\rho_{1}{\bf a}_{1}\! -\! {\bf a}_{1}\! +\! \tfrac{2n_{1} - m_{1} - m_{2}}{n_{1}}{\bf a}_{1}]\notag\\
&\hspace{3cm}- \tfrac{m_{3}}{n_{2}}j + \tfrac{m_{3}}{n_{2}}\tfrac{m_{1}}{n_{1}}i + \tfrac{2n_{1} - m_{1} - m_{2}}{n_{1}}i\}\notag\\
=& \imath 2\pi\{{\bf k}\cdot[{\bf a}_{3} - \tfrac{m_{3}}{n_{2}}{\bf a}_{2} - \tfrac{n_{2} - m_{3}}{n_{2}}\tfrac{m_{1}}{n_{1}}{\bf a}_{1} - \tfrac{m_{2}}{n_{1}}{\bf a}_{1} + {\bf a}_{1} + \rho_{3}{\bf a}_{2} - (\rho_{5}\rho_{4} - \psi_{2}){\bf a}_{1} - \tfrac{m_{3}}{n_{2}}\rho_{1}{\bf a}_{1}]\notag\\
&\hspace{3cm} - \tfrac{m_{3}}{n_{2}}j - \tfrac{n_{2} - m_{3}}{n_{2}}\tfrac{m_{1}}{n_{1}}i - \tfrac{m_{2}}{n_{1}}i\} + \imath 2\pi (2i),\label{eq:obtuse_z341e}
\end{align}
\begin{align}
&\imath 2\pi\{{\bf k}\cdot{\bf a}_{3} + \rho_{3}{\bf k}\cdot{\bf a}_{2} - (\rho_{5}\rho_{4} - \psi_{2}){\bf k}\cdot{\bf a}_{1} - \tfrac{m_{3}}{n_{2}}(j - \tfrac{m_{1}}{n_{1}}i + {\bf k}\cdot({\bf a}_{2} - \tfrac{m_{1}}{n_{1}}{\bf a}_{1} + \rho_{1}{\bf a}_{1}))\notag\\
&\hspace{3cm} - \tfrac{m_{1} + m_{2} - n_{1}}{n_{1}}(i + {\bf k}\cdot{\bf a}_{1})\}\notag\\
=& \imath 2\pi\{{\bf k}\cdot[{\bf a}_{3} + \rho_{3}{\bf a}_{2} - (\rho_{5}\rho_{4} - \psi_{2}){\bf a}_{1} - \tfrac{m_{3}}{n_{2}}{\bf a}_{2} + \tfrac{m_{3}}{n_{2}}\tfrac{m_{1}}{n_{1}}{\bf a}_{1} - \tfrac{m_{3}}{n_{2}}\rho_{1}{\bf a}_{1} - \tfrac{m_{1} + m_{2} - n_{1}}{n_{1}}{\bf a}_{1}]\notag\\
&\hspace{3cm} - \tfrac{m_{3}}{n_{2}}j + \tfrac{m_{3}}{n_{2}}\tfrac{m_{1}}{n_{1}}i - \tfrac{m_{1} + m_{2} - n_{1}}{n_{1}}i\}\notag\\
=& \imath 2\pi\{{\bf k}\cdot[{\bf a}_{3} - \tfrac{m_{3}}{n_{2}}{\bf a}_{2} - \tfrac{n_{2}-m_{3}}{n_{2}}\tfrac{m_{1}}{n_{1}}{\bf a}_{1} - \tfrac{m_{2}}{n_{1}}{\bf a}_{1} + {\bf a}_{1} + \rho_{3}{\bf a}_{2} - \tfrac{m_{3}}{n_{2}}\rho_{1}{\bf a}_{1} - (\rho_{5}\rho_{4} - \psi_{2}){\bf a}_{1}]\notag\\
&\hspace{3cm} - \tfrac{m_{3}}{n_{2}}j - \tfrac{n_{2} - m_{3}}{n_{2}}\tfrac{m_{1}}{n_{1}}i - \tfrac{m_{2}}{n_{1}}i\} + \imath 2\pi i.\label{eq:obtuse_z342e}
\end{align}
\end{subequations}
We have to check whether $\rho_{1} + \rho_{2} = \psi_{2} - \rho_{5}\rho_{4} + 1$:
\begin{itemize}
\item If $\rho_{1} = 0$ and $\rho_{2} = 0$, since $\psi_{2} = 0$, $\rho_{4} = 1$, $\rho_{5} = 1$, we have $\rho_{1} + \rho_{2} = 0 = \psi_{2} - \rho_{5}\rho_{4} + 1$.
\item If $\rho_{1} = 0$ and $\rho_{2} = 1$, since $\psi_{2} = 0$, $\rho_{4} = 0$, $\rho_{5} = 0$, we have $\rho_{1} + \rho_{2} = 1 = \psi_{2} - \rho_{5}\rho_{4} + 1$.
\item If $\rho_{1} = 1$ and $\rho_{2} = 0$, since $\psi_{2} = 0$, $\rho_{4} = 0$, $\rho_{5} = 0$, we have $\rho_{1} + \rho_{2} = 1 = \psi_{2} - \rho_{5}\rho_{4} + 1$.
\item If $\rho_{1} = 1$ and $\rho_{2} = 1$, since $\psi_{2} = 1$, $\rho_{4} = 0$, $\rho_{5} = 0$, we have $\rho_{1} + \rho_{2} = 2 = \psi_{2} - \rho_{5}\rho_{4} + 1$.
\end{itemize}
Substituting \eqref{eq:obtuse_z311e}, \eqref{eq:obtuse_z312e}, \eqref{eq:obtuse_z341e}, and \eqref{eq:obtuse_z342e} into \eqref{eq:obtuse_z311}, \eqref{eq:obtuse_z312}, \eqref{eq:obtuse_z341}, and \eqref{eq:obtuse_z342}, respectively, we reduce \eqref{eq:obtuse_z3} to be
\begin{align}\label{eq:obtuse_z3_m1m2}
z_{1}e^{\imath 2\pi\left({\bf k}\cdot\hat{\bf a}_{3} - \frac{m_{3}}{n_{2}}j - \frac{n_{2} - m_{3}}{n_{2}}\frac{m_{1}}{n_{1}}i - \tfrac{m_{2}}{n_{1}}i\right)} - z_{n_{3}} = \lambda\delta_{z}z_{n_{3}}.
\end{align}
Applying Theorem \ref{thm:eigK1} and combining the results in \eqref{eq:obtuse_z1}, \eqref{eq:obtuse_z2} with \eqref{eq:obtuse_z3_n1},  \eqref{eq:obtuse_z3_m1m2}, we obtain $\lambda = \delta_{z}^{-1}(e^{\theta_{i, j, k}} - 1)$ and $z_{s+1} = e^{s\theta_{i, j, k}}$ for $s = 0, 1, \cdots, n_{3}$.
\end{proof}


\subsection{Eigen-decomposition of the partial derivative operators}

By Theorem \ref{thm:eigK3_acute} and Theorem \ref{thm:eigK3_obtuse}, we can observe that the matrix representations of $K_{3}$ has 4 general forms for various lattice structures, but the eigen-decomposition of $K_{3}$ has only two situations, the difference depends on the value $\cos\theta_{\alpha} - \cos\theta_{\beta}\cos\theta_{\gamma}$. Since the matrices $C_{1}$, $C_{2}$, and $C_{3}$ are commute to each other, they can be diagonalized by the same matrix, now we use the same argument in \cite{hhlw2013eig} to this matrix. 

Let
\begin{align}\label{eq:T}
T = \dfrac{1}{\sqrt{n}}\begin{bmatrix}
T_{1} & T_{2} & \cdots & T_{n_{1}}
\end{bmatrix}\in \mathbb{C}^{n\times n}
\end{align}
where
\begin{align*}
T_{i} = \begin{bmatrix}
T_{i, 1} & T_{i, 2} & \cdots & T_{i, n_{2}}
\end{bmatrix}\in \mathbb{C}^{n \times (n_{2}n_{3})}
\end{align*}
and
\begin{align*}
T_{i, j} = \begin{bmatrix}
{\bf z}_{i, j, 1}\otimes{\bf y}_{i, j}\otimes {\bf x}_{i} & {\bf z}_{i, j, 2}\otimes{\bf y}_{i, j}\otimes {\bf x}_{i} & \cdots & {\bf z}_{i, j, n_{3}}\otimes{\bf y}_{i, j}\otimes {\bf x}_{i}
\end{bmatrix}\in \mathbb{C}^{n \times n_{3}}
\end{align*}
for $i = 1, \cdots, n_{1}$ and $j = 1, \cdots, n_{2}$. By the definition of ${\bf x}_{i}$, ${\bf y}_{i, j}$, and ${\bf z}_{i, j, k}$ in \eqref{eq:evK1}, \eqref{eq:evK2}, \eqref{eq:acute_evK3}, and \eqref{eq:obtuse_evK3}, one can prove that the matrix $T$ is unitary. We rewrite the eigenpairs of $K_{1}$, $K_{2}$, and $K_{3}$ to be
\begin{subequations}
\begin{small}
\begin{align}
\theta_{i} &= \dfrac{\imath2\pi i}{n_{1}}+\dfrac{\imath2\pi{\bf k}\cdot{\bf a}_{1}}{n_{1}}\equiv \theta_{n_{1}, i}+\theta_{{\bf a}_{1}},\\
\theta_{i, j} &= \dfrac{\imath2\pi j}{n_{2}}+\dfrac{\imath2\pi}{n_{2}}\left\{{\bf k}\cdot\hat{\bf a}_{2} - \frac{m_{1}}{n_{1}}i\right\}\equiv \theta_{n_{2}, j}+\theta_{\hat{\bf a}_{2}, i},\\
\theta_{i, j, k} &= \dfrac{\imath2\pi k}{n_{3}}+\dfrac{\imath2\pi}{n_{3}}\left\{{\bf k}\cdot\hat{\bf a}_{3} - \frac{m_{3}}{n_{2}}j + \frac{m_{1}m_{3}-n_{2}m_{2}}{n_{1}n_{2}}i\right\}\equiv \theta_{n_{3}, k}+\theta_{\hat{\bf a}_{3}, i, j}
\end{align}
\end{small}
if $\cos\theta_{\alpha} - \cos\theta_{\beta}\cos\theta_{\gamma} \leq 0$ and
\begin{small}
\begin{align}
\theta_{i, j, k} &= \dfrac{\imath2\pi k}{n_{3}}+\dfrac{\imath2\pi}{n_{3}}\left\{{\bf k}\cdot\hat{\bf a}_{3} - \frac{m_{3}}{n_{2}}j + \frac{m_{1}m_{3}-n_{2}m_{1}-n_{2}m_{2}}{n_{1}n_{2}}i\right\}\equiv \theta_{n_{3}, k}+\theta_{\hat{\bf a}_{3}, i, j}
\end{align}
\end{small}
if $\cos\theta_{\alpha} - \cos\theta_{\beta}\cos\theta_{\gamma} > 0$,
\end{subequations}
\begin{subequations}
\begin{align}
{\bf x}_{i} &= D_{{\bf a}_{1}}\begin{bmatrix}
1 & e^{\theta_{n_{1}, i}} & \cdots & e^{(n_{1}-1)\theta_{n_{1}, i}}
\end{bmatrix}^{\top} \equiv D_{{\bf a}_{1}}{\bf u}_{n_{1}, i},\\
{\bf y}_{i, j} &= D_{\hat{\bf a}_{2}, i}\begin{bmatrix}
1 & e^{\theta_{n_{2}, j}} & \cdots & e^{(n_{2}-1)\theta_{n_{2}, j}}
\end{bmatrix}^{\top} \equiv D_{\hat{\bf a}_{2}, i}{\bf u}_{n_{2}, j},\\
{\bf z}_{i, j, k} &= D_{\hat{\bf a}_{3}, i, j}\begin{bmatrix}
1 & e^{\theta_{n_{3}, k}} & \cdots & e^{(n_{3}-1)\theta_{n_{3}, k}}
\end{bmatrix}^{\top} \equiv D_{\hat{\bf a}_{3}, i, j}{\bf u}_{n_{3}, k},
\end{align}
\end{subequations}
and
\begin{subequations}\label{eq:diagonal_matrix}
\begin{align}
D_{{\bf a}_{1}} &= \mbox{diag}\left(1, e^{\theta_{{\bf a}_{1}}}, \cdots, e^{(n_{1}-1)\theta_{{\bf a}_{1}}}\right),\\
D_{\hat{\bf a}_{2}, i} &= \mbox{diag}\left(1, e^{\theta_{\hat{\bf a}_{2}, i}}, \cdots, e^{(n_{2}-1)\theta_{\hat{\bf a}_{2}, i}}\right),\\
D_{\hat{\bf a}_{3}, i, j} &= \mbox{diag}\left(1, e^{\theta_{\hat{\bf a}_{3}, i, j}}, \cdots, e^{(n_{3}-1)\theta_{\hat{\bf a}_{3}, i, j}}\right),
\end{align}
\end{subequations}
for $i = 1, \cdots, n_{1}$, $j = 1, \cdots, n_{2}$, and $k = 1, \cdots, n_{3}$. We denote
\begin{subequations}\label{eq:DFT_matrix}
\begin{align}
U_{n_{1}} &= \begin{bmatrix}
{\bf u}_{n_{1}, 1} & \cdots {\bf u}_{n_{1}, n_{1}}
\end{bmatrix}\in \mathbb{C}^{n_{1}\times n_{1}},\\
U_{n_{2}} &= \begin{bmatrix}
{\bf u}_{n_{2}, 1} & \cdots {\bf u}_{n_{2}, n_{2}}
\end{bmatrix}\in \mathbb{C}^{n_{2}\times n_{2}},\\
U_{n_{3}} &= \begin{bmatrix}
{\bf u}_{n_{3}, 1} & \cdots {\bf u}_{n_{3}, n_{3}}
\end{bmatrix}\in \mathbb{C}^{n_{3}\times n_{3}},
\end{align}
\end{subequations}
and
\begin{subequations}\label{eq:eigenvalue_K}
\begin{align}
\Lambda_{n_{1}} &= \dfrac{1}{\delta_{x}}\mbox{diag}\left(e^{\theta_{n_{1}, 1}+\theta_{{\bf a}_{1}}}-1, \cdots, e^{\theta_{n_{1}, n_{1}}+\theta_{{\bf a}_{1}}}-1\right),\\
\Lambda_{i, n_{2}} &= \dfrac{1}{\delta_{y}}\mbox{diag}\left(e^{\theta_{n_{2}, 1}+\theta_{\hat{\bf a}_{2}, i}}-1, \cdots, e^{\theta_{n_{2}, n_{2}}+\theta_{\hat{\bf a}_{2}, i}}-1\right),\\
\Lambda_{i, j, n_{3}} &= \dfrac{1}{\delta_{z}}\mbox{diag}\left(e^{\theta_{n_{3}, 1}+\theta_{\hat{\bf a}_{3}, i, j}}-1, \cdots, e^{\theta_{n_{3}, n_{3}}+\theta_{\hat{\bf a}_{3}, i, j}}-1\right).
\end{align}
\end{subequations}
By the eigen-decompositions of $K_{1}$, $K_{2}$, and $K_{3}$, we have
\begin{align*}
K_{1}(D_{{\bf a}_{1}}U_{n_{1}}) = (D_{{\bf a}_{1}}U_{n_{1}})\Lambda_{n_{1}},\\
K_{2}(D_{\hat{\bf a}_{2}, i}U_{n_{2}}) = (D_{\hat{\bf a}_{2}, i}U_{n_{2}})\Lambda_{i, n_{2}},\\
K_{3}(D_{\hat{\bf a}_{3}, i, j}U_{n_{3}}) = (D_{\hat{\bf a}_{3}, i, j}U_{n_{3}})\Lambda_{i, j, n_{3}}.
\end{align*}
These equations imply that
\begin{align*}
C_{1}T_{i, j} &= (I_{n_{3}}\otimes I_{n_{2}} \otimes K_{1})T_{i, j} = \dfrac{e^{\theta_{n_{1}, i}+\theta_{{\bf a}_{1}, n_{1}}}-1}{\delta_{x}}T_{i, j},\\
C_{2}T_{i} &= (I_{n_{3}}\otimes K_{2})T_{i} = T_{i}(\Lambda_{i, n_{2}}\otimes I_{n_{3}}),\\
C_{3}T_{i, j} &= K_{3}T_{i, j} = T_{i, j}\Lambda_{i, j, n_{3}},
\end{align*}
for $i = 1, \cdots, n_{1}$ and $j = 1, \cdots, n_{2}$. Therefore, we found that $C_{1}$, $C_{2}$, and $C_{3}$ are simultaneous diagonalizable by $T$, that is,
\begin{align*}
C_{1}T = T\Lambda_{1}, \hspace{1em} C_{2}T = T\Lambda_{2}, \hspace{1em}\mbox{and}\hspace{1em} C_{3}T = T\Lambda_{3},
\end{align*}
where $\Lambda_{1} = \Lambda_{n_{1}}\otimes I_{n_{2}}\otimes I_{n_{3}}$, $\Lambda_{2} = (\oplus_{i = 1}^{n_{1}}\Lambda_{i, n_{2}})\otimes I_{n_{3}}$, and $\Lambda_{3} = \oplus_{i = 1}^{n_{1}}\oplus_{j = 1}^{n_{2}}\Lambda_{i, j, n_{3}}$.

It is worth noting that $U_{n_{1}}$, $U_{n_{2}}$, and $U_{n_{3}}$ in \eqref{eq:DFT_matrix} are the inverse of discrete Fourier matrices, and $D_{{\bf a}_{1}}$, $D_{\hat{\bf a}_{2}, i}$, and $D_{\hat{\bf a}_{3}, i, j}$ in \eqref{eq:diagonal_matrix} are diagonal matrices, so we can use the FFT to accelerate the numerical computation. However, it is a little different to the simple cases, we will introduce the process in next section.

\section{Singular Value Decomposition and the Fast Solver}\label{section5}

Our main goal is to solve a sequence of generalized eigenvalue problems (GEP)
\begin{align}\label{eq:GEP}
A{\bf x} = \lambda B{\bf x}, \hspace{1em}A = C^{\ast}C.
\end{align}
Huang et al. \cite{hhlw2013eig} provide a powerful scheme to solve this GEP, however, it is only suitable for SC lattice and FCC lattice. We will extend this scheme to all of the crystals. So far, we have already derive the matrix representations of discrete curl operators, and found the eigen-decompositions of discrete partial derivative operators for various lattice structures. In this section, we will briefly introduce some important results in \cite{chhl2015} and \cite{hhlw2013eig} to finish our work.

\subsection{Singular value decompositions for discrete curl operators}

Base on the eigen-decompositions of discrete partial derivative operators, we wish to derive the SVD for discrete curl operators. It is well known that the SVD of curl $C$ is relative to $C^{\ast}C$ and $CC^{\ast}$, therefore we first derive the eigen-decompositions of $C^{\ast}C$ and $CC^{\ast}$. We denote
\begin{subequations}
\begin{align}
\Lambda_{q} = \Lambda_{1}^{\ast}\Lambda_{1} + \Lambda_{2}^{\ast}\Lambda_{2} + \Lambda_{3}^{\ast}\Lambda_{3},
\end{align}
and
\begin{align}
\Lambda_{p} = (\Lambda_{1} + \Lambda_{2} + \Lambda_{3})(\Lambda_{1} + \Lambda_{2} + \Lambda_{3})^{\ast}.
\end{align}
\end{subequations}
By the definition of $\Lambda_{1}$, $\Lambda_{2}$, and $\Lambda_{3}$ in last section, we know that $\Lambda_{q}$ and $\Lambda_{p}$ are depend on the Block vector $2\pi{\bf k}$. If $2\pi{\bf k} \neq 0$ belong to the irreducible Brillouin zone, then $\Lambda_{q}$ is positive definite and $\Lambda_{p}$ is of full column rank. Using the techniques in \cite{chhl2015}, we can find the null space and range space of $C^{\ast}C$ and $CC^{\ast}$. Denote
\begin{subequations}\label{eq:null_basis}
\begin{align}
Q_{0} = (I_{3}\otimes T)\begin{bmatrix}
\Lambda_{1}\\
\Lambda_{2}\\
\Lambda_{3}
\end{bmatrix}\Lambda_{q}^{-\tfrac{1}{2}} \equiv (I_{3}\otimes T)\Pi_{0} \in \mathbb{C}^{3n \times n},
\end{align}
and
\begin{align}
P_{0} = (I_{3}\otimes T)\overline{\Pi_{0}} \in \mathbb{C}^{3n \times n}.
\end{align}
\end{subequations}
Since
\begin{align*}
C^{\ast}CQ_{0} = C^{\ast}\begin{bmatrix}
0 & -C_{3} & C_{2}\\
C_{3} & 0 & -C_{1}\\
-C_{2} & C_{1} & 0
\end{bmatrix}\begin{bmatrix}
T\Lambda_{1}\\
T\Lambda_{2}\\
T\Lambda_{3}
\end{bmatrix}= C^{\ast}\begin{bmatrix}
0 & -C_{3} & C_{2}\\
C_{3} & 0 & -C_{1}\\
-C_{2} & C_{1} & 0
\end{bmatrix}\begin{bmatrix}
C_{1}T\\
C_{2}T\\
C_{3}T
\end{bmatrix} = 0,
\end{align*}
and
\begin{align*}
Q_{0}^{\ast}Q_{0} = I_{n},
\end{align*}
Then $Q_{0}$ forms an orthonormal basis of the null space of $C^{\ast}C$. Similarly, $P_{0}$ forms an orthonormal basis of the null space of $CC^{\ast}$. Let $T_{1} = \begin{bmatrix}
T^{\top} & T^{\top} & T^{\top}
\end{bmatrix}^{\top}$ be a full column rank matrix. Taking the orthogonal projection of $T_{1}$ with respect to $Q_{0}$ and $P_{0}$, respectively, we get
\begin{subequations}\label{eq:range_basis1}
\begin{align}
Q_{1} &= (I - Q_{0}Q_{0}^{\ast})T_{1}(\Lambda_{p}^{\ast}\Lambda_{p}\Lambda_{q}^{-1})^{-\tfrac{1}{2}}\notag\\
&= (I_{3}\otimes T)\begin{bmatrix}
\Lambda_{q} - \Lambda_{1}\Lambda_{s}^{\ast}\\
\Lambda_{q} - \Lambda_{2}\Lambda_{s}^{\ast}\\
\Lambda_{q} - \Lambda_{3}\Lambda_{s}^{\ast}
\end{bmatrix}(\Lambda_{p}^{\ast}\Lambda_{p}\Lambda_{q})^{-\tfrac{1}{2}}\notag\\
&\equiv (I_{3}\otimes T)\Pi_{1} \in \mathbb{C}^{3n \times n},\\
P_{1} &= (I - P_{0}P_{0}^{\ast})T_{1}(\Lambda_{p}^{\ast}\Lambda_{p}\Lambda_{q}^{-1})^{-\tfrac{1}{2}} = (I_{3}\otimes T)\overline{\Pi_{1}}\in \mathbb{C}^{3n \times n}.
\end{align}
\end{subequations}
Then $Q_{0}$ and $P_{0}$ belong to the range space of $C^{\ast}C$ and $CC^{\ast}$, respectively. These matrices are orthonormal and satisfy $(C^{\ast}C)Q_{1} = Q_{1}\Lambda_{q}$ and $(CC^{\ast})P_{1} = P_{1}\Lambda_{q}$. The rest part of the orthonormal basis for range space can be defined by 
\begin{subequations}\label{eq:range_basis2}
\begin{align}
Q_{2} &= C^{\ast}T_{1}(\Lambda_{p}^{\ast}\Lambda_{p})^{-\tfrac{1}{2}} = (I_{3}\otimes T)\begin{bmatrix}
\Lambda_{3}^{\ast} - \Lambda_{2}^{\ast}\\
\Lambda_{1}^{\ast} - \Lambda_{3}^{\ast}\\
\Lambda_{2}^{\ast} - \Lambda_{1}^{\ast}
\end{bmatrix}(\Lambda_{p}^{\ast}\Lambda_{p})^{-\tfrac{1}{2}}\notag\\
&\equiv (I_{3}\otimes T)\Pi_{2} \in \mathbb{C}^{3n \times n},\\
P_{2} &= CT_{1}(\Lambda_{p}^{\ast}\Lambda_{p})^{-\tfrac{1}{2}} = (I_{3}\otimes T)(-\overline{\Pi_{2}})\in \mathbb{C}^{3n \times n}.
\end{align}
\end{subequations}
These two matrices satisfy $(C^{\ast}C)Q_{2} = Q_{2}\lambda_{q}$ and $(CC^{\ast}P_{2}) = P_{2}\Lambda_{q}$. Combining equations \eqref{eq:null_basis}, \eqref{eq:range_basis1}, and \eqref{eq:range_basis2}, we define
\begin{align}
Q &\equiv \begin{bmatrix}
Q_{1} & Q_{2} & Q_{0}
\end{bmatrix} = (I_{3}\otimes T)\begin{bmatrix}
\Pi_{1} & \Pi_{2} & \Pi_{0}
\end{bmatrix},\\
P &\equiv \begin{bmatrix}
P_{2} & P_{1} & P_{0}
\end{bmatrix} = (I_{3}\otimes T)\begin{bmatrix}
-\overline{\Pi_{2}} & \overline{\Pi_{1}} & \overline{\Pi_{0}}
\end{bmatrix}.
\end{align}
Then $Q$ and $P$ are unitary, and the eigen-decompositions of $C^{\ast}C$ and $CC^{\ast}$ are in the form
\begin{align}\label{eq:double_eigendecomposition}
C^{\ast}C = Q\mbox{diag}(\Lambda_{q}, \Lambda_{q}, 0)Q^{\ast} \hspace{1em} \mbox{and}\hspace{1em} CC^{\ast} = P\mbox{diag}(\Lambda_{q}, \Lambda_{q}, 0)P^{\ast}.
\end{align}

\begin{theorem}[\cite{chhl2015}]
By the eigen-decompositions of $C^{\ast}C$ and $CC^{\ast}$ in \eqref{eq:double_eigendecomposition}, we have the SVD for discrete curl operator
\begin{align}\label{eq:SVD}
C = P\textup{diag}(\Lambda_{q}^{\tfrac{1}{2}}, \Lambda_{q}^{\tfrac{1}{2}}, 0)Q^{\ast} = P_{r}\Sigma_{r}Q_{r}^{\ast},
\end{align}
where
\begin{align*}
P_{r} = [P_{2}, P_{1}], \; Q_{r} = [Q_{1}, Q_{2}], \; \Sigma_{r} = \textup{diag}(\Lambda_{q}^{\tfrac{1}{2}}, \Lambda_{q}^{\tfrac{1}{2}}).
\end{align*}
\end{theorem}
If $2\pi{\bf k} = {\bf 0}$, please find more details by refer to \cite{chhl2015}.

\subsection{Null space free method}

We are usually interested in several smallest positive eigenvalues, however, the dimension of null space of $C^{\ast}C$ is $n$, it seriously affects the convergence of iteration. We introduce a stunning technique which is called null space free method to avoid this problem. 

\begin{theorem}\label{thm:invariant_space}
Let $Q_{r}$ and $\Sigma_{r}$ be defined in \eqref{eq:SVD}. Then 
\begin{align*}
\textup{span}\{B^{-1}Q_{r}\Sigma_{r}\}
\end{align*}
forms an invariant subspace corresponding to the nonzero eigenvalues of \eqref{eq:GEP}.
\end{theorem}

\begin{proof}
Base on the SVD \eqref{eq:SVD} for discrete curl operator, we can rewrite $A$ as $Q_{r}\Sigma_{r}^{2}Q_{r}^{\ast}$, then
\begin{align*}
A(B^{-1}Q_{r}\Sigma_{r}) &= Q_{r}\Sigma_{r}^{2}Q_{r}^{\ast}B^{-1}Q_{r}\Sigma_{r}\\
&= B(B^{-1}Q_{r}\Sigma_{r})(\Sigma_{r}Q_{r}^{\ast}B^{-1}Q_{r}\Sigma_{r}).
\end{align*}
Since $\Sigma_{r}Q_{r}^{\ast}B^{-1}Q_{r}\Sigma_{r}$ is invertable, we obtain that $B^{-1}Q_{r}\Sigma_{r}$ is a basis of the invariant space corresponding to the nonzero eigenvalues of \eqref{eq:GEP}.
\end{proof}

Let ${\bf x} = B^{-1}Q_{r}\Sigma_{r}{\bf y}$, then Theorem \ref{thm:invariant_space} implies that
\begin{align*}
\left\{\lambda \Big{|}\; \Sigma_{r}Q_{r}^{\ast}B^{-1}Q_{r}\Sigma_{r}{\bf y} = \lambda{\bf y}\right\} = \left\{\lambda\big{|}\; A{\bf x} = \lambda B{\bf x}, \; \lambda \neq 0\right\}.
\end{align*}
This means that we can reduce the GEP \eqref{eq:GEP} to the null space free standard eigenvale problem (NFSEP)
\begin{align}\label{eq:NFSEP}
A_{r}{\bf y} \equiv \Sigma_{r}Q_{r}^{\ast}B^{-1}Q_{r}\Sigma_{r}{\bf y} = \lambda{\bf y}.
\end{align}
All the eigenvalues of $A_{r}$ are equal to the nonzero eigenvalues of the GEP \eqref{eq:GEP}, note that the dimension of $A_{r}$ is $2n$, which is far smaller than the original problem. To find the several smallest eigenvalues of $A_{r}$, we can apply the inverse Lanczos method to solve the NFSEP \eqref{eq:NFSEP}. In each iterate step of inverse Lanczos method, we have to solve the linear system
\begin{align*}
Q_{r}^{\ast}B^{-1}Q_{r}{\bf b} = {\bf c},
\end{align*}
where ${\bf u}, {\bf c} \in \mathbb{C}^{2n}$ are two arbitrary vectors. It is easy to check that  
\begin{align*}
\kappa(Q_{r}^{\ast}B^{-1}Q_{r}) \leq \kappa(B^{-1}).
\end{align*}
Therefore we can use CG method to solve this Hermition positive definite system without any preconditioner.

In fact, although the dimension of $A_{r}$ is smaller than $A$, but $A_{r}$ is a dense matrix. Fortunately, $Q_{r}$ is consisted of $I_{3}\otimes T$ and some diagonal matrices, thus we can use FFT for reducing computation load and accerlerate the computation time. In other word, the matrix-vector multiplications $T^{\ast}{\bf p}$ and $T{\bf q}$ for any vectors ${\bf p}, {\bf q} \in \mathbb{C}^{n}$ can be computed using the FFT in $O(n\log n)$ time.

\subsection{FFT based matrix-vector multiplication}

In \cite{hhlw2013eig}, Huang et al developed the FFT based scheme to compute the matrix-vector multiplications $T^{\ast}{\bf p}$ and $T{\bf q}$ for any vectors ${\bf p}$ and ${\bf q}$. Here, we provide the sketch and the Algorithms for computing $T^{\ast}{\bf p}$ and $T{\bf q}$, respectively. An important Kronecker product formula will be used:
\begin{align}\label{eq:tensor1}
(C^{\top}\otimes A)\mbox{vec}(B) = \mbox{vec}(ABC).
\end{align}

We denote a vector ${\bf p} \in \mathbb{C}^{n}$ by ${\bf p} = [{\bf p}_{1}^{\top} \; \cdots \; {\bf p}_{2}^{\top}]^{\top}$, where ${\bf p}_{k} = [{\bf p}_{1, k}^{\top} \; \cdots \; {\bf p}_{n_{2}, k}^{\top}] \in \mathbb{C}^{n_{1}n_{2}}$ and ${\bf p}_{j, k} = [p_{1, j, k} \; \cdots \; p_{n_{1}, j, k}]^{\top} \in \mathbb{C}^{n_{1}}$ for $j = 1, \cdots, n_{2}$ and $k = 1, \cdots, n_{3}$. We denote another vector ${\bf q} \in \mathbb{C}^{n}$ by ${\bf q} = [{\bf q}_{1}^{\top} \; \cdots \; {\bf q}_{n_{1}}^{\top}]^{\top}$, where ${\bf q}_{i} = [{\bf q}_{i, 1}^{\top} \; \cdots \; {\bf q}_{i, n_{2}}^{\top}]^{\top} \in \mathbb{C}^{n_{2}n_{3}}$ and ${\bf q}_{i, j} = [q_{i, j, 1} \; \cdots \; q_{i, j, n_{3}}]^{\top} \in \mathbb{C}^{n_{3}}$ for $i = 1, \cdots, n_{1}$ and $j = 1, \cdots, n_{2}$.

By the definition of $T$ in \eqref{eq:T}, we have
\begin{small}
\begin{align*}
T^{\ast}{\bf p} = \dfrac{1}{\sqrt{n}}\begin{bmatrix}
T_{1}^{\ast}{\bf p}\\
T_{2}^{\ast}{\bf p}\\
\vdots\\
T_{n_{1}}^{\ast}{\bf p}
\end{bmatrix}, \hspace{0.5em}T_{i}^{\ast}{\bf p} = \begin{bmatrix}
T_{i, 1}^{\ast}{\bf p}\\
T_{i, 2}^{\ast}{\bf p}\\
\vdots\\
T_{i, n_{2}}^{\ast}{\bf p}
\end{bmatrix}, \hspace{0.5em}T_{i, j}^{\ast}{\bf p} = \begin{bmatrix}
({\bf z}_{i, j, 1}\otimes {\bf y}_{i, j}\otimes{\bf x}_{i})^{\ast}{\bf p}\\
({\bf z}_{i, j, 2}\otimes {\bf y}_{i, j}\otimes{\bf x}_{i})^{\ast}{\bf p}\\
\vdots\\
({\bf z}_{i, j, n_{3}}\otimes {\bf y}_{i, j}\otimes{\bf x}_{i})^{\ast}{\bf p}
\end{bmatrix},
\end{align*}
\end{small}
for $i = 1, \cdots, n_{1}$ and $j = 1, \cdots, n_{2}$. Let $P = [{\bf p}_{1} \; \cdots \; {\bf p}_{n_{3}}]$ and $P_{k} = [{\bf p}_{1, k} \; \cdots \; {\bf p}_{n_{2}, k}]$. Using the equation \eqref{eq:tensor1}, we get
\begin{align*}
T_{i, j}^{\ast}{\bf p} = [{\bf z}_{i, j, 1} \; \cdots \; {\bf z}_{i, j, n_{3}}]^{\ast}P^{\top}(\overline{{\bf y}_{i, j}\otimes{\bf x}_{i}}) = U_{n_{3}}^{\ast}D_{\hat{\bf a}_{3}, i, j}^{\ast}P^{\top}(\overline{{\bf y}_{i, j}\otimes{\bf x}_{i}}).
\end{align*}
It is worth to note that $U_{n_{3}}$ is the inverse of discrete Fourier matrix, $D_{\hat{\bf a}_{3}, i, j}$ is a diagonal matrix, so we just need to deal with the matrix-vector product $P^{\top}(\overline{{\bf y}_{i, j}\otimes{\bf x}_{i}})$. By the similar argument, we get
\begin{align*}
{\bf p}_{k}^{\top}(\overline{{\bf y}_{i, j}\otimes{\bf x}_{i}}) = {\bf y}_{i, j}^{\ast}P_{k}^{\top}\overline{\bf x}_{i} \hspace{1em} \mbox{for}\hspace{1em}k = 1, \cdots, n_{3}.
\end{align*}
Given $i, k$ and let
\begin{small}
\begin{align*}
[\xi_{1, k} \; \cdots \; \xi_{n_{1}, k}] = P_{k}^{\ast}[{\bf x}_{1} \; \cdots \; {\bf x}_{n_{1}}] = P_{k}^{\ast}D_{{\bf a}_{1}}U_{n_{1}}
\end{align*}
\end{small}
and
\begin{small}
\begin{align*}
[\eta_{i, 1} \; \cdots \; \eta_{i, n_{2}}] &= \left([{\bf y}_{i, 1} \; \cdots \; {\bf y}_{i, n_{2}}]^{\ast}[P_{1}^{\top}\overline{\bf x}_{i} \; \cdots \; P_{n_{3}}^{\top}\overline{\bf x}_{i}]\right)^{\ast}\\
&= [\xi_{i, 1} \; \cdots \; \xi_{i, n_{3}}]^{\top}D_{\hat{\bf a}_{2}, i}U_{n_{2}},
\end{align*}
\end{small}
then $P^{\top}(\overline{{\bf y}_{i, j}\otimes{\bf x}_{i}}) = \overline{\eta_{i, j}}$. It means that $P^{\top}(\overline{{\bf y}_{i, j}\otimes{\bf x}_{i}})$ is also comprised of the discrete Fourier matrices and diagonal matrices. Let 
\begin{align*}
Z_{i} = U_{n_{3}}^{\ast}[E_{\hat{\bf a}_{3}, i, 1}^{\ast}\overline{\eta_{i, 1}} \; \cdots \; E_{\hat{\bf a}_{3}, i, n_{2}}^{\ast}\overline{\eta_{i, n_{2}}}],
\end{align*}
then we have
\begin{align*}
T^{\ast}{\bf p} = \dfrac{1}{\sqrt{n}}\mbox{vec}\begin{pmatrix}
Z_{1} & \cdots & Z_{n_{1}}
\end{pmatrix}.
\end{align*}

Algorithm \ref{alg:Tp} give the processes for computing $T^{\ast}{\bf p}$.

\begin{algorithm}[H]
\caption{FFT-based matrix-vector multiplication for $T^{\ast}{\bf p}$ \cite{hhlw2013eig}}
\label{alg:Tp}
\begin{algorithmic}[1]
\Require Any vector ${\bf p} \in \mathbb{C}^{n}$.
\Ensure The vector $T^{\ast}{\bf p}$.
\For{$k = 1, \cdots, n_{3}$}
\State Compute $\xi_{i, k}$ with $[\xi_{1, k} \; \cdots \; \xi_{n_{1}, k}]=(P_{k}^{\ast}D_{{\bf a}_{1}})U_{n_{1}}$.
\EndFor
\For{$i = 1, \cdots, n_{1}$}
\State Compute $\eta_{i, j}$ with $[\eta_{i, 1} \; \cdots \; \eta_{i, n_{2}}]=([\xi_{i, 1} \, \cdots \; \xi_{i, n_{3}}]^{\top}D_{{\bf a}_{2}, i})U_{n_{2}}$.
\State Compute $Z_{i}$ with $Z_{i} = U_{n_{3}}^{\ast}[D_{\hat{\bf a}_{3}, i, 1}^{\ast}\overline{\eta_{i, 1}} \; \cdots \; D_{\hat{\bf a}_{3}, i, n_{2}}^{\ast}\overline{\eta_{i, n_{2}}}]$.
\State Set $(T^{\ast}{\bf p})((i-1)n_{2}n_{3}+1:in_{2}n_{3}) = \tfrac{1}{n}\mbox{vec}(Z_{i})$.
\EndFor
\end{algorithmic}
\end{algorithm}

On the other hand, by the equation \eqref{eq:tensor1}, we have
\begin{align*}
T{\bf q} &= \dfrac{1}{\sqrt{n}}\sum_{i = 1}^{n_{1}}T_{i}{\bf q}_{i} = \dfrac{1}{\sqrt{n}}\sum_{i = 1}^{n_{1}}\sum_{j =1}^{n_{2}}T_{i, j}{\bf q}_{i, j}\\
&= \dfrac{1}{\sqrt{n}}\sum_{i = 1}^{n_{1}}\sum_{j = 1}^{n_{2}}\sum_{k = 1}^{n_{3}}q_{i, j, k}({\bf z}_{i, j, k}\otimes{\bf y}_{i, j}\otimes{\bf x}_{i})\\
&= \dfrac{1}{\sqrt{n}}\sum_{i = 1}^{n_{1}}\sum_{j = 1}^{n_{2}}([{\bf z}_{i, j, 1} \; \cdots \; {\bf z}_{i, j, n_{3}}]{\bf q}_{i, j})\otimes{\bf y}_{i, j}\otimes{\bf x}_{i}\\
&= \dfrac{1}{\sqrt{n}}\sum_{i = 1}^{n_{1}}\mbox{vec}\left(\sum_{j = 1}^{n_{2}}({\bf y}_{i, j}\otimes{\bf x}_{i})(D_{\hat{\bf a}_{3}, i, j}U_{n_{3}}{\bf q}_{i, j})^{\top}\right)\\
&= \dfrac{1}{\sqrt{n}}\sum_{i = 1}^{n_{1}}\mbox{vec}\left([{\bf y}_{i, 1}\otimes{\bf x}_{i} \; \cdots \; {\bf y}_{i, n_{2}}\otimes{\bf x}_{i}]\begin{bmatrix}\begin{smallmatrix}
(D_{\hat{\bf a}_{3}, i, 1}U_{n_{3}}{\bf q}_{i, 1})^{\top}\\
\vdots\\
(D_{\hat{\bf a}_{3}, i, n_{2}}U_{n_{3}}{\bf q}_{i, n_{2}})^{\top}
\end{smallmatrix}\end{bmatrix}\right)\\
&\equiv \dfrac{1}{\sqrt{n}}\sum_{i = 1}^{n_{1}}\mbox{vec}\left(\left([{\bf y}_{i, 1} \; \cdots \; {\bf y}_{i, n_{2}}]Q_{\hat{\bf a}_{3}, i}\right)\otimes{\bf x}_{i}\right)\\
&= \dfrac{1}{\sqrt{n}}\sum_{i = 1}^{n_{1}}\mbox{vec}((D_{\hat{\bf a}_{2}, i}U_{n_{2}}Q_{\hat{\bf a}_{3}, i})\otimes{\bf x}_{i}),
\end{align*}
where $Q_{\hat{\bf a}_{3}, i} = [ D_{\hat{\bf a}_{3}, i, 1}U_{n_{3}}{\bf q}_{i, 1} \; \cdots \; D_{\hat{\bf a}_{3}, i, n_{2}}U_{n_{3}}{\bf q}_{i, n_{2}} ]^{\top}$, $i = 1, \cdots, n_{1}$.\\
Let $[{\bf g}_{i, 1} \; \cdots \; {\bf g}_{i, n_{2}}] = D_{\hat{\bf a}_{2}, i}U_{n_{2}}Q_{\hat{\bf a}_{3}, i}$, then we can rewrite the above equation as
\begin{align*}
T{\bf q} &= \dfrac{1}{\sqrt{n}}\sum_{i = 1}^{n_{1}}\mbox{vec}\left([{\bf g}_{i, 1}\otimes{\bf x}_{i} \; \cdots \; {\bf g}_{i, n_{3}}\otimes{\bf x}_{i}]\right)\\
&= \dfrac{1}{\sqrt{n}}\mbox{vec}\left(\begin{bmatrix}
\mbox{vec}\left(\sum_{i = 1}^{n_{1}}{\bf x}_{i}{\bf g}_{i, 1}^{\top}\right) & \cdots & \mbox{vec}\left(\sum_{i = 1}^{n_{1}}{\bf x}_{i}{\bf g}_{i, n_{3}}^{\top}\right)
\end{bmatrix}\right).
\end{align*}
Finally, we can rewrite each summantion as
\begin{align*}
\sum_{i = 1}^{n_{1}}{\bf x}_{i}{\bf g}_{i, k}^{\top} = [{\bf x}_{1} \; \cdots \; {\bf x}_{n_{1}}][{\bf g}_{1, k} \; \cdots \; {\bf g}_{n_{1}, k}]^{\top}= D_{{\bf a}_{1}}U_{n_{1}}[{\bf g}_{1, k} \; \cdots \; {\bf g}_{n_{1}, k}]^{\top}.
\end{align*}

Algorithm \ref{alg:Tq} give the processes for computing $T{\bf q}$.

\begin{algorithm}[H]
\caption{FFT-based matrix-vector multiplication for $T{\bf q}$ \cite{hhlw2013eig}}
\label{alg:Tq}
\begin{algorithmic}[1]
\Require Any vector ${\bf q} \in \mathbb{C}^{n}$.
\Ensure The vector $T{\bf q}$.
\For{$i = 1, \cdots, n_{1}$}
\State Compute $Q_{\hat{\bf a}_{3}, i}$ with {$Q = U_{n_{3}}[{\bf q}_{i, 1} \; \cdots \; {\bf q}_{i, n_{2}}]$,} 

{$Q_{\hat{\bf a}_{3}, i} = [D_{\hat{\bf a}_{3}, i, 1}Q(:, 1) \; \cdots \; D_{\hat{\bf a}_{3}, i, n_{2}}Q(:, n_{2})]^{\top}$}.
\State Compute ${\bf g}_{i, k}$ with $[{\bf g}_{i, 1} \; \cdots \; {\bf g}_{i, n_{3}}] = D_{\hat{\bf a}_{2}, i}(U_{n_{2}}Q_{\hat{\bf a}_{3}, i})$.
\EndFor
\For{$k = 1, \cdots, n_{3}$}
\State Compute $Q = D_{{\bf a}_{1}}U_{n_{1}}[{\bf g}_{1, k} \; \cdots \; {\bf g}_{n_{1}, k}]^{\top}$.
\State Set $(T{\bf q})((k-1)n_{1}n_{2}+1:kn_{1}n_{2}) = \tfrac{1}{n}\mbox{vec}(Q)$.
\EndFor
\end{algorithmic}
\end{algorithm}


\section{Numerical Experiments}\label{section6}

In this section, we will show some results of numerical experiments by appling the method we developed before. We design several photonic crystals with different shapes for each Bravais lattices and compute their band structures. The shape of each material satisfies the symmetry in one of 230 space groups, most of these shapes consist of two patterns: sphere and cylinder. The radius of spheres $r$ and the radius of the base circle of the cylinders $s$ are set to $\tfrac{r}{a} = 0.15 = \tfrac{s}{a}$. We set the numbers of discrete grid $n_{1} = n_{2} = n_{3} = 120$ in every cases, and the dimensions each GEP are 5.184 million. The permittivity in material is set to $\varepsilon = 13$, and permittivity out side the material is equal to the vacuum permittivity $\varepsilon_{0} = 1$, the permeability is alway equal to $\mu_{0} = 1$. In all figures of band structure, the frequency $\omega = \tfrac{\sqrt{\lambda}}{2\pi}$ is shown on the vertical axis, the horizontal axis represent the Bloch wave vectors $2\pi {\bf k}$. 

We solve the eigenvalue problems associate $2\pi{\bf k}$ along the segments connecting any two corner points of the associated irreducible Brillouin zone, every positions of corner points can be found in Appendix, we will show the path we choose above the figures of band structure. To sovle the GEP, we use the inverse Lanczos method as the eigen-solver and CG method as the linear solver, the stopping criteria $tol_{outer}$ for eigen-solver and $tol_{inner}$ for linear solver are given by $10^{-12}$ and $10^{-13}$, respectively. In the inverse Lanczos iteration, we use {\bf{LAPACK}} to evaluate the Ritz pairs, and use {\bf{cuFFT}} to compute the matrix-vector product $T^{\ast}{\bf p}$ and $T{\bf q}$. We partition each segment into 10 regular part, and compute 10 smallest positive eigenvalues of the corresponding GEP for each wave vector. Each GEP can be solved in less than 100 seconds.

We complete all these numerical experiments of the workstation equipped with two Intel(R) Xeon(R) CPU E5-2650 v2 2.60GHz CPUs, NVIDIA Tesla K20c GPU and 5GB GDDR5 on the GPU, and the Red Hat 4.4.7-3 Linux operation system.

Before we show the band structures of materious in various lattices, let us consider the gyroid structure, which can be approximated by
\begin{align*}
\sin(2\pi \tfrac{x}{a})\cos(2\pi \tfrac{y}{a}) + \sin(2\pi \tfrac{y}{a})\cos(2\pi \tfrac{z}{a}) + \sin(2\pi \tfrac{z}{a})\cos(2\pi \tfrac{x}{a}) > 1.1.
\end{align*}
The gyroid structure is belong to the BCC lattice structure.

\begin{figure}[H]
\center
\includegraphics[height=6cm]{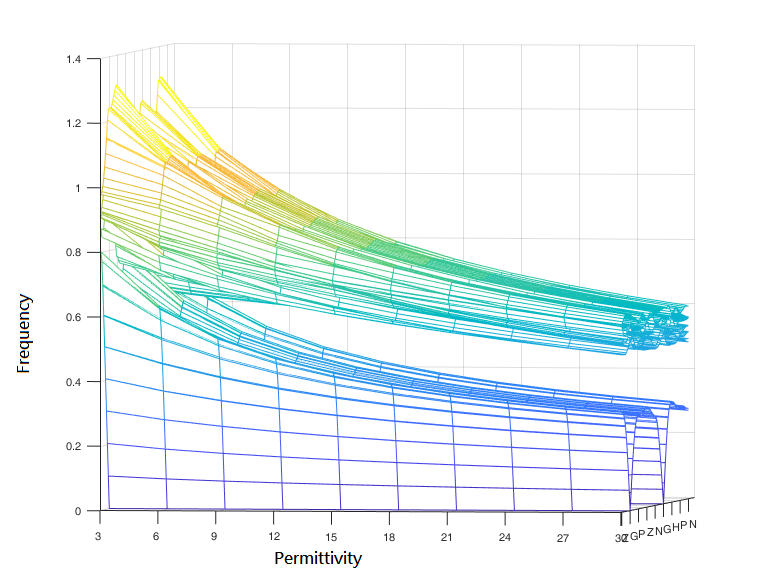}
\caption{The change of band structure when permittivity varies from 3 to 30.}
\label{fig:permittivity_change}
\end{figure}

In Figure \ref{fig:permittivity_change}, we show the change of band structure when permittivity varies from 3 to 30. The axis below represents the permittivities, and the vertical tangent plane for any permittivity is a band structure. As $\varepsilon = 3$, the material is like air, so there is no band gap. When $\varepsilon$ greater than 10, the shapes of band structures for every permittivities are very similar. We can see that the width of band gap is getting wider as the permittivity increases, and the frequency of the band gap is getting lower and lower. This phenomenon appears in all lattices, thus we compute the band structures for various lattices with the fix permittivity $\varepsilon$ if we just want to realize the shape of the band structure. In table \ref{tab:CPU_time}, we can observe that the average computational time for solving a GEP on GPU is far less than that on MATLAB. These results suggest that the performance improvement for computing the GEP is significant by using the GPU. 

\begin{table}
\begin{center}
\begin{tabular}{|c|c|c||c|c|c|}
\hline
&\multicolumn{2}{c||}{Time (seconds)} & & \multicolumn{2}{c|}{Time (seconds)}\\ \hline
numbers of grid & MATLAB & GPU & numbers of grid & MATLAB & GPU\\ \hline
30  & 61 & 4 &78  & 1014 & 26\\
36  & 105 & 5 & 84  & 1229 & 28\\
42  & 151 & 6 & 90 & 1605 & 32\\
48  & 231 & 8 & 96 & 1689 & 39\\
54  & 370 & 10 & 102 & 2389 & 48\\
60  & 481 & 13 & 108 & 2462 & 52\\
66  & 601 & 15 & 114 & 3431 & 94\\
72  & 778 & 19 & 120 & 3126 & 70 \\
\hline
\end{tabular}
\caption{Average time for solving a GEP on MATLAB and GPU with various $n_{1}$.}\label{tab:CPU_time}
\end{center}
\end{table}

Now we show the band structures for materials in various lattices as follows.

\begin{itemize}
\item {\bf{SC: }} The path of wave vector: $\Gamma \rightarrow X \rightarrow M \rightarrow \Gamma \rightarrow R \big{|} M \rightarrow R$.

\begin{figure}[H]
\centering
\subfigure[ ]{
\includegraphics[width=3.5cm]{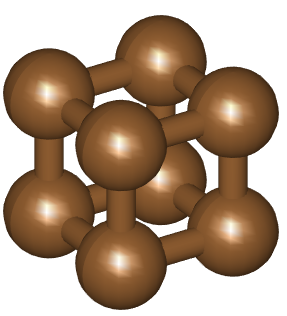}}\!\!
\subfigure[ ]{
\includegraphics[width=5cm]{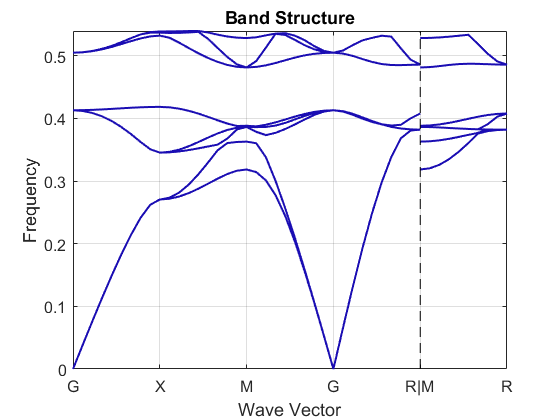}}
\caption{(a) The material satifies the space group No.221 Pm$\overline{3}$m. (b) The band structure for SC lattice.}
\end{figure}

\item {\bf{FCC: }} The path of wave vector: $X \rightarrow U \rightarrow L \rightarrow \Gamma \rightarrow X \rightarrow W \rightarrow K$.

\begin{figure}[H]
\centering
\subfigure[ ]{
\includegraphics[width=3.5cm]{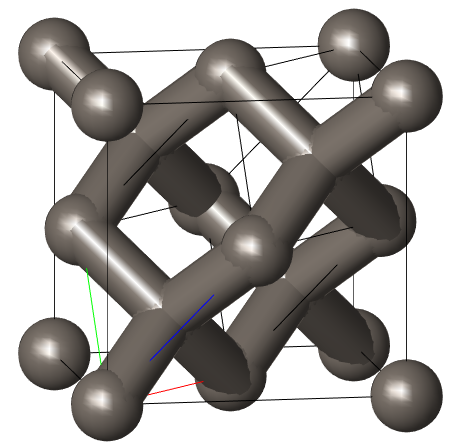}}\!\!
\subfigure[ ]{
\includegraphics[width=5cm]{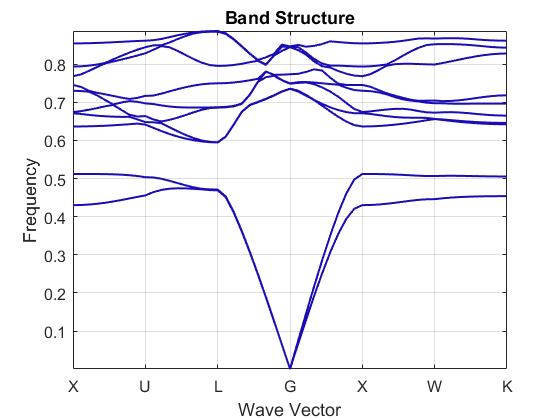}}
\caption{(a) The material satifies the space group No.227 Fd$\overline{3}$m. (b) The band structure for diamond structure in FCC lattice.}
\end{figure}

\item {\bf{BCC: }} We set the permittivity $\varepsilon = 16$ and the path of wave vector to $\Gamma \rightarrow M \rightarrow K \rightarrow \Gamma \rightarrow A \rightarrow L \rightarrow H \rightarrow A$.

\begin{figure}[H]
\centering
\subfigure[ ]{
\includegraphics[width=3.5cm]{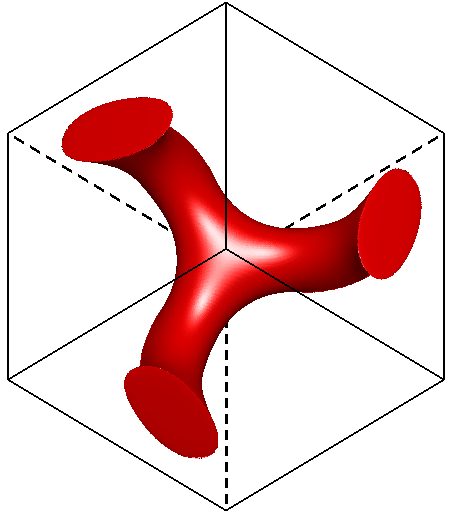}}\!\!
\subfigure[ ]{
\includegraphics[width=5cm]{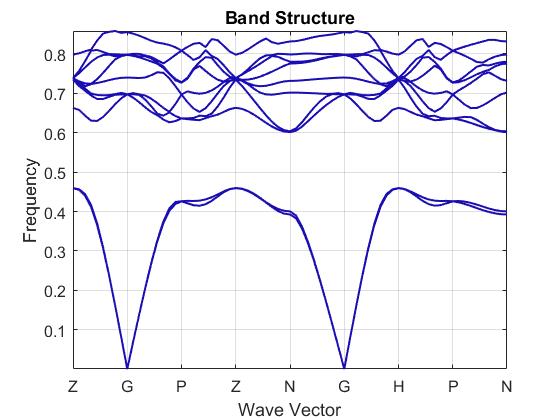}}
\caption{(a) The material satifies the space group No.214 I$4_{1}32$. (b) The band structure for single gyroid structure in BCC lattice.}
\end{figure}

\item {\bf{Hexagonal: }} The path of wave vector: $\Gamma \rightarrow M \rightarrow K \rightarrow \Gamma \rightarrow A \rightarrow L \rightarrow H \rightarrow A \big{|} L \rightarrow M \big{|} K \rightarrow H$.

\begin{figure}[H]
\centering
\subfigure[ ]{
\includegraphics[width=3.2cm]{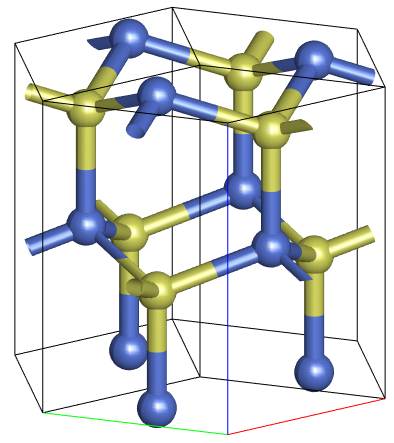}}\!\!
\subfigure[ ]{
\includegraphics[width=5cm]{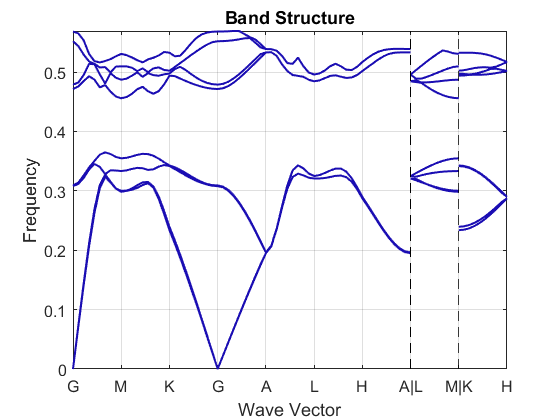}}
\caption{(a) The material satifies the space group No.186 P$6_3$mc. (b) The band structure for hexagonal lattice.}
\end{figure}

\item {\bf{Rhombohedral: }} The path of wave vector: $\Gamma \rightarrow L \rightarrow B_{1} \big{|} B \rightarrow Z \rightarrow \Gamma \rightarrow X \big{|} Q \rightarrow F \rightarrow P_{1} \rightarrow Z \big{|} L \rightarrow P$.

\begin{figure}[H]
\centering
\subfigure[ ]{
\includegraphics[width=3.2cm]{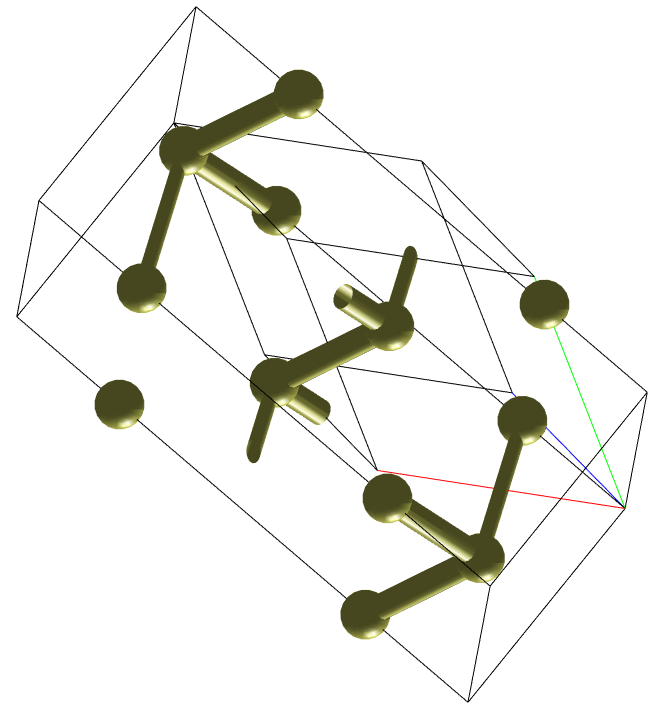}}\!\!
\subfigure[ ]{
\includegraphics[width=5cm]{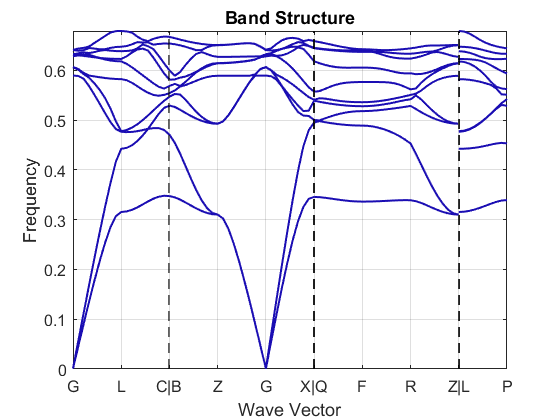}}
\caption{(a) The material satifies the space group No.166 R$\overline{3}$m. (b) The band structure for rhombohedral lattice.}
\end{figure}

\item {\bf{Primitive Tetragonal: }} The path of wave vector: $\Gamma \rightarrow X \rightarrow M \rightarrow \Gamma \rightarrow Z \rightarrow R \rightarrow A \rightarrow Z \big{|} X \rightarrow R \big{|} M \rightarrow A$.

\begin{figure}[H]
\centering
\subfigure[ ]{
\includegraphics[width=2cm]{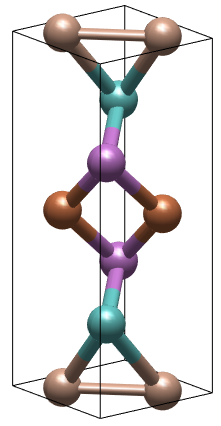}}\!\!
\subfigure[ ]{
\includegraphics[width=5cm]{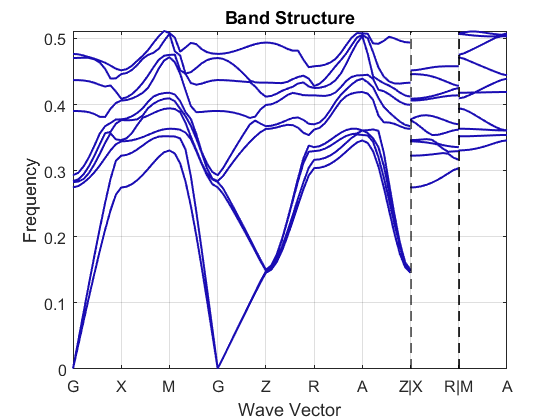}}
\caption{(a) The material satifies the space group No.129 P$\tfrac{4}{\mbox{n}}$mm. (b) The band structure for primitive tetragonal lattice.}
\end{figure}

\item {\bf{Body-Centered Tetragonal: }} The path of wave vector: $\Gamma \rightarrow X \rightarrow Y \rightarrow S \rightarrow \Gamma \rightarrow Z \rightarrow \Sigma_{1} \rightarrow N \rightarrow P \rightarrow Y_1 \rightarrow Z \big{|} X \rightarrow P$.

\begin{figure}[H]
\centering
\subfigure[ ]{
\includegraphics[width=4.5cm]{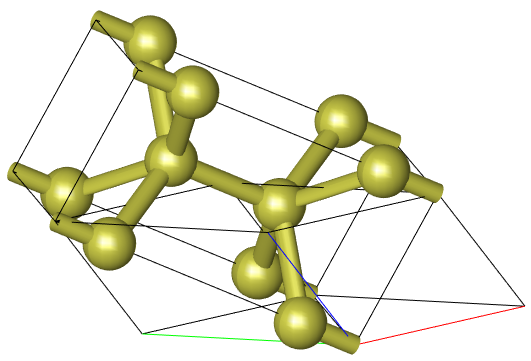}}\!\!
\subfigure[ ]{
\includegraphics[width=5cm]{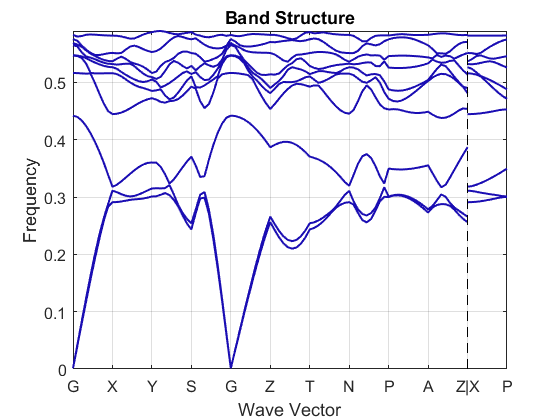}}
\caption{(a) The material satifies the space group No.139 I$\tfrac{4}{\mbox{m}}$mm. (b) The band structure for body-centered tetragonal lattice.}
\end{figure}

\item {\bf{Primitive Orthorhombic: }} The path of wave vector: $\Gamma \rightarrow X \rightarrow S \rightarrow Y \rightarrow \Gamma \rightarrow Z \rightarrow U \rightarrow R \rightarrow T \rightarrow Z \big{|} Y \rightarrow T \big{|} U \rightarrow X \big{|} S \rightarrow R$.

\begin{figure}[H]
\centering
\subfigure[ ]{
\includegraphics[width=4.2cm]{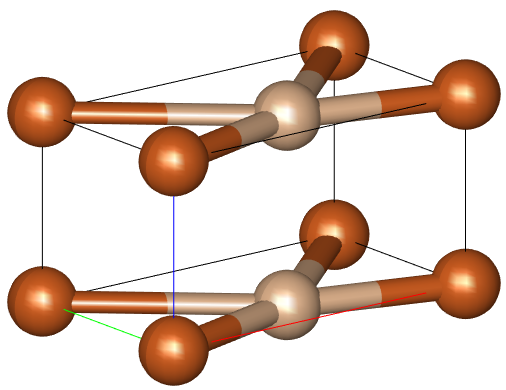}}\!\!
\subfigure[ ]{
\includegraphics[width=5cm]{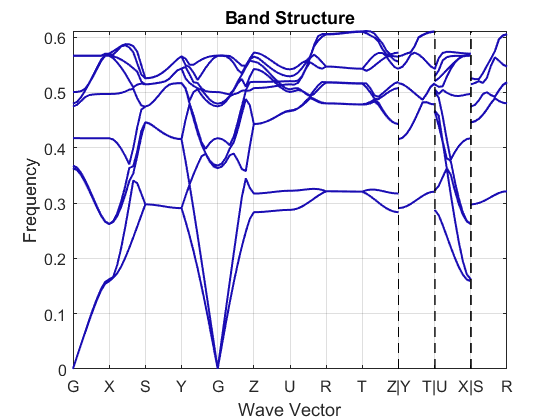}}
\caption{(a) The material satifies the space group No.25 Pmm2. (b) The band structure for primitive orthorhombic lattice.}
\end{figure}

\item {\bf{Base-Centered Orthorhombic: }}The path of wave vector: $\Gamma \rightarrow X \rightarrow S \rightarrow R \rightarrow A \rightarrow Z \rightarrow \Gamma \rightarrow Y \rightarrow X_{1} \rightarrow A_{1} \rightarrow T \rightarrow Y \big{|} Z \rightarrow T$.

\begin{figure}[H]
\centering
\subfigure[ ]{
\includegraphics[width=4.5cm]{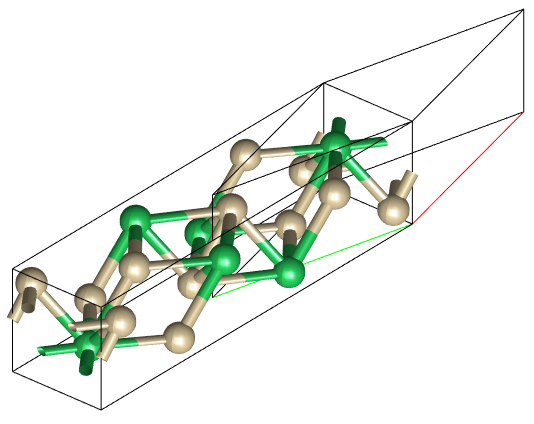}}\!\!
\subfigure[ ]{
\includegraphics[width=5cm]{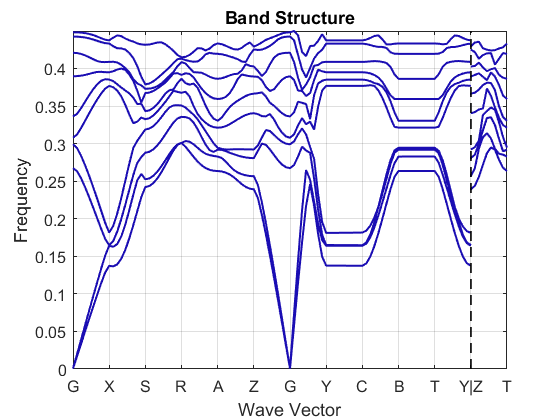}}
\caption{(a) The material satifies the space group No.63 Cmcm. (b) The band structure for base-centered orthorhombic lattice.}
\end{figure}

\item {\bf{Face-Centered Orthorhombic: }} The path of wave vector: $\Gamma \rightarrow Y \rightarrow T \rightarrow Z \rightarrow \Gamma \rightarrow X \rightarrow A_{1} \rightarrow Y \big{|} T \rightarrow X_{1} \big{|} X \rightarrow A \rightarrow Z \big{|} L \rightarrow \Gamma$.

\begin{figure}[H]
\centering
\subfigure[ ]{
\includegraphics[width=4cm]{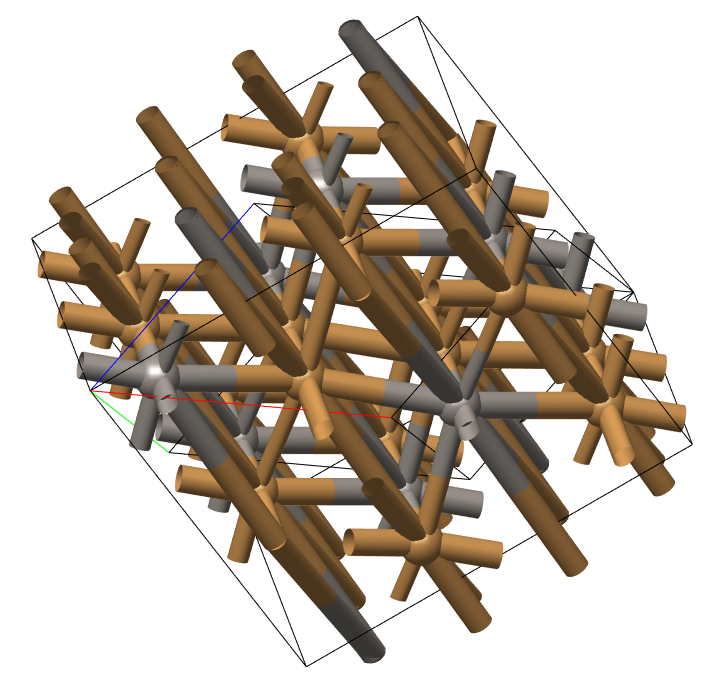}}\!\!
\subfigure[ ]{
\includegraphics[width=5cm]{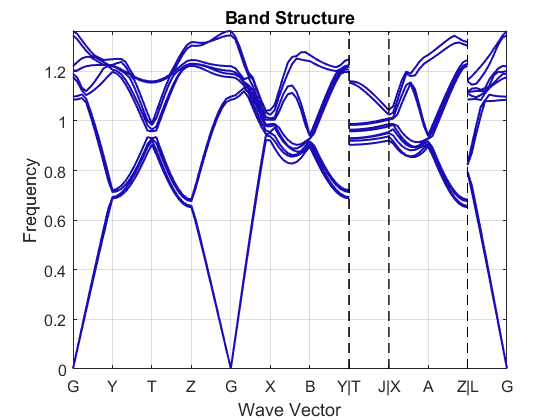}}
\caption{(a) The material satifies the space group No.70 Fddd. (b) The band structure for face-centered orthorhombic lattice.}
\end{figure}

\item {\bf{Body-Centered Orthorhombic: }} The path of wave vector: $\Gamma \rightarrow X \rightarrow L \rightarrow T \rightarrow W \rightarrow R \rightarrow X_{1} \rightarrow Z \rightarrow \Gamma \rightarrow Y \rightarrow S \rightarrow W \big{|} L_{1} \rightarrow Y \big{|} Y_{1} \rightarrow Z$.

\begin{figure}[H]
\centering
\subfigure[ ]{
\includegraphics[width=3cm]{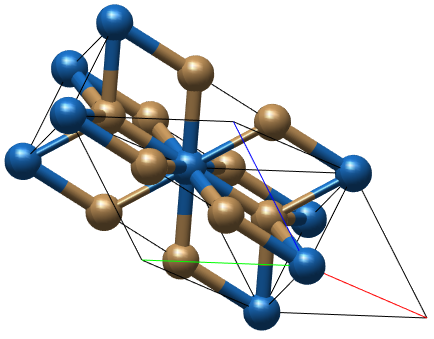}}\!\!
\subfigure[ ]{
\includegraphics[width=5cm]{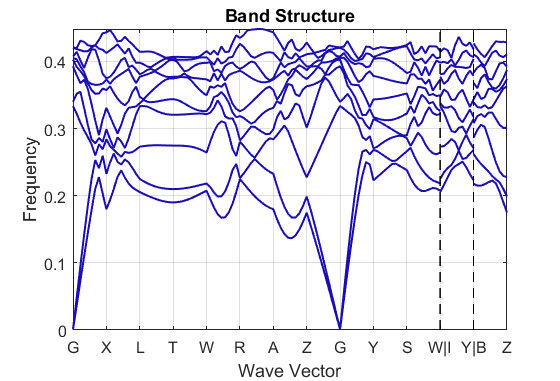}}
\caption{(a) The material satifies the space group No.71 Immm. (b) The band structure for body-centered orthorhombic lattice.}
\end{figure}

\item {\bf{Primitive Monoclinic: }} The path of wave vector: $\Gamma \rightarrow Y \rightarrow H \rightarrow C \rightarrow E \rightarrow M_{1} \rightarrow A \rightarrow X \rightarrow H_{1} \big{|} M \rightarrow D \rightarrow Z \big{|} Y \rightarrow D$.

\begin{figure}[H]
\centering
\subfigure[ ]{
\includegraphics[width=2.3cm]{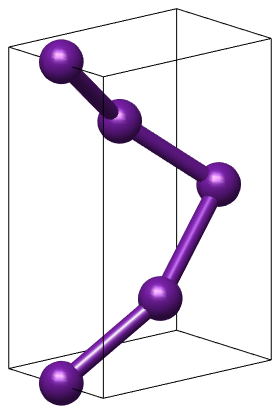}}\!\!
\subfigure[ ]{
\includegraphics[width=5cm]{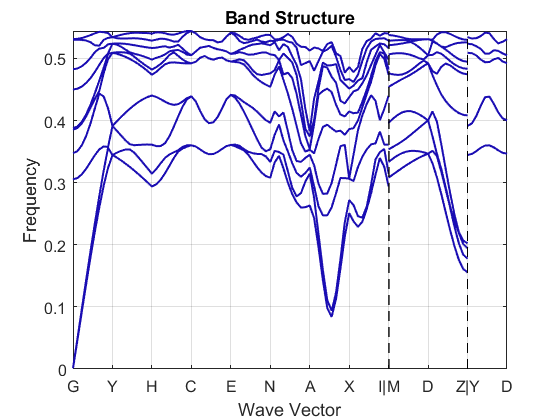}}
\caption{(a) The material satifies the space group No.4 P$2_1$. (b) The band structure for primitive monoclinic lattice.}
\end{figure}

\item {\bf{Base-Centered Monoclinic: }} The path of wave vector: $\Gamma \rightarrow Y \rightarrow F \rightarrow H \rightarrow Z \rightarrow I \rightarrow F_{1} \big{|} H_{1} \rightarrow Y_{1} \rightarrow X \rightarrow \Gamma \rightarrow N \big{|} M \rightarrow \Gamma$.

\begin{figure}[H]
\centering
\subfigure[ ]{
\includegraphics[width=4cm]{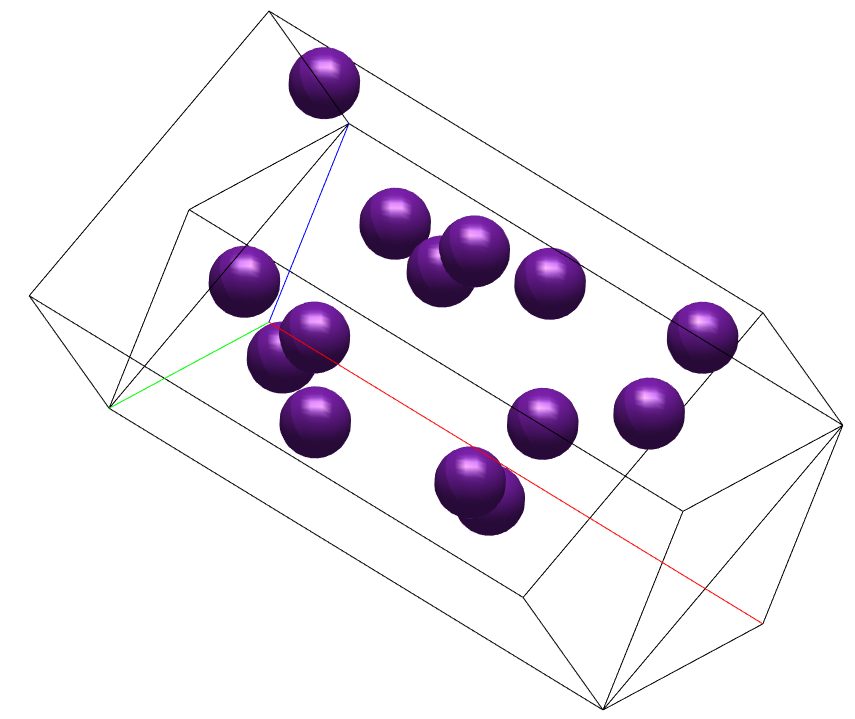}}\!\!
\subfigure[ ]{
\includegraphics[width=5cm]{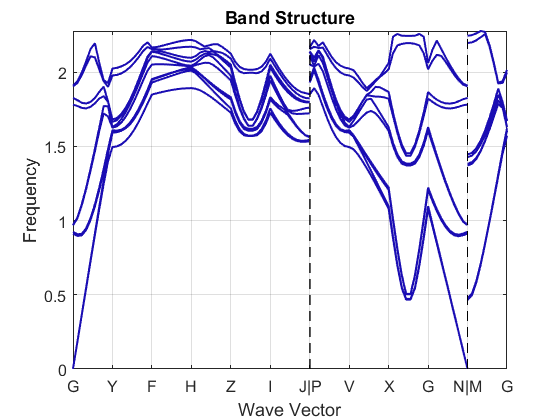}}
\caption{(a) The material satifies the space group No.5 C2. (b) The band structure for base-centered monoclinic lattice.}
\end{figure}

\item {\bf{Triclinic: }} The path of wave vector: $X \rightarrow \Gamma \rightarrow Y \big{|} L \rightarrow \Gamma \rightarrow Z \big{|} N \rightarrow \Gamma \rightarrow M \big{|} R \rightarrow \Gamma$.

\begin{figure}[H]
\centering
\subfigure[ ]{
\includegraphics[width=4.5cm]{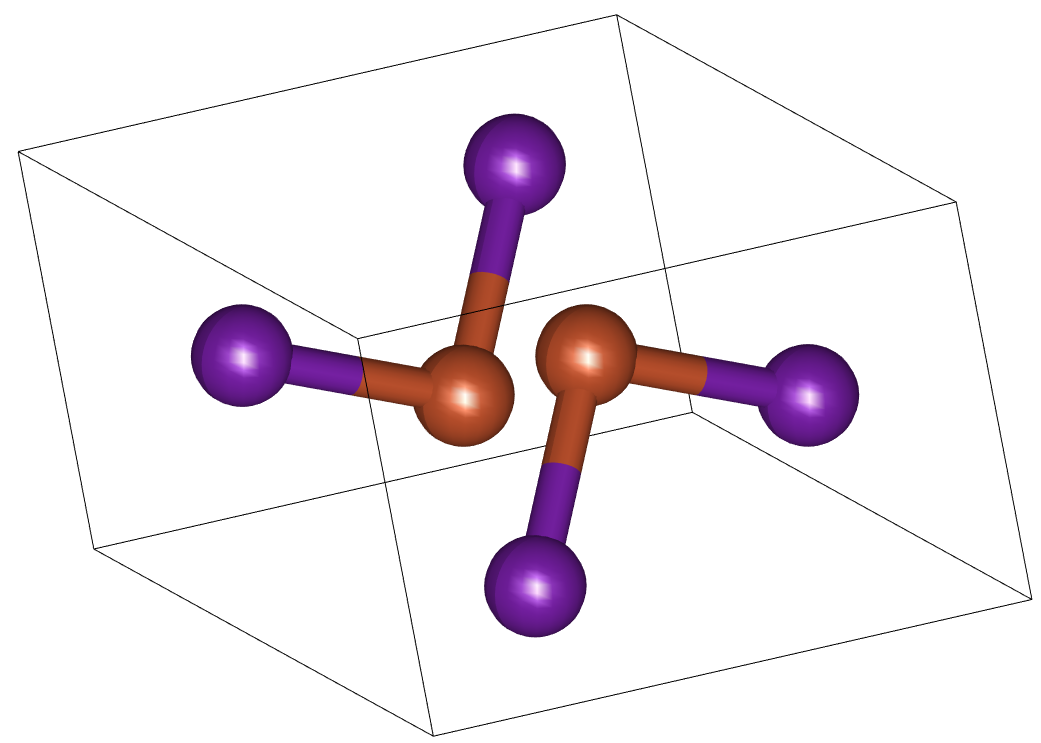}}\!\!
\subfigure[ ]{
\includegraphics[width=5cm]{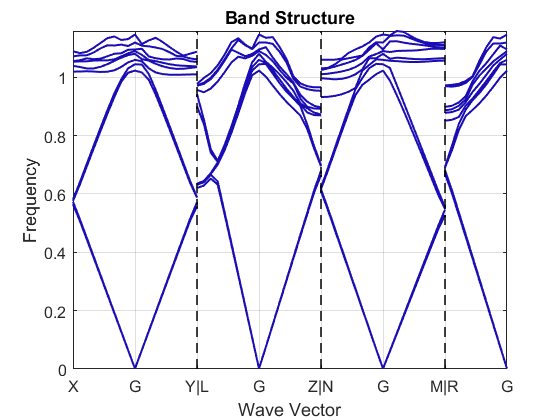}}
\caption{(a) The material satifies the space group No.2 P$\overline{1}$. (b) The band structure for triclinic lattice.}
\end{figure}

\end{itemize}

\section{Concluding Remarks}\label{section7}

In this section, we summarize the results of this dissertation and provide some future works.

The contributions in this disserlation can be conclude as follows:
\begin{enumerate}
\item We use Yee's scheme to discretize the equation $\nabla \times \nabla \times E = \lambda\varepsilon E$ for various lattice structures and find 4 general matrix representations that can be suitable for any situations.
\item We drive the eigen-decompositions of $K_{1}$, $K_{2}$, and $K_{3}$ for the 4 general matrix representations. There are only two kind of the decompositions which depend on the value $\cos\theta_{\alpha} - \cos\theta_{\beta}\cos\theta_{\gamma}$. Moreover, we get the eigen-decompositions of discrete partial derivative operators.
\item Using this decompositions, we successfully extend the powerful scheme in \cite{hhlw2013eig} to all the 3D Bravais lattices, the scheme include SVD for discrete curl operator, null space free method, and FFT based matrix-vector multiplications.
\item For each lattices, we design several photonic crystals that satisfy the space groups, and then compute their band structures as in section \ref{section6}. We use GPU to achieve the high performance computation, so the computational efficiency is far better than any other methods in the world.
\end{enumerate}

Next, we list some future works as follows:
\begin{enumerate}
\item In this dissertation, we only consider the perfect crystal, but there are still quasicrystals, non-crystalline solid, and crystals with defects we have not studied yet. In these materials, all of the techniques we developed  can also be applied.
\item We correct the angles $\theta'_{\alpha}$, $\theta'_{\beta}$, and $\theta'_{\gamma}$ between the lattice translation vectors, such kind of perturbation will cause a little error. One may develop another method to avoid the problem.
\item The computation domain is the primitive cell until now, but the supercell that involves different materials in each primitive cells is important now, we must extend our method to deal with the supercell soon.
\end{enumerate}

\bibliographystyle{abbrv}
\bibliography{ref}

\end{document}